\documentclass[11pt]{amsart}
\usepackage{amssymb}

\newtheorem{theo}{Theorem} 

\newtheorem{lemma}{Lemma}[section]
\newtheorem{prop}[lemma]{Proposition}
\newtheorem{corol}[lemma]{Corollary}
\newtheorem{claim}[lemma]{Claim}

\theoremstyle{remark}
\newtheorem{remark}[lemma]{Remark}

\theoremstyle{definition}

\newcommand{\dd}{\mathsf{d}}

\newcommand{\CC}{\mathbb{C}}
\newcommand{\NN}{\mathbb{N}}
\newcommand{\RR}{\mathbb{R}}
\newcommand{\ZZ}{\mathbb{Z}}

\newcommand{\eps}{\varepsilon}

\newcommand{\BB}{B}

\newcommand{\LLL}{\mathcal{L}}
\newcommand{\MMM}{\mathcal{M}}

\newcommand{\PPP}{\mathcal{P}}
\newcommand{\SSS}{\mathcal{S}}

\newcommand{\VVV}{\mathcal{V}}

\newcommand{\YYY}{\mathcal{Y}}

\newcommand{\teps}{\widetilde{\eps}}

\newcommand{\tW}{W_1}

\newcommand{\aexp}{a}

\newcommand{\tphi}{\widetilde{\varphi}}

\newcommand{\zbar}{\overline{z}}
\newcommand{\ubar}{\overline{u}}

\newcommand{\fbar}{\overline{f}}
\newcommand{\gbar}{\overline{g}}

\DeclareMathOperator{\re}{Re}
\DeclareMathOperator{\im}{Im}

\newcommand{\hdot}{\dot{H}^1}
\newcommand{\Hdot}{\dot{H}^1(\RR^N)}

\newcommand{\NH}[1]{\left\|#1\right\|_{\hdot}}

\DeclareMathOperator{\vect}{span}
\DeclareMathOperator{\Sp}{\sigma}

\newcommand{\ds}{\displaystyle}

\numberwithin{equation}{section} 

\title[Dynamic for energy critical NLS]{Dynamic of threshold solutions for energy-critical NLS}

\author[T.~Duyckaerts]{Thomas Duyckaerts$^1$}
\email{thomas.duyckaerts@u-cergy.fr}
\address{Thomas Duyckaerts\\
Universit{\'e} de Cergy-Pontoise\\
D\'epartement de Math\'ematiques\\ 
Site de Saint Martin, 2 avenue Adolphe-Chauvin\\ 
95302 Cergy-Pontoise cedex, France. }
\author[F.~Merle]{Frank Merle$^2$}
\thanks{$^1$Cergy-Pontoise (UMR 8088)}
\thanks{$^2$Cergy-Pontoise, IHES, CNRS}
\thanks{This work was partially supported by the French ANR Grant ONDNONLIN}

\date{\today}
\begin{document}

\begin{abstract}
We consider the energy-critical non-linear focusing Schr\"odinger equation in dimension $N=3,4,5$. An explicit stationnary solution, $W$, of this equation is known.
 In \cite{KeMe06}, the energy $E(W)$ has been shown to be a threshold for the dynamical behavior of solutions of the equation. In the present article, we study the dynamics at the critical level $E(u)=E(W)$ and classify the corresponding solutions. This gives in particular a dynamical characterization of $W$.
\end{abstract}

\maketitle
\tableofcontents

\section{Introduction}
We consider the focusing energy-critical Schr\"odinger equation on an interval $I$ ($0\in I$)
\begin{equation}
\label{CP}
\left\{ 
\begin{gathered}
i\partial_t u +\Delta u+|u|^{p_c-1}u=0,\quad (t,x)\in I\times \RR^N\\
u_{\restriction t=0}=u_0\in \hdot,
\end{gathered}\right.
\end{equation}
where 
$$N\in\{3,4,5\}, \quad p_c:=\frac{N+2}{N-2}$$
and $\hdot:=\Hdot$ is the homogeneous Sobolev space on $\RR^N$ with the norm 
$\|f\|_{\hdot}^2:=\int |\nabla f|^2.$
The Cauchy problem for \eqref{CP} was studied in \cite{CaWe90}. Namely, if $u_0$ is in $\hdot$, there exists an unique solution defined on a maximal interval $I=(-T_-,T_+)$, 
such that
\begin{equation*}
J\Subset I\Longrightarrow \|u\|_{S(J)}<\infty,\quad S(J):=L^{2p_c}\big(J\times \RR^N\big),
\end{equation*}
and the energy
$$ E(u(t))=\frac{1}{2} \int |\nabla u(t,x)|^2dx-\frac{1}{2^*}\int |u(t,x)|^{2^*}dx$$
is constant (here $2^*:=\frac{2N}{N-2}=p_c+1$ is the critical exponent for the $H^1$-Sobolev embedding in $\RR^N$). In addition, $u$ satisfies the following global existence criterium:
\begin{equation*}
T_+<\infty\Longrightarrow \|u\|_{S(0,T_+)}=\infty.
\end{equation*}
Moreover, solutions of equation \eqref{CP} are invariant by the following transformations: if $u(t,x)$ is such a solution so is 
$$\frac{e^{i\theta_0} }{\lambda_0^{(N-2)/2}} u\Big(\frac{t_0+t}{\lambda_0^2},\frac{x_0+x}{\lambda_0}\Big),\quad (\theta_0,\lambda_0,t_0,x_0)\in \RR\times (0,\infty)\times\RR\times\RR^N.$$ 
Note that these transformations preserve the $S(\RR)$-norm, as well as the $\hdot$-norm, the $L^{2^*}$-norm and thus the energy.\par
An explicit solution of \eqref{CP} is the stationnary solution in $\hdot$ (but in $L^2$ only if $N\geq 5$)
\begin{equation}
\label{defW}
W:=\frac{1}{\left(1+\frac{|x|^2}{N(N-2)}\right)^{\frac{N-2}{2}}}.
\end{equation}
The works of Aubin and Talenti \cite{Au76,Ta76}, give the following elliptic characterization of $W$
\begin{gather}
\label{SobolevIn}
\forall u\in\hdot,\quad \|u\|_{L^{2^*}}\leq C_N\|u\|_{\hdot}\\
\label{CarW}
\|u\|_{L^{2*}}=C_N\|u\|_{\hdot}\Longrightarrow \exists \;\lambda_0,x_0,z_0\quad u(x)=z_0 W\Big(\frac{x+x_0}{\lambda_0}\Big),
\end{gather}
where $C_N$ is the best Sobolev constant in dimension $N$.\par
In \cite{KeMe06}, Kenig and Merle has shown that $W$ plays an important role in the dynamical behavior of solutions of Equation  \eqref{CP}. Indeed, $E(W)=\frac{1}{NC_N^N}$ 
is an energy threshold for the dynamics in the following sense. Let $u$ be a radial solution of \eqref{CP} such that
\begin{equation}
\label{hypsubcrit}
E(u_0)<E(W).
\end{equation}
Then if $\|u_0\|_{\hdot}<\|W\|_{\hdot}$, we have
\begin{equation}
\label{scattering}
T_+=T_-=\infty \text{ and } \|u\|_{S(\RR)}<\infty.
\end{equation}
On the other hand if $\|u_0\|_{\hdot}>\|W\|_{\hdot}$, and $u_0\in L^2$ then
\begin{equation}
\label{explosion}
T_+<\infty \text{ and } T_-<\infty.
\end{equation}
Our goal is to give a classification of solutions of \eqref{CP} with \textit{critical} energy, that is with initial condition such that
\begin{equation*}
u_0\in  \hdot,\quad E(u_0)=E(W).
\end{equation*}
A new example of such a solution (not satisfying \eqref{scattering} nor \eqref{explosion}) is given by $W$. We start with the following theorem, which shows that the dynamics at this critical level is richer, in the sense that there exists orbit connecting different types of behavior for $t>0$ and $t<0$.
\begin{theo}
\label{th.exist}
Let $N\in\{3,4,5\}$. There exist radial solutions $W^-$ and $W^+$ of \eqref{CP} such that
\begin{gather}
\label{ex.energy}
E(W)=E(W^+)=E(W^-),\\
\label{ex.lim}
T_+(W^-)=T_+(W^+)=+\infty \text{ and }\lim_{t\rightarrow +\infty} W^{\pm}(t)=W \text{ in } \hdot,\\
\label{ex.sub}
\big\|W^{-}\big\|_{\hdot}<\|W\|_{\hdot},\quad  T_-(W^-)=+\infty,\quad \|W^-\|_{S((-\infty,0])}<\infty,\\
\label{ex.super}
\big\|W^{+}\big\|_{\hdot}>\|W\|_{\hdot},\text{ and, if }N=5,\;T_-(W^+)<+\infty.
\end{gather}
\end{theo}
\begin{remark}
As for $W$, $W^+(t)$ and $W^-(t)$ belongs to $L^2$ if and only if $N=5$. We still expect $T_-(W^+)<+\infty$ for $N=3,4$.
\end{remark}
Our classification result is as follows.
\begin{theo}
\label{th.classif}
Let $N\in\{3,4,5\}$. Let $u_0\in \hdot$ radial, such that 
\begin{equation}
\label{threshold}
E(u_0)=E(W)=\frac{1}{NC_N^N}.
\end{equation}
Let $u$ be the solution of \eqref{CP} with initial condition $u_0$ 
and $I$ its maximal interval of definition. Then the following holds:
\begin{enumerate}
\item \label{theo.sub} If $\ds \int |\nabla u_0|^2<\int |\nabla W|^2=\frac{1}{C_N^N}$ then $I=\RR$. Furthermore, either $u=W^-$ up to the symmetry of the equation, or $\|u\|_{S(\RR)}<\infty$.
\item \label{theo.crit} If $\ds \int |\nabla u_0|^2=\int |\nabla W|^2$ then $u=W$ up to the symmetry of the equation.
\item \label{theo.super} If $\ds \int |\nabla u_0|^2>\int |\nabla W|^2$, and $u_0\in L^2$ then either $u=W^+$ up to the symmetry of the equation, or $I$ is finite.
\end{enumerate}
\end{theo}
The constant $C_N$ is defined in \eqref{SobolevIn}. In the theorem, by \emph{$u$ equals $v$ up to the $(\hdot-)$sym\-me\-try of the equation}, we mean that there exist $t_0\in\RR$, $\theta_0\in\RR$, $\lambda_0>0$ such that
$$ u(t,x)=\frac{e^{i\theta_0}}{\lambda_0^{(N-2)/2}}v\Big(\frac{t_0+t}{\lambda_0^2},\frac{x}{\lambda_0}\Big) \text{ or }u(t,x)=\frac{e^{i\theta_0}}{\lambda_0^{(N-2)/2}}\overline{v}\Big(\frac{t_0-t}{\lambda_0^2},\frac{x}{\lambda_0}\Big).$$
\begin{remark}
\label{RemPersist}
Case \eqref{theo.crit} is a direct consequence of the variational characterization of $W$ given by Aubin and Talenti \cite{Au76}, \cite{Ta76}. Furthermore, using assumption \eqref{threshold}, it shows (by continuity of $u$ in $\hdot$) that the assumptions $ \ds \int |\nabla u(t_0)|^2<\int |\nabla W|^2, \; \ds \int |\nabla u(t_0)|^2>\int |\nabla W|^2$
do not depend on the choice of the initial time $t_0$. Of course, this dichotomy does not persist when $E(u_0)>E(W)$.
\end{remark}
\begin{remark}
In the supercritical case \eqref{theo.super}, our theorem shows that in dimension $N=3$ or $N=4$, an $L^2$-solution blows up for negative and positive times. We conjecture that case \eqref{theo.super} holds without the assumption ``$u_0\in L^2$'', i.e. that the only solution with critical energy such that $\int |\nabla u_0|^2>\int |\nabla W|^2$ and whose interval of definition is not finite is $W^+$ up to the symmetry of the equation.
\end{remark}
\begin{remark}
We expect that the extension of the results of \cite{KeMe06} to the non-radial case, together with the material in this paper would generalize Theorem \ref{th.classif} to the non-radial case.
\end{remark}
From \cite{Bo99BO,Bo99JA}, we know that a solution such that $\|u\|_{S(\RR)}<\infty$ scatters in $\hdot$ at $\pm\infty$. Cases \eqref{theo.sub} and \eqref{theo.crit} of Theorem \ref{th.classif} shows:
\begin{corol} 
Up to the symmetry of the equation, $W$ is the only radial solution such that $E(u_0)=E(W)$ and $\int |\nabla u_0|^2\leq \int |\nabla W|^2$ which does not scatter in $\hdot$ for neither positive nor negative times. 
\end{corol}
The behavior exhibited here for $\hdot$-critical NLS is the analogue of the one of the $L^2$-critical NLS. For this equation, Merle has shown in \cite{Me93} that a $H^1$-solution $u(t)$ at the critical level in $L^2$ and such that $xu\in L^2$ is either a periodic solution of the form $e^{i\omega t}Q$, an explicit blow-up solution converging to $Q$ after rescaling or a solution scattering at $\pm \infty$.\par
The outline of the paper is as follows. In Section \ref{sec.compact}, we use arguments of \cite{KeMe06} to show the compactness, up to modulation, of a subcritical threshold solution of \eqref{CP} such that $\|u\|_{S(0,+\infty)}=\infty$ (case \eqref{theo.sub} of Theorem \ref{th.classif}). 
Section \ref{sec.convergence} is devoted to the proof of the fact that such a solution converges to $W$ as $t\rightarrow +\infty$.
In Section \ref{sec.super}, we show a similar result for $L^2$ super-critical solutions of \eqref{CP} (case \eqref{theo.super}). The last ingredient of the proof, which is the object of Section \ref{sec.linear}, is an analysis of the linearized equation associated to \eqref{CP} near $W$. Both theorems are proven in Section \ref{sec.proof}.\par

\section{Compactness properties for nonlinear subcritical threshold solutions}
\label{sec.compact}

In this section we prove a preliminary result related to compactness properties of threshold solutions of \eqref{CP}, and which is the starting point of the proofs of Theorems \ref{th.exist} and \ref{th.classif}. It is essentially proven in \cite{KeMe06}, Proposition 4.2. We give the proof for the sake of completeness.\par
If $v$ is a function defined on $\RR^N$, we will write:
\begin{equation}
\label{def.modulations}
v_{[\lambda_0]}(x)=\frac{1}{\lambda_0^{(N-2)/2}} v\left(\frac{x}{\lambda_0}\right), \quad v_{[\theta_0,\lambda_0]}=e^{i\theta_0} \frac{1}{\lambda_0^{(N-2)/2}}v\left(\frac{x}{\lambda_0}\right).
\end{equation}
\begin{prop}[Global existence and compactness]
\label{compactness.u}
Let $u$ be a radial solution of \eqref{CP} and $I=(T_-,T_+)$ its maximal interval of existence. Assume
\begin{equation}
\label{AssuComp}
E(u_0)=E(W), \quad \NH{u_0}<\NH{W}.
\end{equation}
Then 
$$I=\RR.$$
 Furthermore, if $\ds \|u\|_{S(0,+\infty)}=\infty$,
there exists a map $\lambda$ defined on $[0,\infty)$ such that the set
\begin{equation} 
\label{defK+}
K_+:=\big\{u_{[\lambda(t)]}(t),\; t\in [0,+\infty)\big\}
\end{equation}
is relatively compact in $\hdot$. An analogous assertion holds on $(-\infty,0]$.
\end{prop}
As a corollary we derive the existence of threshold mixed behavior solutions for \eqref{CP} in the subcritical case in $\hdot$-norm.
\begin{corol}
\label{corol.exist}
There exists a solution $w^-$ of \eqref{CP} defined for $t\in \RR$, and such that
\begin{gather*}
E(w^-)=E(W),\quad \|w^-(0)\|_{\hdot}<\|W\|_{\hdot}\\
\|w^-\|_{S(0,+\infty)}=\infty,\quad \|w^-\|_{S(-\infty,0)}<\infty.
\end{gather*}
\end{corol}
The crucial point of the proofs of Proposition \ref{compactness.u} and Corollary \ref{corol.exist} is a compactness lemma for threshold solutions of \eqref{CP} which is the object of Subsection \ref{sub.compact.lemma}. We give a sketch of the proof, which is essentially contained in \cite{KeMe06}, and refer to \cite{KeMe06} for the details. In Subsections \ref{sub.compactness.u} and \ref{sub.exist}, we prove respectively Proposition \ref{compactness.u} and Corollary \ref{corol.exist}. We start with a quick review of the Cauchy Problem for \eqref{CP}.

\subsection{Preliminaries on the Cauchy  Problem}
In this subsection we quickly review existence, uniqueness and related results for the Cauchy problem \eqref{CP}. See \cite[Section 2]{KeMe06} for the details.
In the sequel, $I\ni 0$ is an interval.
We first recall the two following relevant function spaces for equation \eqref{CP}:
\begin{equation}
\label{defSZ}
S(I):=L^{\frac{2(N+2)}{N-2}}\Big(I\times \RR^N\Big), \quad Z(I):=L^{\frac{2(N+2)}{N-2}}\Big(I;L^{\frac{2N(N+2)}{N^2+4}}\Big).
\end{equation}
Note that $Z(I)$ is a Strichartz space for the Schr\"odinger equation, so that
\begin{equation}
\label{StrichartzI}
\|\nabla e^{it\Delta}u_0\|_{Z(\RR)}\leq C\|u_0\|_{\hdot}
\end{equation}
and that by Sobolev inequality,
\begin{equation}
\label{SobI}
  \|f\|_{S(I)}\leq C\|\nabla f\|_{Z(I)}.
\end{equation}
Following \cite{CaWe90},
we say that $u\in C^0(I,\hdot(\RR^N))$ is a solution of \eqref{CP} if for any $J\Subset I$, $u\in S(J)$, $|\nabla u|\in Z(J)$ and
$$ \forall t\in I,\quad u(t)=e^{it\Delta}u_0+i\int_0^t e^{i(t-s)\Delta}|u(s)|^{p_c-1}u(s)ds.$$
The following holds for such solutions.
\begin{lemma}
\label{lem.CP}
\begin{enumerate}
\item[]
\item\emph{Uniqueness.} Let $u$ and $\tilde{u}$ be two solutions of \eqref{CP} on an interval $I\ni 0$ with the same initial condition $u_0$. Then $u=\tilde{u}$.
\item\emph{Existence.} For $u_0\in \hdot$, there exists an unique solution $u$ of \eqref{CP} defined on a maximal interval of definition $(-T_-(u_0),T_+(u_0))$.
\item\emph{Finite blow-up criterion.} Assume that $T_+=T_+(u_0)<\infty$. Then
$\ds \|u\|_{S(0,T_+)}=+\infty.$
An analogous result holds for $T_-(u_0)$.
\item\emph{Scattering.} If $T_+(u_0)=\infty$ and $\|u\|_{S([0,+\infty))}<\infty$, there exists $u_+\in \hdot$ such that
$$ \lim_{t\rightarrow+\infty} \|u(t)-e^{it\Delta} u_+\|_{\hdot}=0.$$
\item\label{CPcontinuity}\emph{Continuity.}
Let $\tilde{u}$ be a solution of \eqref{CP} on $I\ni 0$. Assume that for some constant $A>0$, $$\sup_{t\in I} \|\tilde{u}(t)\|_{\hdot}+\|\tilde{u}\|_{S(I)}\leq A.$$
Then there exist $\eps_0=\eps_0(A)>0$ and $C_0=C_0(A)$ such that for any $u_0\in \hdot$ with 
$ \|\tilde{u}_0-u_0\|_{\hdot}=\eps<\eps_0,$
the solution $u$ of \eqref{CP} with initial condition $u_0$ is defined on $I$ and satisfies
$\|u\|_{S(I)}\leq C_0$ and $\sup_{t\in I}\|u(t)-\tilde{u}(t)\|_{\hdot}\leq C_0 \eps.$
\end{enumerate}
\end{lemma}
(See \cite{CaWe90}, \cite{Bo99JA}, \cite{TaVi05}, \cite{KeMe06}.)
\begin{remark}
\label{remark.exist}
Precisely, the existence result states that there is an $\eps_0>0$ such that if
\begin{equation}
\label{critere.exist}
\left\|e^{it\Delta} u_0\right\|_{S(I)}=\eps<\eps_0,
\end{equation}
then \eqref{CP} has a solution $u$ on $I$ such that $\|u\|_{S(I)}\leq 2\eps$. In particular, by \eqref{StrichartzI} and \eqref{SobI}, for small initial condition in $\hdot$, $u$ is globally defined and scatters.
\end{remark}

\subsection{Compactness or scattering for sequences of threshold $\hdot$-subcritical solutions}
\label{sub.compact.lemma}
The following lemma (closely related to Lemma 4.9 of \cite{KeMe06}) is a consequence, through the profile decomposition of Keraani \cite{Ke01} (which characterizes the defect of compactness of Strichartz estimates for solutions of linear Schr\"odinger equation), of the scattering of radial subcritical solutions of \eqref{CP} shown in \cite{KeMe06}.
\begin{lemma}
\label{lemmacc}
Let $(u_n^0)_{n\in \NN}$ be a sequence of radial functions in $\hdot$ such that
\begin{equation}
\label{hyplemcc} \forall n,\quad E(u_{n}^0)\leq E(W),\quad \NH{u_{n}^0}\leq \NH{W}.
\end{equation}
Let $u_n$ be the solution of \eqref{CP} with initial condition $u_n^0$.
Then, up to the extraction of a subsequence of $(u_n)_n$, one at least of the following holds:
\begin{enumerate}
\item \textbf{Compactness.} \label{compactness} There exists a sequence $(\lambda_n)_n$ such that the sequence $\big((u_n^0)_{[\lambda_n]}\big)_{n}$ converges in $\hdot$;
\item \textbf{Vanishing for $t\geq 0$.} \label{rightSC} For every $n$, $u_n$ is defined on $[0,+\infty)$ and $\ds \lim_{n\rightarrow +\infty} \|u_n\|_{S(0,+\infty)}=0;$
\item \textbf{Vanishing for $t\leq 0$.}\label{leftSC} For every $n$, $u_n$ is defined on $(-\infty,0]$ and $\ds \lim_{n\rightarrow +\infty} \|u_n\|_{S(-\infty,0)}=0;$
\item \textbf{Uniform scattering.} \label{allSC} For every $n$, $u_n$ is defined on $\RR$. Furthermore, there exists a constant $C$ independent of $n$ such that $$\|u_n\|_{S(\RR)}\leq C.$$
\end{enumerate}
\end{lemma}
\begin{proof}[Sketch of the proof of Lemma \ref{lemmacc}]
We will need the following elementary claim (see \cite[Lemma 3.4]{KeMe06}).
\begin{claim}
\label{lem.var}
Let $f\in \hdot$ such that $\NH{f}\leq \NH{W}$. Then
$$ \frac{\NH{f}^2}{\NH{W}^2}\leq \frac{E(f)}{E(W)}.$$
In particular, $E(f)$ is positive.
\end{claim}
\begin{remark}
\label{rem.compa}
Clearly, $\|u(t)\|^2_{\hdot}\geq 2E(u(t))$, so that Claim \ref{lem.var} implies that for solutions of \eqref{CP} satisfying \eqref{AssuComp},
$$ \exists C>0,\;\forall t,\quad C^{-1}\|u(t)\|_{\hdot}^2\leq E(u(t)) \leq C\|u(t)\|_{\hdot}^2.$$
\end{remark}
\begin{proof}
Let $\Phi(y)=\frac{1}{2} y - \frac{C_N^{2^*}}{2^*}y^{2^*/2}$.  Then by Sobolev embedding
$$\Phi\big(\|f\|^2_{\hdot}\big) \leq \frac{1}{2}\NH{f}^2-\frac{1}{2^*} \|f\|^{2^*}_{L^{2^*}}=E(f).$$
Note that $\Phi$ is concave on $\RR_+$, $\Phi(0)=0$ and $\Phi\big(\NH{W}^2\big)=E(W)$. Thus
\begin{align*}
\forall s\in(0,1),\quad \Phi\big(s\NH{W}^2\big) \geq s \Phi(\NH{W}^2)=sE(W).
\end{align*}
Taking $s=\frac{\NH{f}^2}{\NH{W}^2}$ yields the lemma.
\end{proof}
By the lemma of concentration compactness of Keraani (see \cite{Ke01}), there exists a sequence $(V_{j})_{j\in \NN}$ of solutions of Schr\"odinger \emph{linear} equation with initial condition in $\hdot$, and sequences $(\lambda_{jn},t_{jn})_{n\in \NN^*}$, $\lambda_{jn}>0$, $t_{jn}\in\RR$, which are pairwise orthogonal in the sense that
$$ j\neq k \Rightarrow \lim_{n\rightarrow +\infty} \frac{\lambda_{kn}}{\lambda_{jn}}+\frac{\lambda_{jn}}{\lambda_{kn}}+\frac{|t_{jn}-t_{kn}|}{\lambda^2_{jn}}=+\infty$$ 
such that for all $J$
\begin{gather}
\label{decomposition}
u_n^0=\sum_{j=1}^J V_j(s_{jn})_{[\lambda_{jn}]}+w_n^J, \text{ with } s_{jn}=\frac{-t_{jn}}{\lambda_{jn}^2},\\
\label{w.small}
\lim_{J\rightarrow +\infty}\limsup_{n\rightarrow +\infty} \| e^{it\Delta}w_n^J\|_{S(\RR)}=0,\\
\label{decompo.E0}
\NH{u_n^0}^2=\sum_{j=1}^J \NH{V_{j}}^2+\NH{w_n^J}^2+o(1) \text{ as } n\rightarrow +\infty,\\
\label{decompo.E}
E(u_n^0)=\sum_{j=1}^J E(V_{j}(s_{jn}))+E(w_n^J)+o(1) \text{ as } n\rightarrow +\infty.
\end{gather}
If all the $V_j$'s are identically $0$, then by \eqref{w.small}, $\|e^{it\Delta} u_n^0\|_{S(\RR)}$ tends to $0$ as $n$ tends to infinity,
and the sequence $(u_n)_n$ satisfies simultaneously \eqref{rightSC}, \eqref{leftSC} and \eqref{allSC}. Thus we may assume without loss of generality that $V_1\neq 0$. 
Furthermore, by assumption \eqref{hyplemcc} and by \eqref{decompo.E0}, $\|V_j\|_{\hdot}\leq \|W\|_{\hdot}$ and for large $n$, $\|w_n^J\|_{\hdot}< \|W\|_{\hdot}$, which implies by Claim \ref{lem.var} that the energies $E(V_j(s_{jn}))$ and $E(w_n^J)$ are nonnegative, and thus that $E(V_j(s_{jn}))\leq E(W)$. Extracting once again a subsequence if necessary, we distinguish two cases.\par
\subsubsection*{First case:}
$$
\lim_{n\rightarrow+\infty} E(V_1(s_{1n}))=E(W).
$$
By assumption \eqref{hyplemcc} and by \eqref{decompo.E} (all the energies being nonnegative for large $n$), $E(V_j(s_{jn}))$ ($j\geq 2$), and $E(w_n^J)$ tend to $0$ as $n$ tends to infinity. Thus by
Claim \ref{lem.var} , for $j\geq 2$, $V_j=0$, and $w_n^J=w_n^1$ tends to $0$ in $\hdot$. As a consequence
$$ u_n^0=V_1(s_{1n})_{[\lambda_{1n}]}+o(1) \text{ in } \hdot, \;n\rightarrow +\infty.$$
Up to the extraction of a subsequence, $s_{1n}$ converges to some $s\in [-\infty,+\infty]$.
If $s\in \RR$, It is easy to see that we are in case \eqref{compactness} (compactness up to modulation) of Lemma \ref{lemmacc}.
If $s=+\infty$, then $ \lim_{n\rightarrow+\infty} \|e^{it\Delta}u_n^0\|_{S(0,+\infty)}=0$,
so that by existence theory for \eqref{CP} (see Remark \ref{remark.exist}) case \eqref{rightSC} holds. Similarly, if $s=-\infty$ case \eqref{leftSC} holds.
\subsubsection*{Second case:}
\begin{equation}
\label{case2}
\exists \eps_1,\; 0<\eps_1<E(W) \text{ and } \forall n, \;E\left(V_{1}(s_{1n})\right)\leq E(W)-\eps_1.
\end{equation}
Here we are exactly in the situation of the first case of \cite[Lemma 4.9]{KeMe06}. We refer to the proof of this lemma for the details. Recall that for large $n$, all the energies are nonnegative in \eqref{decompo.E}.
Thus, in view of \eqref{decompo.E} (and of Claim \ref{lem.var} for the second inequality)
\begin{equation}
\label{EVj.petit}
\limsup_{n\rightarrow+\infty} E(V_j(s_{jn}))\leq \eps_1<E(W), \quad \NH{V_j(0)}<\NH{W}.
\end{equation}
Furthermore by assumption \eqref{hyplemcc} and by \eqref{decompo.E0}
\begin{equation}
\label{sum.square}
\sum_{j\geq 1} \|V_j(s_{jn})\|_{\hdot}^2\leq \|W\|_{\hdot}^2.
\end{equation}
Thus, according to the results of \cite{KeMe06} and the Cauchy problem theory for \eqref{CP}, $u_n^0$ is, up to the small term $w_n^J$, a sum \eqref{decomposition} of terms $U_{jn}^0=V_j(s_{jn})_{[\lambda_{jn}]}$ that are all initial conditions of a solution $U_{jn}$ of \eqref{CP} satisfying an uniform bound $\|U_{jn}\|_{S(\RR)}\leq c_j$ (with $\sum c_j^2$ finite by \eqref{sum.square}). Using the pairwise orthogonality of the sequences $(\lambda_{jn},t_{jn})_{n\in\NN^*}$, together with a long-time perturbation result for \eqref{CP}, it is possible to show that for some constant $C$ independant of $n$,
$$\|u_n\|_{S(\RR)}\leq C,$$
that is that case \eqref{allSC} of the Lemma holds. Up to the technical proof of this fact, which we omit, the proof of Lemma \ref{lemmacc} is complete.
\end{proof}

\subsection{Compactness up to modulation and global existence of threshold solutions}
\label{sub.compactness.u}
We now prove Proposition \ref{compactness.u}. 

\medskip

\noindent\emph{Step 1: compactness.}
We start by showing the compactness up to modulation of the threshold solution $u$. In Step 2 we will show  that $u$ is defined on $\RR$.
\begin{lemma}
\label{lem.compact}
Let $u$ be a solution of \eqref{CP} of maximal interval of definition $[0,T_+)$ such that $E(u_0)=E(W)$, $\|u_0\|_{\hdot}<\|W\|_{\hdot}$ and
\begin{equation*}
\|u\|_{S(0,T_+)}=+\infty.
\end{equation*}
Then there exists a function $\lambda$ on $[0,T_+)$ such that the set 
\begin{equation}
\label{defK+lem}
K_+:=\big\{u_{[\lambda(t)]}(t),\; t\in [0,T_+)\big\}
\end{equation}
is relatively compact in $\hdot$.
\end{lemma}
\begin{proof}
The proof is similar to the one in \cite{KeMe06}.
The main point of the proof is to show that for every sequence $(t_n)_n$, $t_n\in[0,T_+)$, there exists, up to the extraction of a subsequence, a sequence $(\lambda_n)_n$ such that $(u_{[\lambda_n]}(t_n))$ converges in $\hdot$. By continuity of $u$, we just have to consider the case $\lim_{n}t_n=T_+$.\par
Let us use Lemma \ref{lemmacc} for the sequence $u_n^0=u(t_n)$. We must show that we are in case \eqref{compactness}. Clearly, cases \eqref{rightSC} (vanishing for $t\geq 0$) and \eqref{allSC} (uniform scattering) are excluded by the assumption that $\|u\|_{S(0,T_+)}$ is infinite. 
Furthermore, $\|u\|_{S(0,t_n)}=\|u_n\|_{S(-t_n,0)}$ (where $u_n$ is the solution of \eqref{CP} with initial condition $u_n^0$) so that case \eqref{leftSC} would imply that  $\|u\|_{S(0,t_n)}$ tends to $0$, i.e that $u$ is identically $0$ which contradicts our assumptions. Thus case \eqref{compactness} holds: there exists, up to the extraction of a subsequence, a sequence $(\lambda_n)_n$ such that $\big(u_{[\lambda_n]}(t_n)\big)_{n}$ converges.\par
The existence of $\lambda(t)$ such that the set $K_+$ defined by \eqref{defK+lem} is relatively compact is now classical. Indeed
\begin{equation}
\label{borninf}
\forall t\in[0,T_+),\quad 2E(W)=2E(u(t))\leq\NH{u(t)}^2\leq \NH{W}^2.
\end{equation}
Fixing $t\in[0,T_+)$, define
\begin{equation}
\label{def.lambda}
\lambda(t):=\sup\left\{\lambda>0,\text{ s.t. } \int_{|x|\leq 1/\lambda} |\nabla u|^2(t,x) dx=E(W)\right\}.
\end{equation}
By \eqref{borninf}, $0<\lambda(t)<\infty$. Let $(t_n)_n$ be a sequence in $[0,T_+)$. As proven before, up to the extraction of a subsequence, there exists a sequence $(\lambda_n)_n$ such that $\big(u_{[\lambda_n]}(t_n)\big)_n$ converges in $\hdot$ to a function $v_0$ of $\hdot$.
One may check directly, using \eqref{borninf}, that for a constant $C>0$,
$$ C^{-1}\lambda(t_n)\leq \lambda_n \leq C\lambda(t_n),$$
which shows (extracting again subsequences if necessary) the convergence of $\big(u_{[\lambda(t_n)]}(t_n)\big)_n$ in $\hdot$.
The compactness of $\overline{K}_+$ is proven, which concludes the proof of Lemma \ref{lem.compact}.
\end{proof}
\medskip

\noindent\emph{Step 2: global existence.}
To complete the proof of Proposition \ref{compactness.u}, it remains to show that the maximal time of existence $T_+=T_+(u_0)$ is infinite. Here we use an argument in \cite{KeMe06}. Assume 
\begin{equation}
\label{T+finite.absurd}
T_+<\infty,
\end{equation}
 and consider a sequence $t_n$ that converges to $T_+$. By the finite blow-up criterion of Lemma \ref{lem.CP}, $\|u\|_{S(0,T_+)}=+\infty$. By Lemma \ref{lem.compact}, there exists $\lambda(t)$ such that the set $K_+$ defined by \eqref{defK+lem} is relatively compact in $\hdot$.\par
If there exists a sequence $(t_n)_n$ converging to $T_+$ such that $\lambda(t_n)$ has a finite limit $\lambda_0\geq 0$, then it is easy to show, using the compactness of $\big(u_{[\lambda(t_n)]}(t_n)\big)_n$ and the scaling invariance of \eqref{CP} that $u$ is defined in a neighborhood of $T_+$, which contradicts the fact that $T_+$ is the maximal positive time of definition of $u$. Thus we may assume
\begin{equation}
\label{lambda.infty}
\lim_{t\rightarrow T_+} \lambda(t)=+\infty.
\end{equation}
Consider a positive radial function $\psi$ on $\RR^N$, such that $\psi=1$ if $|x|\leq 1$ and $\psi=0$ if $|x|\geq 2$. Define, for $R>0$ and $t\in[0,T_+)$,
$$ F_R(t):=\int_{\RR^N} |u(t,x)|^2 \psi\big(\frac{x}{R}\big) dx.$$
By \eqref{lambda.infty}, the relative compactness of $K_+$ in $\hdot$ and Sobolev inequality, for all $r_0>0$, $\int_{|x|\geq r_0} |u(t,x)|^{2^*}dx$ tends to $0$ as $t$ tends to $T_+$. Thus, by H\"older and Hardy inequalities
\begin{equation}
\label{F_R0}
\lim_{t\rightarrow T_+} F_R(t)=0.
\end{equation}
Using equation \eqref{CP}, $ F'_R(t)=\frac{2}{R}\im \int u(x)\nabla\ubar(x)(\nabla\psi)\big(\frac{x}{R}\big) dx$, which shows (using Cauchy-Schwarz and Hardy inequalities) that $|F_R'(t)|\leq C\NH{u(t)}^2\leq C_0$, where $C_0$ is a constant which is independent of $R$. Fixing $t\in [0,T_+)$, we see that
\begin{equation}
\label{bound.f_R}
\forall T\in[0,T_+),\quad |F_R(t)-F_R(T)|\leq C_0|t-T|.
\end{equation}
Thus, letting $T$ tends to $T_+$, and using \eqref{F_R0},
$|F_R(t)|\leq C_0|t-T_+|$. Letting $R$ tends to infinity, one gets that $u(t)$ is in $L^2(\RR^N)$ and satisfies
\begin{equation*}
\int_{\RR^N} |u(t,x)|^2 dx\leq C|t-T_+|.
\end{equation*}
By conservation of the $L^2$ norm, we get that $u_0=0$ which contradicts the fact that $E(u_0)=E(W)$. 
This completes the proof that $T_+=+\infty$. By a similar argument, $T_-=-\infty$. The proof of Proposition \ref{compactness.u} is complete.
\qed

\subsection{Existence of mixed behavior solutions}
\label{sub.exist}
We now prove Corollary \ref{corol.exist}.\par
Let $v_n^0=\big(1-\frac 1n\big)W$ and $v_n$ the solution of \eqref{CP} with initial condition $v_n^0$. One may check that $\|v_n^0\|_{\hdot}<\|W\|_{\hdot}$ and $E(v_n^0)<E(W)$.
By the results in \cite{KeMe06}
\begin{equation*}
T_+(v_n^0)=T_-(v_n^0)=\infty,\quad \|v_n\|_{S(\RR)}<\infty.
\end{equation*}
Since $\|W\|_{S(\RR)}=\infty$,
$\|v_n\|_{S(\RR)}$ tends to $+\infty$ by Lemma \ref{lem.CP} \eqref{CPcontinuity}.
Chose $t_n$ such that the solution $u_n(\cdot)=v_n(\cdot+t_n)$ of \eqref{CP} satisfies
\begin{equation}
\label{exist.u_n-}
\|u_n\|_{S(-\infty,0)}=1.
\end{equation}
Therefore
\begin{equation}
\label{exist.u_n+}
\|u_n\|_{S(0,+\infty)}\underset{n\rightarrow+\infty}{\longrightarrow}+\infty.
\end{equation}
Let us use Lemma \ref{lemmacc}. By \eqref{exist.u_n-} and \eqref{exist.u_n+}, cases  \eqref{rightSC}, \eqref{leftSC} and \eqref{allSC} are excluded. Extracting a subsequence from $(u_n)_n$, there exists a sequence $(\lambda_n)_n$ such that $(u_n)_{[\lambda_n]}$ converges in $\hdot$. Rescaling each $u_n$ if necessary (which preserves properties \eqref{exist.u_n-} and \eqref{exist.u_n+}), we may assume that the sequence $\big(u_n^0\big)_n$ converges in $\hdot$ to some $w_0^-$.
Let $w^-$ be the solution of \eqref{CP} such that $w^-(0)=w^-_0$. Clearly
$$ E(w_0^-)=E(W),\quad \NH{w_0^-}\leq \NH{W}.$$
Thus by Proposition \ref{compactness.u}, $w^-$ is defined on $\RR$.\par 
Fix a large integer $n$. By \eqref{exist.u_n-} and Lemma \ref{lem.CP} \eqref{CPcontinuity}  with $I=(-\infty,0]$, $\tilde{u}=u_n$, and $u_0=w^-_0$,  we get
$$\|w^-\|_{S(-\infty,0)}<\infty.$$
Assume that $\|w^-\|_{S(0,+\infty)}$ is finite. Using again Lemma \ref{lem.CP} \eqref{CPcontinuity}  with $I=[0,+\infty)$, $\tilde{u}=w^-$ and $u_0=u_n(0)$, we would get that $\|u_n\|_{S(0,+\infty)}$ is bounded independently of $n$, contradicting \eqref{exist.u_n+}. Thus 
$$\|w^-\|_{S(0,+\infty)}=\infty.$$
The proof is complete.
\qed
\begin{remark}
The above proof gives a general proof of existence of a mixed-behavior solution at the threshold. In Section \ref{sec.proof}, we will give another proof by a fixed point argument, which also works in the supercritical case $\|u_0\|_{\hdot}>\|W\|_{\hdot}$.
\end{remark}
\section{Convergence to $W$ in the subcritical case}
\label{sec.convergence}
In this section, we consider a threshold subcritical radial solution $u$ of \eqref{CP}, satisfying
\begin{gather}
\label{hyp.sub}
E(u_0)=E(W), \quad \NH{u_0}^2<\NH{W}^2\\
\label{hyp.noscatter}
\|u\|_{S(0,+\infty)}=+\infty.
\end{gather}
We will show:
\begin{prop}
\label{prop.CV0}
Let $u$ be a radial solution of \eqref{CP} satisfying \eqref{hyp.sub} and \eqref{hyp.noscatter}. Then there exist $\theta_0\in\RR$, $\mu_0>0$ and $c,C>0$ such that 
$$ \forall t\geq 0,\quad \|u(t)-W_{[\theta_0,\mu_0]}\|_{\hdot} \leq Ce^{-ct}.$$
\end{prop}
(See \eqref{def.modulations} for the definition of $W_{[\theta_0,\mu_0]}$).
As a corollary of the preceding proposition and a step of its proof we get the following result, which completes the proof of the third assertion in \eqref{ex.sub} of Theorem \ref{th.classif}.
\begin{corol}
\label{corol.incomp}
There is no solution $u$ of \eqref{CP} satisfying \eqref{hyp.sub} and
\begin{equation}
\label{non-scatter+-}
\|u\|_{S(-\infty,0)}=\|u\|_{S(0,+\infty)}=+\infty.
\end{equation}
\end{corol}
Let 
\begin{equation}
\label{def.dd}
\dd(f):=\left|\|f\|_{\hdot}^2-\|W\|^2_{\hdot}\right|.
\end{equation}
The key to proving Proposition \ref{prop.CV0} is to show
\begin{equation}
\label{key.delta}
\lim_{t\rightarrow+\infty} \dd(u(t))=0.
\end{equation}
Our starting point of the proof of \eqref{key.delta} is to prove the existence of a sequence $t_n$ going to infinity such that $\dd(u(t_n))$ goes to $0$ (Subsection \ref{sub.mean}). After giving, in Subsection \ref{sub.decompo}, some useful results on the modulation of threshold solutions with respect to the manifold $\{W_{[\theta,\mu]},\; \mu>0,\;\theta\in\RR\}$, we will show the full convergence \eqref{key.delta} and finish the proof of Proposition \ref{prop.CV0} and Corollary \ref{corol.incomp}. 
\subsection{Convergence to $W$ for a sequence}
\label{sub.mean} 
\begin{lemma}
\label{lem.delta.mean}
Let $u$ be a radial solution of \eqref{CP} satisfying \eqref{hyp.sub}, \eqref{hyp.noscatter}, and thus defined on $\RR$ by Proposition \ref{compactness.u}. Then 
\begin{equation}
\label{delta.0}
\lim_{t\rightarrow +\infty} \frac 1T \int_0^T\dd(u(t))dt=0.
\end{equation}
\end{lemma}
\begin{corol}
\label{corol.subseq}
Under the assumptions of Lemma \ref{lem.delta.mean}, there exists a sequence $t_n\rightarrow +\infty$ such that $\dd(u(t_n))$ tends to $0$. 
\end{corol}
\begin{proof}[Proof of Lemma \ref{lem.delta.mean}] 
Let $u$ be a solution of \eqref{CP} satisfying \eqref{hyp.sub} and \eqref{hyp.noscatter}. According to Proposition \ref{compactness.u}, there exists a function $\lambda(t)$ such that
$K_+:=\big\{u_{[\lambda(t)]}(t),\; t\geq 0\big\}$
is relatively compact in $\hdot$. \par
The proof take three steps.
\subsubsection*{Step 1: virial argument}
Let $\varphi$ be a radial, smooth, cut-off function on $\RR^N$ such that
$$ \varphi(x)=|x|^2, \; 0\leq |x|\leq 1,\quad \varphi(x)=0, \; |x|\geq 2.$$
Let $R>0$ and $\varphi_R(x)=R^2 \varphi\big(\frac{x}{R}\big)$, so that $\varphi_R(x)=|x|^2$ for $|x|\leq R$.
Consider the quantity (which is the time-derivative of the localized variance)
$$ G_R(t):=2\im \int \ubar(t) \nabla u(t) \cdot\nabla \varphi_R,\quad t\in \RR.$$
Let us show
\begin{gather}
\label{bound1.gR}
\exists C_*>0,\;\forall t\in \RR,\quad |G_R(t)|\leq C_* R^2,\\
\label{bound.larged}
\forall \eps>0,\;\exists \rho_{\eps}>0,\;\forall R>0, \; \forall t\geq 0,\quad R\lambda(t)\geq \rho_{\eps}\Longrightarrow G'_R(t)\geq \frac{16}{N-2}\dd(u(t))-\eps.
\end{gather}
Using that $|x|\leq 2 R$ on the support of $\varphi_R$ and that $|\nabla \varphi_R|\leq C R$, we get
\begin{equation*}
\forall t\in \RR,\quad |G_R(t)|\leq C R^2 \int_{\RR^N} \frac{1}{|x|}|u(t)||\nabla u(t)|\leq C R^2 \left(\int_{\RR^N} |\nabla u(t)|^2\right)^{1/2}\left(\int_{\RR^N} \frac{1}{|x|^2}{|u(t)|^2}\right)^{1/2},
\end{equation*}
which yields \eqref{bound1.gR}, as a consequence of Hardy's inequality and $\|u(t)\|_{\hdot}\leq \|W\|_{\hdot}$.
To show \eqref{bound.larged}, we will use the compactness of $\overline{K}_+$. By direct computation
\begin{equation}
\label{gR'bis}
G_R'(t)=\frac{16}{N-2}\dd(u(t)) +A_R(u(t)),
\end{equation}
where
\begin{multline}
\label{defAR}
A_R(u):=\int_{|x|\geq R} |\partial_r u|^2\Big(4\frac{d^2\varphi_R}{dr^2}-8\Big)r^{N-1}dr\\+\int_{|x|\geq R}|u|^{2^*}\left(-\frac{4}{N}\Delta\varphi_R+8\right)r^{N-1}dr -\int |u|^2(\Delta^2\varphi_R)r^{N-1}dr.
\end{multline}

Indeed, an explicit calculation yields, together with equation \eqref{CP} 
\begin{align*}
\notag
G_R'(t)=&
4\int \Big|\frac{\partial u}{\partial r}\Big|^2 \frac{d^2\varphi_R}{dr^2} r^{N-1}dr-\frac{4}{N}\int |u|^{2^*}(\Delta \varphi_R) r^{N-1}dr-\int |u|^2(\Delta^2\varphi_R) r^{N-1}dr\\
=& 8\left(\int_{\RR^N} |\nabla u(t)|^2-\int_{\RR^N} |u(t)|^{2*}\right) +A_R(u(t)),
\end{align*}
and as a consequence of assumption \eqref{hyp.sub}, $\int |\nabla u(t)|^2-\int |u(t)|^{2*}=\frac{2}{N-2}\dd(u(t))$ which yields \eqref{gR'bis}. According to \eqref{gR'bis}, we must show
\begin{equation}
\label{eta.0}
\forall \eps>0,\;\exists \rho_{\eps}>0,\;\forall R>0, \; \forall t\geq 0,\quad R\lambda(t)\geq \rho_{\eps}\Longrightarrow |A_R(u(t))|<\eps.
\end{equation}
We have $|\partial_r^2\varphi_R|+|\Delta \varphi_R|\leq C$ and $|\Delta^2\varphi_R|\leq C/R^2$. Thus by \eqref{defAR},
\begin{align}
\notag
|A_R(u(t))| &\leq \int_{|x|\geq R} |\nabla u(t,x)|^{2}+|u(x)|^{2^*}+\frac {1}{|x|^2} |u(t,x)|^2dx\\
\label{forlater}
|A_R(u(t))| &\leq \int_{|y|\geq R\lambda(t)} \big|\nabla u_{[\lambda(t)]}(t,y)\big|^{2}+\big|u_{[\lambda(t)]}(y)\big|^{2^*}+\frac {1}{|y|^2} \big|u_{[\lambda(t)]}(t,y)\big|^2dx,
\end{align}
The set $\overline{K}_+$ being compact in $\hdot$, we get \eqref{eta.0} in view of Hardy and Sobolev inequalities. The proof of \eqref{bound.larged} is complete.

\subsubsection*{Step 2: a bound from below for $\lambda(t)$}
The next step is to show
\begin{equation}
\label{bound.lambda}
\lim_{t\rightarrow +\infty} \sqrt{t}\lambda(t)=+\infty.
\end{equation}
We argue by contradiction. If \eqref{bound.lambda} does not hold, there exists $t_n\rightarrow +\infty$ such that
\begin{equation}
\label{absurd.lambda}
\lim_{n\rightarrow+\infty} \sqrt{t_n}\lambda(t_n)=\tau_0<\infty.
\end{equation}
Consider the sequence $(v_n)_n$ of solutions of \eqref{CP} defined by
\begin{equation}
\label{vn.tau.y}
v_n(\tau,y)=\frac{1}{\lambda(t_n)^\frac{N-2}{2}} u\left(t_n+\frac{\tau}{\lambda^2(t_n)},\frac{y}{\lambda(t_n)}\right).
\end{equation}
By the compactness of $\overline{K}_+$, one may assume that $\big(v_n(0)\big)_n$ converges in $\hdot$ to some function $v_0$. Let $v(\tau)$ be the solution of \eqref{CP} with initial condition $v_0$ at time $\tau=0$. Clearly $E(v_0)=E(W)$ and $\|v_0\|_{\hdot}\leq \|W\|_{\hdot}$. From Proposition \ref{compactness.u}, $v$ is defined on $\RR$. Furthermore, by \eqref{absurd.lambda} and the uniform continuity of the flow of equation \eqref{CP} (Lemma \ref{lem.CP} \eqref{CPcontinuity})
\begin{equation}
\label{CVv_n}
\lim_{n\rightarrow +\infty} v_n(-\lambda^2(t_n)t_n)=v(-\tau_0),\text{ in } \hdot.
\end{equation}
By \eqref{vn.tau.y},
$$ v_n(-\lambda^2(t_n)t_n,y)=\frac{1}{\lambda^{(N-2)/2}(t_n)}u_0\Big(\frac{y}{\lambda(t_n)}\Big).
$$
Assumption \eqref{absurd.lambda} implies that $\lambda(t_n)$ tends to $0$, which shows that $v_n(-\lambda^2(t_n)t_n)\rightharpoonup 0$ weakly in $\hdot$, contradicting \eqref{CVv_n} unless $v(-\tau_0)=0$. This is excluded by the fact that $E(v(-\tau_0))=E(W)$, which concludes
the proof of \eqref{bound.lambda}.
\subsubsection*{Step 3: conclusion of the proof}
\label{CV.in.mean}
Fix $\eps>0$. We will use the estimates \eqref{bound1.gR} and \eqref{bound.larged} of Step $1$ with an appropriate choice of $R$. Consider the positive number $\rho_{\eps}$ given by \eqref{bound.larged}. Take $\eps_0$ and $M_0$ such that
$$ 2C_* \eps_0^2=\eps,\qquad M_0\eps_0=\rho_{\eps},$$
where $C_*$ is the constant of inequality \eqref{bound1.gR}. By Step 2 there exists $t_0$ such that
\begin{equation*}
\forall t\geq t_0, \quad \lambda(t)\geq \frac{M_0}{\sqrt{t}}.
\end{equation*}
Consider, for $T\geq t_0$ 
$$R:=\eps_0\sqrt{T}.$$ 
If $t\in [t_0,T]$, then the definitions of $R$, $M_0$ and $t_0$ imply
$ R\lambda(t)\geq \eps_0\sqrt{T}\frac{M_0}{\sqrt{t}}\geq \rho_{\eps}$.
Integrating \eqref{bound.larged} between $t_0$ and $T$ and using estimate \eqref{bound1.gR} on $G_R$, we get, by the choice of $\eps_0$ and $R$
$$ \frac{16}{N-2} \int_{t_0}^{T}|\dd(u(t))| dt\leq 2C_* R^2+\eps(T-t_0) \leq 2C_*R^2+\eps T\leq \frac{\eps}{\eps_0^2}\eps_0^2 T+\eps T\leq 2\eps T.$$
Letting $T$ tends to $+\infty$,
$$\limsup_{T\rightarrow +\infty} \frac{1}{T}\int_{0}^{T} |\dd(u(t))| dt\leq \frac{N-2}{8}\eps,$$
which concludes the proof of Lemma \ref{lem.delta.mean}.
\end{proof}

\subsection{Modulation of threshold solutions}
\label{sub.decompo}
Let $f$ be in $\hdot$ such that $E(f)=E(W)$. The variational characterization of $W$ \cite{Au76,Ta76,Li85Reb} shows
$$\inf_{\substack{\theta\in\RR\\ \mu>0}} \|f_{[\theta,\mu]}-W\|_{\hdot}\leq \eps(\dd(f)), \quad \lim_{\delta \rightarrow 0^+} \eps(\delta)=0,$$ 
where $\dd(f)$ is defined in \eqref{def.dd}.
We introduce here a choice of the modulation parameters $\theta$ and $\mu$ for which the quantity $\dd(f)$ controls linearly $\|f_{[\theta,\mu]}-W\|_{\hdot}$ and other relevant parameters of the problem. This choice is made through two orthogonality conditions given by the two groups of transformations $f\mapsto e^{i\theta}f$, $\theta\in\RR$ and $f\mapsto f_{[\mu]}$, $\mu>0$. This decomposition is then applied to solutions of \eqref{CP}.
Let us start with a few notations.

Since $W$ is a critical point of $E$, we have the following development of the energy near $W$:
\begin{equation}
\label{W+g}
E(W+g)=E(W)+Q(g)+O\big(\|g\|^3_{\hdot}\big),\quad g\in \hdot,
\end{equation}
where $Q$ is the quadratic form on $\hdot$ defined by
$$Q(g):=\frac{1}{2}\int |\nabla g|^2-\frac{1}{2} \int W^{p_c-1}\big(p_c(\re g)^2+(\im g)^2\big).$$
Let us specify an important coercivity property of $Q$. Consider the three orthogonal directions $W$, $iW$ and $\tW:=\frac{N-2}{2}W+x\cdot \nabla W$ in the real Hilbert space $\hdot=\hdot(\RR^N,\CC)$. Note that
\begin{equation}
\label{iW.W1}
iW=\frac{d}{d\theta}\left(e^{i\theta} W\right)_{\restriction \theta=0},\quad \tW=-\frac{d}{d\lambda}\left( W_{[\lambda]} \right)_{\restriction \lambda=1}.
\end{equation}
Let $H:=\vect\{W,iW,\tW\}$ and $H^{\bot}$ its orthogonal subspace in $\hdot$ for the usual scalar product. Then
\begin{equation}
\label{valeursQ}
Q(W)=-\frac{2}{(N-2)C_N^N}, \quad Q_{\restriction \vect\{iW,\tW\}}= 0,
\end{equation} 
where $C_N$ is the best Sobolev constant in dimension $N$.  The first assertion follows from direct computation and the fact that $\|W\|_{L^{2^*}}^{2^*}= \NH{W}^2=\frac{1}{C_N^N}$. The second assertion is an immediate consequence of \eqref{W+g}, \eqref{iW.W1}, and the invariance of $E$ by the transformations $f\mapsto f_{[\theta,\lambda]}$. The quadratic form $Q$ is nonpositive on $H$. By the following claim, $Q$ is positive definite on $H^{\bot}$.
\begin{claim}
\label{coercivity}
There is a constant $\tilde{c}>0$ such that for all radial function $\tilde{f}$ in $H^{\bot}$
$$ Q(\tilde{f}) \geq \tilde{c}\|\tilde{f}\|_{\hdot}^2.$$
\end{claim}
\begin{proof}
Let $\tilde{f}_1:=\re \tilde{f}$, $\tilde{f}_2:=\im \tilde{f}$. We have 
$$Q(\tilde{f})=\frac{1}{2}\int_{\RR^N} |\nabla \tilde{f}_1|^2-\frac{p_c}{2}\int_{\RR^N}W^{p_c-1}|\tilde{f}_1|^2+\frac{1}{2}\int_{\RR^N} |\nabla \tilde{f}_2|^2-\frac{1}{2}\int_{\RR^N}W^{p_c-1}|\tilde{f}_2|^2.$$ 
The inequality
$$ \exists c_1>0,\;\forall \tilde{f}_1\in \{W,W_1\}^{\bot},\quad\frac{1}{2}\int_{\RR^N} |\nabla \tilde{f}_1|^2-\frac{p_c}{2}\int_{\RR^N}W^{p_c-1}|\tilde{f}_1|^2\geq c_1 \int_{\RR^N} |\nabla \tilde{f}_1|^2$$
is known. We refer to \cite[Appendix D]{Re90} for a proof in a slightly different context, but which readily extends to our case. It remains to show
\begin{equation}
\label{im.coer}
\exists c_2>0,\;\forall \tilde{f}_2\in \hdot, \quad \tilde{f}_2\bot W\Longrightarrow \frac{1}{2}\int_{\RR^N} |\nabla \tilde{f}_2|^2-\frac{1}{2}\int_{\RR^N}W^{p_c-1}|\tilde{f}_2|^2\geq c_2\int_{\RR^N} |\nabla \tilde{f}_2|^2.
\end{equation}
Indeed by H\"older and Sobolev inequality we have, for any real-valued $v\in \hdot$
\begin{align*}
\int_{\RR^N} |\nabla v|^2-\int_{\RR^N}W^{p_c-1}v^2&\geq \int_{\RR^N} |\nabla v|^2-\left(\int_{\RR^N} W^{p_c+1}\right)^{\frac{p_c-1}{p_c+1}}\left(\int v^{p_c+1}\right)^{\frac{2}{p+1}}\\
&\geq \left\{1-\left(\frac{1}{C_N^N}\right)^{\frac{p_c-1}{p_c+1}}C_N^2\right\}\int_{\RR^N} |\nabla v|^2\geq 0,
\end{align*}
with equality if and only if $v\in \vect(W)$. This shows that $\int_{\RR^N} |\nabla \tilde{f}_2|^2-\int_{\RR^N}W^{p_c-1}|\tilde{f}_2|^2>0$ for $\tilde{f}_2\neq 0$, $\tilde{f}_2\bot W$. Noting that the quadratic form $\int_{\RR^N} |\nabla \cdot|^2-\int_{\RR^N}W^{p_c-1}|\cdot|^2$ is a compact perturbation of $\int_{\RR^N} |\nabla \cdot|^2$, it is easy to derive \eqref{im.coer}, using a straightforward compactness argument that we omit here.
\end{proof}

The following lemma, proven in Appendix \ref{Appendix.Decompo}, is a consequence of the Implicit Function Theorem.
\begin{lemma}
\label{lem.ortho}
There exists $\delta_0>0$ such that for all $f$ in $\hdot$ with $E(f)=E(W)$, $\dd(f)<\delta_0$, there exists a couple $(\theta,\mu)$ in $\RR\times (0,+\infty)$ with
$$ f_{[\theta,\mu]} \bot iW,\quad f_{[\theta,\mu]} \bot \tW.$$
The parameters $\theta$ and $\mu$ are unique in $\RR /_{\ds 2\pi\ZZ} \times \RR$, and the mapping $f\mapsto (\theta,\mu)$ is $C^1$.
\end{lemma}

Let $u$ be a solution of \eqref{CP} on an interval $I$ such that $E(u_0)=E(W)$, and, on $I$, $\dd(u(t))<\delta_0$. According to Lemma \ref{lem.ortho}, there exist real parameters $\theta(t)$, $\mu(t)>0$ such that
\begin{align}
\label{decompo.v}
&u_{[\theta(t),\mu(t)]}(t)=(1+\alpha(t)) W+\tilde{u}(t),\\ \notag &\quad\text{where }1+\alpha(t)=\frac{1}{\|W\|_{\hdot}^2}\left(u_{[\theta(t),\mu(t)]},W\right)_{\hdot} \text{ and } \tilde{u}(t)\in H^{\bot}.
\end{align}
We define $v(t)$ by
$$ v(t):=\alpha(t)W+\tilde{u}(t)=u_{[\theta(t),\mu(t)]}(t)-W.$$
Recall that $\mu$, $\theta$ and $\alpha$ are $C^1$. 
If $a$ and $b$ are two positive quantities, we write $a\approx b$ when $C^{-1}a \leq b\leq Ca$ with a positive constant $C$ independent of all parameters of the problem. We will prove the following lemma, which is a consequence of Claim \ref{coercivity} and of the equation satisfied by $v$, in Appendix \ref{Appendix.Decompo}.
\begin{lemma}[Modulation for threshold solutions of \eqref{CP}]
\label{bound.modul}
Taking a smaller $\delta_0$ if necessary, we have the following estimates on $I$.
\begin{gather}
\label{approximations}
|\alpha(t)|\approx \|v(t)\|_{\hdot}\approx\|\tilde{u}(t)\|_{\hdot}\approx \dd(u(t))\\
\label{bound.derives}
|\alpha'(t)|+|\theta'(t)|+\left|\frac{\mu'(t)}{\mu(t)}\right| \leq C\mu^2(t)\dd(u(t)).
\end{gather}
Furthermore, $\alpha(t)$ and $\|u(t)\|^2_{\hdot}-\|W\|_{\hdot}^2$ have the same sign.
\end{lemma}

\subsection{Non-oscillatory behavior near $W_{[\theta_0,\mu_0]}$}
\label{sub.CVseq}
\begin{lemma}
\label{lem.CV0.seq}
Let $(t_{0n})_n$ and $(t_{1n})_n$, $t_{0n}<t_{1n}$, be $2$ real sequences, $(u_n)_n$ a sequence of radial solutions of \eqref{CP} on $\left[t_{0n},t_{1n}\right]$ such that $u_n\left(t_{1n}\right)$ fullfills assumptions \eqref{hyp.sub} and \eqref{hyp.noscatter}, and  $(\lambda_n)_{n\in\NN}$ a sequence of positive functions such that the set:
$$\widetilde{K}=\left\{(u_n(t))_{[\lambda_n(t)]},\;n\in \NN,\;t\in (t_{1n},t_{2n})\right\}$$
is relatively compact in $\hdot$. Assume
\begin{equation}
\label{delta.un.hypo}
\lim_{n\rightarrow+\infty}\dd(u_n(t_{0n}))+\dd(u_n(t_{1n}))=0.
\end{equation}
Then
\begin{equation}
\label{delta.un.conc}
\lim_{n\rightarrow+\infty}\, \left\{\sup_{t\in(t_{0n},t_{1n})}\dd\big(u_n(t)\big)\right\} =0.
\end{equation}
\end{lemma}

\begin{remark}
Let $u$ be a solution of \eqref{CP} satisfying the assumptions of Proposition \ref{prop.CV0}, and $\lambda(t)$ the parameter given by Proposition \ref{compactness.u}. Let $(t_n)$ be a sequence, given by Corollary \ref{corol.subseq}, such that $\dd(u(t_n))$ tends to $0$. Then the assumptions of the preceding proposition are fullfilled with $u_n=u$, $\lambda_n=\lambda$, $t_{0n}=t_n$, $t_{1n}=t_{n+1}$.
\end{remark}
Under the assumptions of Lemma \ref{lem.CV0.seq}, if $n$ is large enough so that $\dd(u_n(t))<\delta_0$ on the interval $(t_{0n},t_{1n})$, we will denote by $\theta_n(t)$, $\mu_n(t)$ and $\alpha_n(t)$ the parameters of decomposition \eqref{decompo.v}
\begin{equation}
\label{defmun}
\big(u_n(t)\big)_{[\theta_n(t),\mu_n(t)]}=\big(1+\alpha_n(t)\big)W+\tilde{u}_n(t).
\end{equation}
Then we can complete Lemma \ref{lem.CV0.seq} by the following.
\begin{lemma}
\label{lem.mu.seq}
Under the assumptions of Lemma \ref{lem.CV0.seq},
\begin{equation}
\label{mu.seq.conclu}
\lim_{n\rightarrow+\infty} \frac{\sup_{t\in(t_{0n},t_{1n})} \mu_n(t)}{\inf_{t\in(t_{0n},t_{1n})}\mu_n(t)}=1.
\end{equation}
\end{lemma}
Using the scaling invariance, it is sufficient to prove the preceding Lemmas assuming 
\begin{equation}
 \label{minor.lambda}
\forall n,\quad \inf_{t\in[t_{0n},t_{1n}]} \lambda_n(t)=1.
\end{equation}
Indeed, let $\ds \ell_n:=\inf_{t\in(t_{0n},t_{1n})} \lambda_n(t)$, and
\begin{gather*}
u^*_n(t,x)=\frac{1}{\ell_{n}^{\frac{N-2}{2}}}u_n\left(\frac{t}{\ell_n^2},\frac{x}{\ell_n}\right),\quad \lambda_n^*(t)=\frac{\lambda_n(t)}{\ell_n},\quad t_{0n}^*=\frac{t_{0n}}{\ell_n^2},\quad t_{1n}^*=\frac{t_{1n}}{\ell_n^2}\\ \widetilde{K}^*=\left\{(u_n^*(t))_{[\lambda_n^*(t)]},\;n\in \NN,\;t\in (t_{0n}^*,t_{1n}^*)\right\}.
\end{gather*}
Then $u_n^*$, $t_{0n}^*$, $t_{1n}^*$, $\lambda_n^*$ and $\widetilde{K}^*$ fullfill the assumptions of Lemmas \ref{lem.CV0.seq} and \ref{lem.mu.seq}. Furthermore, the conclusions \eqref{delta.un.conc} and \eqref{mu.seq.conclu} of the Lemmas are not changed by the preceding transformations. We will thus assume \eqref{minor.lambda} throughout the proofs. 

The key point of the proofs is the following claim, which is a consequence of a localized virial argument.
\begin{claim}
\label{claim.virial.seq}
Let $(u_n)_n$ be a sequence fullfilling the assumptions of Lemma \ref{lem.CV0.seq} and \eqref{minor.lambda}.
Then 
$$ \forall n\in\NN,\quad \int_{t_{0n}}^{t_{1n}} \dd(u_n(t))dt\leq C\big[\dd(u_n(t_{0n}))+\dd(u_n(t_{1n}))\big].$$
\end{claim}
Before proving Claim \ref{claim.virial.seq}, we will show that it implies the above lemmas.
\subsubsection*{Proof of Lemma \ref{lem.CV0.seq}}
Let $(u_n)_n$ be as in Lemma \ref{lem.CV0.seq}, and assume \eqref{minor.lambda}. 
We first prove:
\begin{claim}
\label{keyclaim}
If $t_n\in (t_{0n},t_{1n})$ and the sequence $\lambda_n(t_n)$ is bounded, then
\begin{equation}
\label{delta.hatt.n}
\lim_{n\rightarrow +\infty} \dd\big(u_n(t_n)\big)=0.
\end{equation}
\end{claim}
\begin{proof} 
By our assumptions, $1\leq\lambda_n(t_n)\leq C$, for some $C>1$, so that the sequence $u_n(t_n)$ is compact. Assume that \eqref{delta.hatt.n} does not hold, so that, up to the extraction of a subsequence
\begin{equation}
\label{dv0>0}
\lim_{n\rightarrow +\infty} u_n(t_n)=v^0\text{ in } \Hdot,\quad \dd(v^0)>0,\;E(v^0)=E(W)\text{ and } \|v^0\|_{\hdot}<\|W\|_{\hdot}.
\end{equation}
Let $v$ be the solution of \eqref{CP} with initial condition $v^0$ at time $t=0$, which is defined for $t\geq 0$. Note that for large $n$, $1+t_{n}\leq t_{1n}$. If not, $t_{1n}\in (t_n,1+t_{n})$ for an infinite number of $n$, so that extracting a subsequence, $t_{1n}-t_n$ has a limit $\tau\in[0,1]$. By the continuity of the flow of \eqref{CP} in $\hdot$, $u_n(t_{1n})$ tends to $v(\tau)$ with $E(v(\tau))=E(W)$ and, by \eqref{delta.un.hypo}, $\dd(v(\tau))=0$. This shows that $v=W_{[\theta_0,\lambda_0]}$ for some $\theta_0,\,\lambda_0$, contradicting \eqref{dv0>0}. Thus $(t_n,1+t_n)\subset(t_{0n},t_{1n})$. By \eqref{dv0>0} and the continuity of the flow of \eqref{CP},
\begin{equation}
\label{contra1}
\lim_{n\rightarrow +\infty} \int_{t_n}^{1+t_n}\dd(u_n(t))dt=\int_0^1 \dd(v(t))dt>0.
\end{equation}
Furthermore, according to Claim \ref{claim.virial.seq}
$\lim_{n} \int_{t_{0n}}^{t_{1n}} \dd(u_n(t)) dt=0.$
which contradicts \eqref{contra1}. The proof is complete.
\end{proof}
By assumption \eqref{minor.lambda}, one may chose, for every $n$, $b_n\in (t_{0n},t_{1n})$ such that 
\begin{equation}
\label{lambda1}
\lim_{n\rightarrow+\infty} \lambda_n(b_n)=1.
\end{equation}
By Claim \ref{keyclaim}
\begin{equation}
\label{delta.b.n}
\lim_{n\rightarrow +\infty} \dd(u_n(b_n))=0.
\end{equation}
We will show \eqref{delta.un.conc} by contradiction. Let us assume (after extraction) that for some $\delta_1>0$,
\begin{equation*}
\forall n,\quad \sup_{t\in (t_{0n},b_n)} \dd(u_n(t))\geq \delta_1>0
\end{equation*}
 (the proof is the same when $(t_{0n},b_n)$ is replaced by $(b_n,t_{1n})$ in the supremum). 
Fix $\delta_2>0$ smaller than $\delta_1$ and the constant $\delta_0$ given by Lemma \ref{lem.ortho}.
The mapping $t\mapsto\dd\big(u_n(t)\big)$ being continuous, there exists $a_n\in (t_{0n},b_n)$ such that
\begin{equation}
\label{defan}
\dd(u_n(a_n))=\delta_2 \text{ and }\forall t\in(a_n,b_n),\; \dd(u_n(t))<\delta_2.
\end{equation}
On $(a_n,b_n)$, the modulation parameter $\mu_n$ is well defined. Furthermore, by the relative compactness of $\widetilde{K}$ and decomposition \eqref{defmun}, the set $\bigcup_n \Big\{W_{\big[{\lambda_n(t_n)}_{\scriptstyle/\mu_n(t_n)}\big]}(t),\; t\in [a_n,b_n]\Big\}$ must be relatively compact, which shows
\begin{equation}
\label{mu=lambda}
\exists C>0,\; \forall t\in (a_n,b_n),\quad C^{-1}\lambda_n(t)\leq \mu_n(t)\leq C\lambda_n(t).
\end{equation}
By \eqref{lambda1}, extracting a subsequence if necessary, we may assume 
$$\mu_n(b_n) \underset{n\rightarrow +\infty}{\longrightarrow} \mu_{\infty}\in (0,\infty).$$
Let us show by contradiction
\begin{equation}
\label{mubounded}
\sup_{n,t\in(a_n,b_n)}\mu_n(t)<\infty.
\end{equation}
If not, in view of the continuity of $\mu_n$, there exists (for large $n$) $c_n\in (a_n,b_n)$ such that 
\begin{equation}
 \label{defcn}
\mu_n(c_n)=2 \mu_{\infty},\quad \mu_n(t)<2\mu_{\infty},\; t\in (c_n,b_n).
\end{equation}
By Claim \ref{keyclaim}, $\lim_n \dd(u(c_n))=0$. Furthermore, by Lemma \ref{bound.modul}, $\left|\frac{\mu'_n(t)}{\mu^3_n(t)}\right|\leq C\dd(u_n(t))$. Integrating between $c_n$ and $b_n$, we get, by Claim \ref{claim.virial.seq},
\begin{equation}
\label{bound.mu2}
\left|\frac{1}{\mu^2_n(c_n)}-\frac{1}{\mu^2_n(b_n)} \right|\leq C\int_{c_{n}}^{b_n} \dd(u_n(s)) ds\underset{n\rightarrow +\infty}{\longrightarrow} 0,
\end{equation}
which contradicts the fact that $\mu_n(c_n)=2\mu_{\infty}$ and $\mu_n(b_n)\rightarrow \mu_{\infty}$, and thus concludes the proof of \eqref{mubounded}.\par
By \eqref{mubounded}, $\mu_n(a_n)$ is bounded. Claim \ref{keyclaim} shows that $\dd(u_n(a_n))$ tends to $0$, contradicting \eqref{defan}. The proof of Lemma \ref{lem.CV0.seq} is complete.\qed

\subsubsection*{Proof of Lemma \ref{lem.mu.seq}}
It follows from the argument before Claim \ref{claim.virial.seq} that we may assume \eqref{minor.lambda} in addition to the assumptions of the lemma, so that by \eqref{mu=lambda} 
$$\exists C>0, \;\forall n, \quad C^{-1} \leq \inf_{t\in[t_{0n},t_{1n}]}\mu_n(t) \leq C.$$ 
Furthermore, in view of the continuity of $\mu_n$, there exist $a_n,\,b_n\in [t_{0n},t_{1n}]$ such that
$$ \mu_n(a_n)=\inf_{t\in[t_{0n},t_{1n}]}\mu_n(t),\quad \mu_n(b_n)=\sup_{t\in[t_{0n},t_{1n}]}\mu_n(t).$$
By the bound $\left|\frac{\mu_n'(t)}{\mu_n^3(t)}\right|\leq C\dd(u_n(t))$, Claim \ref{claim.virial.seq},  and Lemma \ref{lem.CV0.seq} , we get
\begin{equation*}
\lim_{n\rightarrow +\infty} \left|\frac{1}{\mu_n^2(a_n)}-\frac{1}{\mu_n^2(b_n)}\right| =0.
\end{equation*}
In particular, $\mu_n(b_n)$ is bounded. Multiplying the preceding limit by $\mu_n^2(b_n)$ yields \eqref{mu.seq.conclu}. 
\qed

\subsubsection*{Proof of Claim \ref{claim.virial.seq}}
Let us consider, for $R>0$, the function $G_{R,n}$ defined as in Subsection \ref{CV.in.mean} by
$$ G_{R,n}(t)=2\im \int \ubar_n(t) \nabla u_n(t) \cdot\nabla \varphi_R$$
($\varphi_R$ is defined in Subsection \ref{CV.in.mean}). 

\medskip

\paragraph{\emph{Step 1: a bound for $G_{R,n}$.}}

In this step we show that there exists a constant $C>0$ such that
\begin{equation}
\label{bound.above.GR}
\forall R>0,\;\forall n,\; \forall t\in(t_{0n},t_{1n}) ,\quad |G_{R,n}(t)|\leq C R^2\dd(u_n(t)).
\end{equation}
We have
\begin{equation*}
G_{R,n}(t)=2\im \int \ubar_n(t,x)\nabla u_n(t,x) \cdot R\nabla \varphi(x/R)dx.
\end{equation*}
By Cauchy-Schwarz and Hardy inequalities $|G_{R,n}(t)|\leq R^2\|u_n\|^2_{\hdot}$, so that it suffices to show \eqref{bound.above.GR} when $\dd(u_n(t))\leq \delta_1$ for some small $\delta_1$. In this case, one may decompose $u_n$ as in \eqref{decompo.v}, writing $\left(u_n(t)\right)_{[\theta_n(t),\mu_n(t)]}=W+v_n(t)$, with $\|v_n(t)\|_{\hdot}\leq C\dd(u_n(t))$ by Lemma \ref{lem.CV0.seq}. By the change of variable $x=\frac{y}{\mu_n(t)}$,
\begin{align*}
G_{R,n}(t)=&2\im \frac{R}{\mu_n(t)}\int \frac{1}{\mu_n^{\frac{N-2}{2}}(t)}\ubar_n\Big(t,\frac{y}{\mu_n(t)}\Big)\frac{1}{\mu_n^{N/2}(t)}(\nabla u_n)\Big(t,\frac{y}{\mu_n(t)}\Big)\cdot \nabla \varphi\Big(\frac{y}{R\mu_n(t)}\Big)dy\\
=&2R^2\im \int \frac{1}{R\mu_n(t)}(W+\overline{v}_n)\nabla \big(W+v_n\big)\cdot (\nabla\varphi)\Big(\frac{y}{R\mu_n(t)}\Big)dy. 
\end{align*}
Write $$\im\left[(W+\overline{v}_n)\nabla \big(W+v_n\big)\right]=\im\big(W\nabla v_n+\overline{v}_n\nabla W+\overline{v}_n\nabla v_n\big),$$ 
and note that on the support of $\nabla\varphi\Big(\frac{y}{R\mu_n(t)}\Big)$, $\frac{1}{R\mu_n(t)}$ is bounded by $\frac{2}{|y|}$. As a consequence of Cauchy-Schwarz and Hardy inequalities, we get the bound $|G_{R,n}(t)|\leq CR^2\big(\|v_n(t)\|_{\hdot}+\|v_n(t)\|_{\hdot}^2\big)$, which yields \eqref{bound.above.GR}, for $\dd(u_n(t))\leq \delta_1$, $\delta_1$ small. The proof of \eqref{bound.above.GR} is complete.

\medskip

\paragraph{\emph{Step 2: a bound from below for $G_{R,n}'$}}
The next and last step of the proof of Claim \ref{claim.virial.seq} is to show 
\begin{equation}
\label{bound.G'Rn}
\exists R_0,\; \forall R\geq R_0,\; \forall n,\;\forall t\in(t_{0n},t_{1n}),\quad G'_{R,n}(t)\geq \frac{8}{N-2}\dd(u_n(t)).
\end{equation}

It is clear that \eqref{bound.above.GR} and \eqref{bound.G'Rn} imply the conclusion of Claim \ref{claim.virial.seq}. Indeed, integrating \eqref{bound.G'Rn} between $t_{0n}$ and $t_{1n}$ we get
$$ \frac{8}{N-2} \int_{t_{0n}}^{t_{1n}} \dd(u_n(t))dt\leq G_{R_0,n}(t_{0n})+G_{R_0,n}(t_{1n}),$$
which shows the Claim in view of \eqref{bound.above.GR}.

Let us show \eqref{bound.G'Rn}. Recall that by direct calculation we have, as in \eqref{gR'bis},
\begin{equation}
\label{gR'ter}
G_{R,n}'(t)=\frac{16}{N-2}\dd\big(u_n(t)\big) +A_R(u_n(t)),
\end{equation}
where $A_R$ is defined by \eqref{defAR}. We first claim the following bounds on $A_R(u_n(t))$:
\begin{gather}
\label{bound.larged.2}
\forall \eps>0, \; \exists\rho_{\eps}>0,\;\forall n,\; \forall t\in\big(t_{0n},t_{1n}\big),\; \forall R\geq\frac{\rho_{\eps}}{\lambda_n(t)},\quad |A_R(u_n(t))|\leq \eps\\
\label{bound.smalld}
\exists \delta_2>0,\; \forall n,\; \forall t\in\big(t_{0n},t_{1n}\big),\; \forall R\geq \frac{1}{\mu_n(t)},\qquad\qquad\qquad \qquad\qquad \qquad\qquad\qquad \\ 
\notag
\qquad\qquad\qquad\dd(u_n(t))\leq \delta_2 \Longrightarrow  |A_{R,n}(u_n(t))|\leq C\left(\frac{1}{\left(R\mu_n(t)\right)^{\frac{N-2}{2}}}\dd(u_n(t))+\dd(u_n(t))^2\right).
\end{gather}

The bound \eqref{bound.larged.2}, follows directly from the compactness of $\widetilde{K}$, assumption \eqref{minor.lambda}, and the bound \eqref{forlater} of $A_R$ shown in the preceding section. \par
Let us show \eqref{bound.smalld}. Write as before 
\begin{equation}
\label{decompo2}
(u_n(t))_{[\theta_n(t),\mu_n(t)]}=W+v_n(t),\quad \|v_n(t)\|_{\hdot}\leq C\dd(u_n(t)).
\end{equation}
In view of \eqref{decompo2}, estimate \eqref{bound.smalld} is an immediate consequence of the existence of $\delta_2>0$ such that
\begin{multline}
\label{ARforlater}
\forall g\in \hdot_r,\; \forall \mu_0>0,\;\forall R\geq \frac{1}{\mu_0},\\ \|g\|_{\hdot}\leq \delta_2\Longrightarrow
\bigg|A_{R}\Big((W+g)_{[\mu_0^{-1}]}\Big)\bigg|\leq C\left(\frac{1}{\left(R\mu_0\right)^{\frac{N-2}{2}}}\|g\|_{\hdot}+\|g\|_{\hdot}^2\right).
\end{multline}
Let us show \eqref{ARforlater}. A change of variable in $A_R$ gives
$A_R\Big((W+g)_{[\mu_0^{-1}]}\Big)= A_{R\mu_0}(W+g).$
The function $W$ is a stationnary solution of \eqref{CP}, satisfying $\dd(W)=0$ and $G_R(W)=0$,
so that by \eqref{gR'bis}, $A_R(W)=0$ for any $R>0$. Thus we must bound
$A_{R\mu_0}(W+g)-A_{R\mu_0}(W)$. By the explicit form of $A_R$,
\begin{multline*}
A_R(f)=\int |\nabla f|^2\Big(4\frac{d^2\varphi_R}{dr^2}-8\Big)+|f|^{2^*}\left(-\frac{4}{N}\Delta\varphi_R+8\right)r^{N-1}dr -\int |f|^2(\Delta^2\varphi_R)r^{N-1}dr,
\end{multline*}
and noting that the integrand in the first integral is supported in $\{|x|\geq R\}$ and in the second integral in $\{R\leq |x|\leq 2R\}$, we get
\begin{multline*}
\left|A_{R\mu_0}(W+g)-A_{R\mu_0}(W)\right|
\leq C\bigg[\int_{|x|\geq R\mu_0} |\nabla g|^2+|\nabla W\cdot\nabla g|+ W^{2^*-1} |g|+|g|^{2*}dx\\ +\int_{R\mu_0\leq|x|\leq 2R\mu_0} \frac{1}{(R\mu_0)^2}\left(W|g|+|g|^2\right)dx\bigg].
\end{multline*}
By explicit calculation, $\|\nabla W\|_{L^2(\{|x|\geq \rho\})}\approx \|W\|_{L^{2^*}(\{|x|\geq \rho\})}\approx \frac{1}{\rho^{\frac{N-2}{2}}}$ for large $\rho$. Hence, by Hardy, Sobolev and Cauchy-Schwarz inequalities
\begin{equation*}
|A_{R\mu_0}(W+g)-A_{R\mu_0}(W)|\leq C\left[\|g\|_{\hdot}^2+\|g\|_{\hdot}^{2*}+\left(\frac{1}{(R\mu_0)^{\frac{N-2}{2}}}+\frac{1}{(R\mu_0)^{\frac{N+2}{2}}}\right)\|g\|_{\hdot}\right],
\end{equation*}
which yields \eqref{ARforlater}, and thus \eqref{bound.smalld}.\par

We are now ready to show \eqref{bound.G'Rn}. By assumption \eqref{minor.lambda}, $\lambda_n(t)$ is bounded from below. By \eqref{mu=lambda}, $\mu_n(t)\geq C_*>0$. Thus \eqref{bound.smalld} implies for some $\delta_3>0$, $R_1>0$
$$ \dd(u_n(t))\leq \delta_3,\; R\geq R_1\Longrightarrow \left|A_R(u_n(t))\right|\leq \frac{8}{N-2}\dd(u_n(t)).$$
Now, using \eqref{bound.larged.2} with $\eps=\frac{8\delta_3}{N-2}$ and again \eqref{minor.lambda}, we get $\left|A_R(u_n(t))\right|\leq \frac{8}{N-2}\dd(u_n(t))$ for $\dd(u_n(t))\geq \delta_3$, $R\geq R_2$. In view of \eqref{gR'ter}, estimate \eqref{bound.G'Rn} holds with $R_0:=\max\{R_1,R_2\}$, which concludes the proof of Claim \ref{claim.virial.seq}.
\qed

\subsection{Proof of the convergence as $t$ goes to infinity}
\label{sub.dem.SC}
Let us show Proposition \ref{prop.CV0} and Corollary \ref{corol.incomp}.
Let $u$ be a radial solution of \eqref{CP} satisfying \eqref{hyp.sub} and \eqref{hyp.noscatter}.
\subsubsection*{Step 1: convergence of $\dd(u(t))$ to $0$}
We first prove \eqref{key.delta}.
From Corollary \ref{corol.subseq}, there exists a strictly increasing sequence $(t_n)_{n\in \NN}$ such that:
$$ \lim_{n\rightarrow +\infty} t_n=+\infty,\quad \lim_{n\rightarrow +\infty} \dd(u(t_n))=0.$$
Let $t_{0n}=t_n$, $t_{1n}=t_{n+1}$, and $\lambda_n(t)=\lambda(t)$, where $\lambda$ is given by Proposition \ref{compactness.u}. Then the sequences $(u_n)_n$, $(t_{0n})_n$, $(t_{1n})_n$ and $(\lambda_n)_n$ clearly satisfy the assumptions of Lemma \ref{lem.CV0.seq}. Hence
$$ \lim_{n\rightarrow +\infty} \left(\sup_{t\in [t_n,t_{n+1}]} \dd(u(t))\right)=0,$$
which clearly implies \eqref{key.delta}.\par

As a consequence of \eqref{key.delta}, we may decompose $u$ for large $t$ as in \eqref{decompo.v}:
\begin{equation*}
u_{[\theta(t),\mu(t)]}=(1+\alpha(t))W+\tilde{u}(t),\quad \tilde{u}(t)\in H^{\bot}.
\end{equation*}
If $\theta$, $\mu$ and $\alpha$ are given by the preceding decomposition, the conclusion of  Proposition \ref{prop.CV0} is equivalent to the existence of $\mu_{\infty}>0$, $\theta_{\infty}\in\RR$ and $c,C>0$ such that
\begin{gather}
\label{lim.expo}
\dd(u(t))+
|\alpha(t)|+\|\tilde{u}(t)\|_{\hdot}+|\theta(t)-\theta_{\infty}|+|\mu(t)-\mu_{\infty}| \leq Ce^{-ct}.
\end{gather}
\subsubsection*{Step 2: convergence of $\mu$}
We start to show by contradiction that $\mu(t)$ has a limit $\mu_{\infty}\in (0,+\infty)$ as $t\rightarrow +\infty$. If not, $\log(\mu(t))$ does not satisfy the Cauchy criterion as $t\rightarrow+\infty$, thus there exists sequences $T_n,\,T'_ n\rightarrow+\infty$ such that
\begin{equation}
\label{contra.mu}
\lim_{n\rightarrow +\infty} \frac{|\mu(T_n)|}{|\mu(T'_n)|}=L\neq 1.
\end{equation}
Without loss of generality, we may assume $T_n<T'_n$.
By the preceding step, $d(u(T_n))$ and $d(u(T'_n))$ tends to $0$. Let $u_{n}=u$, $t_{0n}=T_n$, $t_{1n}=T'_n$, and $\lambda_{n}(t)=\lambda(t)$, where $\lambda$ is again given by Proposition \ref{compactness.u}. Then the assumptions of Lemma \ref{lem.mu.seq} are fullfilled, which shows
$$ \lim_{n\rightarrow +\infty} \frac{\inf_{T_n\leq t\leq T'_n} \mu(t)}{\sup_{T_n\leq t\leq T'_n} \mu(t)}=1.$$
This contradicts \eqref{contra.mu}. Hence
\begin{equation}
\label{mu.limit}
\lim_{t\rightarrow +\infty} \mu(t)=\mu_{\infty} \in (0,\infty).
\end{equation}
\subsubsection*{Step 3: proof of Proposition \ref{prop.CV0}.}
\label{sss.CVhdot}
We are now ready to prove \eqref{lim.expo}, which will complete the proof of  Proposition \ref{prop.CV0}. 
Let us first show that $\dd(u(t))$ tends exponentially to $0$. We first claim the following inequality
\begin{equation}
\label{ineg.int}
\exists C>0,\; \forall t\geq 0,\quad \int_t^{+\infty} \dd(u(\tau))d\tau \leq C\dd(u(t)).
\end{equation}
Indeed if \eqref{ineg.int} does not hold, there exists a sequence $T_n\rightarrow +\infty$ such that 
\begin{equation}
\label{contra.expo1}
\int_{T_n}^{+\infty} \dd(u(\tau))d\tau \geq n\dd(u(T_n)).
\end{equation}
By \eqref{mu.limit}, $\mu(t)$ is bounded from below. As usual, this implies that the parameter $\lambda(t)$ of Proposition \ref{compactness.u} is bounded from below. 
By Step $1$ of the proof, the assumptions of Claim \ref{claim.virial.seq} are fullfilled for the sequence $(u_k)_k$, with $k=(n,n')$, $n<n'$, and $u_k=u$, $\lambda_k(t)=\lambda(t)$, $t_{0k}=T_{n}$ and $t_{1k}=T_{n'}$. Hence
$$ \forall n,n',\; n<n', \quad \int_{T_{n}}^{T_{n'}} \dd(u(t))dt \leq C\big[\dd(u(T_{n}))+\dd(u(T_{n'}))\big],$$
Thus $\int_{T_{n}}^{+\infty} \dd(u(t))dt \leq C\dd(u(T_{n}))$ which contradicts \eqref{contra.expo1}, showing \eqref{ineg.int}.\par

Now by \eqref{ineg.int} we have, for some constants $C,c>0$
$$ \int_t^{+\infty} \dd(u(\tau))d\tau \leq Ce^{-ct}.$$
Together with the estimate $|\alpha'(t)|\leq C\dd(u(t))$ of Lemma \ref{bound.modul}, we get
$$ |\alpha(t)|=\left|\int_t^{+\infty} \alpha'(\tau)d\tau\right| \leq C e^{-ct}.$$
Recalling that by Lemma \ref{bound.modul} $|\alpha(t)|\approx \dd(u(t))$, we get the bound on $\dd(u(t))$ in \eqref{lim.expo}.

Estimate \eqref{lim.expo} is then a straightforward consequence of the estimate $\|\tilde{u}(t)\|_{\hdot}+|\theta'(t)|+\left|\frac{\mu'(t)}{\mu(t)}\right|\leq C\mu^2(t)\dd(u(t))$ of Lemma \ref{bound.modul} and the boundedness of $\mu$.
The proof of Proposition \ref{prop.CV0} is complete.

\medskip

\noindent\emph{Proof of Corollary \ref{corol.incomp}.} We must show that there is no solution $u$ of \eqref{CP} satisfying \eqref{hyp.sub} and \eqref{non-scatter+-}. Let $u$ be such a solution. By Proposition \ref{prop.CV0} applied forward and backward, the set $\{u(t),\; t\in\RR\}$, is relatively compact in $\hdot$. Furthermore
$$ \lim_{t\rightarrow+\infty} \dd(u(t))=\lim_{t\rightarrow -\infty} \dd(u(t))=0.$$
By Claim \ref{claim.virial.seq} with $u(t)=u_n(t)$, $t_{0n}=-n$, $t_{1n}=n$ and  $\lambda_n(t)=1$, we have
$ \int_{-\infty}^{+\infty}\dd(u(t))dt=\lim_{n\rightarrow +\infty} \int_{-n}^{+n} \dd(u(t)) dt=0.$
Thus $\dd(u_0)=0$ which contradicts \eqref{hyp.sub}. Corollary \ref{corol.incomp} is proven.\qed

\section{Convergence to $W$ in the supercritical case}
\label{sec.super}
In this section we consider a solution of \eqref{CP} with initial condition $u_{\restriction t=0}=u_0$ and such that
\begin{equation}
\label{supercrit}
E(u_0)=E(W),\quad \|u_0\|_{\hdot}>\|W\|_{\hdot}.
\end{equation}
\begin{prop}
\label{PropSuper}
Let $u$ be a radial solution of \eqref{CP} satisfying \eqref{supercrit} and defined on $[0,+\infty)$. Assume furthermore that $u_0\in L^2(\RR^N)$. Then there exist constants $\theta_0\in\RR$, $\mu_0,\,c,\,C>0$ such that
\begin{equation}
\label{conclusiontheo}
\forall t\geq 0,\quad \|u(t)-W_{[\theta_0,\mu_0]}\|_{\hdot}\leq Ce^{-ct}.
\end{equation}
A similar result holds for negative times if $u$ satisfies \eqref{supercrit} and is defined on $(-\infty,0]$. 
\end{prop}
\begin{corol}
\label{corol.sur}
Let $u$ be a radial solution of \eqref{CP} satisfying \eqref{supercrit} and such that $u_0\in L^2(\RR^N)$. Then $u$ is not defined on $\RR$. 
\end{corol}

The proof relies again on the localized virial argument. Consider a radial function $\varphi$ in $C^{\infty}_0(\RR^N)$ such that
\begin{gather}
\label{condphi}
\varphi(r)=r^2,\;r\leq 1,\quad \varphi(r)\geq 0\text{ and }\frac{d^2\varphi}{dr^2}(r)\leq 2,\; r\geq 0.
\end{gather}
Consider the function $G_{R}$ of Subsection \ref{sub.mean}
$$  G_{R}(t):=2\im \int_{\RR^N} \ubar(t) \nabla u(t) \cdot\nabla \varphi_{R}=H_R'(t),\quad
H_R(t):=\int_{\RR^N}|u(t)|^2\varphi_R. $$
where $\varphi_{R}(x)=R^2\varphi\left(\frac xR\right)$, 

As usual, the key point of the proof is to bound $G_{R}$ and $G_{R}'$. 
\begin{claim}
\label{ClaimSuperViriel}
Under the assumptions of Proposition \ref{PropSuper}, there exist constants $C,R_0>0$ (depending only on $\int|u_0|^2$), such that for $R\geq R_0$, and all $t\geq 0$
\begin{align}
\label{GRsuper}
G_{R}(t)&\leq CR^2\dd(u(t)),\\
\label{GR'super}
G_{R}'(t)&\leq -\frac{8}{N-2}\dd(u(t)).
\end{align}
\end{claim}
Let us show that Claim \ref{ClaimSuperViriel} implies Proposition \ref{PropSuper}. 
\begin{proof}[Proof of Proposition \ref{PropSuper}]

\noindent\emph{Step 1: exponential convergence of $\dd(u(t))$.}
Let us prove 
\begin{equation}
\label{exposuper}
\exists c,C>0,\;\forall t\geq 0,\quad \dd(u(t))\leq Ce^{-ct}.
\end{equation}
Fix $R\geq R_0$. We first remark
\begin{equation}
\label{GRpositive}
\forall t\geq 0,\quad G_{R}(t)>0.
\end{equation}
Indeed by \eqref{GR'super}, $G_{R}$ is strictly decreasing with time, so that if $G_{R}(t_0)\leq 0$ for some $t_0\geq 0$, then
$$ \forall t\geq t_0+1,\quad H'_R(t)= G_{R}(t)\leq G_R(t_0+1)<0.$$ 
This contradicts the fact that $\varphi_R$ is positive and $u$ defined on $[0,+\infty)$, proving \eqref{GRpositive}.\par
Consider two positive times $t<T$. Integrating \eqref{GR'super} between $t$ and $T$, and using \eqref{GRsuper}, we get
\begin{equation}
\label{major}
\frac{8}{N-2} \int_t^T \dd(u(s))ds\leq G_{R}(t)-G_{R}(T)\leq G_{R}(t)\leq CR^2\dd(u(t)).
\end{equation}
Letting $T$ tends to infinity yields,
$\int_t^{+\infty} \dd(u(s))ds\leq C\dd(u(t))$, for some $C>0$ and thus, by Gronwall Lemma
\begin{equation}
\label{inegint.super} 
\int_t^{+\infty} \dd(u(s))ds\leq Ce^{-ct}.
\end{equation}
Our next claim is that
\begin{equation}
\label{superCV}
\lim_{t\rightarrow +\infty} \dd(u(t))=0.
\end{equation}
Indeed, by \eqref{inegint.super}, there exists $t_n\rightarrow +\infty$ such that $\dd(u(t_n))\rightarrow 0$. Assume that \eqref{superCV} does not hold. Then, extracting a subsequence from $(t_n)$, there exists $t_n'>t_n$ such that 
$$\dd(u(t_n'))=\delta_0,\text{ and }\forall t\in(t_n,t_n'),\; 0<\dd(u(t))<\delta_0,$$
where $\delta_0$ is such that \eqref{decompo.v} and Lemma \ref{bound.modul} hold.
Consider the parameter $\alpha$ of decomposition \eqref{decompo.v}. By Lemma \ref{bound.modul}, $|\alpha'(t)|\leq C\dd(u(t))$, for $t\in[t_n,t_n']$ thus \eqref{inegint.super} implies that $\alpha(t_n)-\alpha(t_n')$ tends to $0$. Furthermore, again by Lemma \ref{bound.modul}, $|\alpha(t)|\approx\dd(u(t))$, which shows that $\dd(u(t_n'))$ tends to $0$, contradicting the definition of $t_n'$. Hence \eqref{superCV}.\par
By \eqref{superCV}, the parameter $\alpha(t)$ is well defined for large $t$. In view of the estimates $|\alpha'(t)|\leq C\dd(u(t))$ and $|\alpha(t)|\approx\dd(u(t))$, \eqref{inegint.super} yields \eqref{exposuper}.

\medskip

\noindent\emph{Step 2: convergence of $\mu(t)$ and end of the proof.}
Let us prove
\begin{equation}
\label{musuper}
\lim_{t\rightarrow +\infty} \mu(t)=\mu_{\infty} \in(0,\infty).
\end{equation}
By \eqref{exposuper} and the estimate $\left|\frac{\mu'(t)}{\mu^3(t)}\right|\leq C\dd(u(t))$ of Proposition \ref{bound.modul}, we know that $\frac{1}{\mu^2(t)}$ satisfies the Cauchy criterion of convergence. This shows that
$\lim_{t\rightarrow +\infty} \mu(t)=\mu_{\infty}\in (0,+\infty]$. It remains to show that $\mu_{\infty}$ is finite.\par
 Assume that $\mu_{\infty}=+\infty$. As $u_{[\theta(t),\mu(t)]}$ tends to $W$ in $\hdot$, it implies that for any $\eps>0$, $\int_{|x|\geq \eps} |u|^{2^*}$ tends to $0$ as $t$ tends to $\infty$. By H\"older  inequality and the boundedness of $\int |\nabla u(t)|^2$, this shows that $\lim_{t\rightarrow +\infty} H_R(t)=0$. Since by \eqref{GRpositive}, $H'_R(t)=G_R(t)>0$, this implies that $H_R(t)<0$ for $t\geq 0$ which contradicts the fact that $\varphi_R$ is positive. Hence \eqref{musuper}.\par

In particular, $\mu$ is bounded. Thus by Lemma \ref{bound.modul}, $$\|u-W_{[\theta(t),\mu(t)]}\|_{\hdot}+\left|\mu'(t)\right|+|\theta'(t)|\leq C\dd(u(t))\leq Ce^{-ct},$$
 which shows the Proposition.
\end{proof}
\begin{proof}[Proof of Corollary \ref{corol.sur}.]
Let $u$ be a solution of \eqref{CP} satisfying the assumptions of the corollary and defined on $\RR$.  Then by Proposition \ref{PropSuper} 
\begin{equation}
\label{dd.sur.0}
\lim_{t\rightarrow \pm \infty} \dd(u(t))=0.
\end{equation}
Define $G_R(t)$ as in the proof of Proposition \ref{PropSuper}. Applying Claim \ref{ClaimSuperViriel} to $t\mapsto \overline{u}(-t)$, we see that it holds also for negative times. By \eqref{GR'super}, $G_R'(t)<0$ and by \eqref{GRsuper} and \eqref{dd.sur.0} $G_R(t)\rightarrow 0$ for $t\rightarrow \pm\infty$. This is a contradiction, yielding the corollary.
\end{proof}

\begin{proof}[Proof of Claim \ref{ClaimSuperViriel}]
As in the proof of Lemma \ref{lem.delta.mean}, since $E(u_0)=E(W)$ and $\|u_0\|_{\hdot}>\|W\|_{\hdot}$
\begin{equation}
\label{GR'bis}
G_{R}'(t)= 8\left( \int |\nabla u|^2-\int |u|^{2*}\right) +A_{R}(u(t))=-\frac{16}{N-2} \dd(u(t))+A_{R}(u(t)),
\end{equation}
where $A_{R}$ is defined in \eqref{defAR}. 

\medskip

\noindent\emph{Step 1: a general bound on $A_R$.} We show that there exist $C_1,R_1>0$ (depending only on $\int |u_0|^2$) such that 
\begin{equation}
\label{MajorAR}
\forall R\geq R_1,\; \forall t\geq 0,\quad A_R(u(t))\leq \left\{\frac{C_1}{R^2}+\frac{C_1}{R^{\frac{2N-2}{N-2}}} \|u(t)\|_{\hdot}^{\frac{2}{N-2}}\right\}.
\end{equation}
Indeed, according to \eqref{defAR}, the definition of $\varphi_R$ and \eqref{condphi},
\begin{align*}
A_R(u(t))&=\int_{|x|\geq R} |\nabla u(t)|^2\Big(4\frac{d^2\varphi_R}{dr^2}-8\Big)r^{N-1}dr\\
&\qquad +\int_{|x|\geq R}|u(t)|^{2^*}\left(-\frac{4}{N}\Delta\varphi_R+8\right)r^{N-1}dr -\int |u(t)|^2(\Delta^2\varphi_R)r^{N-1}dr\\
&\leq C\int_{|x|\geq R} |u(t)|^{2*}dx+\frac{C}{R^2} \|u(t)\|_{L^2}^2.
\end{align*}
To bound the first term, we will use the radiality of $u(t)$ and Strauss Lemma \cite{St77}:
\begin{lemma}
\label{Strauss}
There is a constant $C>0$ such that for any radial function $f$ in $H^1(\RR^N)$
$$ \forall x,\;|x|\geq 1,\quad |f(x)|\leq \frac{C}{|x|^{(N-1)/2}} \|f\|_{L^2}^{1/2}\|f\|_{\hdot}^{1/2}.$$
\end{lemma}
 We have $\int_{|x|\geq R} |u(t)|^{2^*}\leq \|u(t)\|_{L^{\infty}(\{|x|\geq R\})}^{\frac{4}{N-2}}\|u(t)\|_{L^2}^2$, and thus, by Lemma \ref{Strauss},
\begin{equation}
\label{inegL6L2}
\int_{|x|\geq R} |u(t)|^{2^*}dx\leq \frac{C}{R^{\frac{2N-2}{N-2}}} \|u(t)\|_{\hdot}^{\frac{2}{N-2}}\|u(t)\|_{L^2}^{\frac{2N-2}{N-2}}
\end{equation}
which concludes the proof of \eqref{MajorAR} by the conservation of the $L^2$-norm.

\medskip

\noindent\emph{Step 2: estimate on $A_R$ when $\dd(u(t))$ is small.}
Let us show that there exists $\delta_2, R_2,C_2>0$ (depending only on $\int|u_0|^2_{L^2}$) such that
\begin{equation}
\label{MajorAR2}
\forall t\geq 0,\; \forall R\geq R_2,\quad 
\dd(u(t))\leq \delta_2 \Longrightarrow |A_{R}(u(t))|\leq C_2\left(\frac{1}{R^{\frac{N-2}{2}}}\dd(u(t))+\dd(u(t))^2\right).
\end{equation}
Taking a small $\delta_2$, we write by \eqref{decompo.v}, $u_{[\theta(t),\mu(t)]}=W+v$, with $\|v\|_{\hdot}\leq C\dd(u(t))$. In view of the bound \eqref{ARforlater} of $A_R$ shown in the preceding section, it is sufficient to prove
\begin{equation}
\label{m>0}
\mu_-:=\inf\{ \mu(t),\;t\geq 0,\;\dd(u(t))\leq \delta_2\}>0.
\end{equation}
Inequality \eqref{m>0} follows again from the fact that $u_0$ is in $L^2$. 
Indeed, if $\dd(u(t))\leq \delta_2$ we have, from the conservation of the $L^2$ norm, the equality $u_{[\theta(t),\mu(t)]}=W+V$ and Lemma \ref{bound.modul}
\begin{equation}
\label{contra.musuper}
\|u_0\|_{L^2}^2\geq \int_{|x|\leq\mu(t)} |u(t)|^2=\frac{1}{\mu(t)^2}\int_{|x|\leq 1} \left|u_{[\mu(t)]}(t)\right|^2\geq \frac{1}{\mu(t)^2}\left[\Big(\int_{|x|\leq 1} W^2\Big)-C\delta_2^2\right].
\end{equation}
If $\delta_2$ is small enough, this shows \eqref{m>0}, and thus the announced inequality \eqref{MajorAR2}.

\medskip

\noindent\emph{Step 3: conclusion of the proof.} 
In view of \eqref{GR'bis}, it is sufficient to prove
\begin{equation}
\label{major.AR.sur}
\exists R_0>0, \; \forall R\geq R_0,\;\forall t\geq 0,\quad |A_{R}(u(t))|\leq \frac{8}{N-2}\dd(u(t)).
\end{equation}
From \eqref{MajorAR2}, there exist $\delta_3$, $R_3$ such that \eqref{major.AR.sur} holds if $R\geq R_3$ and $\dd(u(t))\leq \delta_3$. Let $R_4>0$ and
$$ \Phi_{R_4}(\delta):=\frac{C_1}{R_4^2}+\frac{C_1}{R_4^{\frac{2N-2}{N-2}}}\left(\delta+\|W\|^2_{\hdot}\right)^{\frac{1}{N-2}}-\frac{8}{N-2}\delta,$$
where $C_1$ is given by step 1. Clearly, $\Phi_{R_4}$ is concave. Chose $R_4\geq R_2$ large enough so that $\Phi_{R_4}(\delta_3)\leq 0$, $\Phi_{R_4}'(\delta_3)\leq 0$. Then $\Phi_{R_4}(\delta)\leq 0$ for all $\delta\geq \delta_3$. Thus \eqref{MajorAR} implies \eqref{major.AR.sur} when $R\geq R_4$ and $\delta\geq \delta_3$, which concludes the proof of \eqref{major.AR.sur} with $R_0:=\max\{R_3,R_4\}$. The proof of Claim \ref{ClaimSuperViriel} is complete.
\end{proof}

\section{Preliminaries on the linearized equation around $W$}
\label{sec.linear}
By Propositions \ref{prop.CV0} and \ref{PropSuper}, in order to conclude the proofs of Theorems \ref{th.exist} and \ref{th.classif}, we need to study solutions $u$ of \eqref{CP} on $[t_0,+\infty)$, ($t_0\geq 0$)  such that
\begin{equation}
\label{devtu1}
\|u(t)-W\|_{\hdot}\leq Ce^{-\gamma_0 t},\quad E(u)=E(W)
\end{equation}
for some $\gamma_0>0$. 

We will write indifferently $f=f_1+if_2$ or $f=\begin{pmatrix} f_1\\f_2\end{pmatrix}$ for a complex valued function $f$ with real part $f_1$ and imaginary part $f_2$. For a solution $u$ of \eqref{CP} satisfying \eqref{devtu1}, we will write
$$ v(t):=u(t)-W.$$
Equation \eqref{CP} yields
\begin{gather}
\label{equation.v}
\partial_t v+ \LLL(v)+R(v)=0,\quad \LLL:=\begin{pmatrix} 0 & \Delta+W^{p_c-1} \\ -\Delta-p_cW^{p_c-1} & 0 \end{pmatrix},\\
\notag
R(v):=-i \left|W+v\right|^{p_c-1}(W+v)+iW^{p_c}+ip_cW^{p_c-1}v_1-W^{p_c-1}v_2.
\end{gather}

The proofs of our theorems in Section \ref{sec.proof} rely on a careful analysis of solutions of the linearized equation $\partial_t h+\LLL h=\eps$, with $h$ and $\eps$ exponentially small as $t\rightarrow +\infty$. Before this analysis, carried out in Subsection \ref{sub.key.anal}, we need to establish some spectral properties of $\LLL$ and Strichartz type estimates for the equation.

\subsection{Spectral theory for the linearized operator}
\label{sub.spectral}
We are interested here by real eigenvalues and other spectral properties of $\LLL$. Note that by direct calculation, 
\begin{equation}
\label{kernelL}
\LLL (iW)=\LLL(W_1)=0.
\end{equation}
\begin{lemma}
\label{lem.Y}
The operator $\LLL$ admits two eigenfunctions $\YYY_+,\,\YYY_-\in \SSS$ with real eigenvalues
\begin{equation}
\label{cond.Y}
\LLL \YYY_+=e_0\YYY_+,\quad \LLL \YYY_-=-e_0\YYY_-,\quad \YYY_+=\overline{\YYY}_-,\quad e_0\in (0,+\infty).
\end{equation}
\end{lemma}
See Appendix \ref{app.spectral} for the proof.

Consider the symmetric bilinear form $\BB$ on $\hdot$ such that $Q(f)=\BB(f,f)$, where $Q$ is the quadratic form of Subsection \ref{sub.decompo}
$$ \BB(f,g)=\frac{1}{2}\int \nabla f_1\cdot\nabla g_1-\frac{p_c}{2}\int f_1 g_1 W^{p_c-1}+\frac{1}{2}\int \nabla f_2\cdot\nabla g_2-\frac{1}{2}\int f_2 g_2 W^{p_c-1}=\frac{1}{2}\im \int (\LLL f)\gbar.$$
As a consequence of the definition of $\BB$ we have,
\begin{gather}
\label{BBeasy}
\BB(f,g)=\BB(g,f),\quad \BB(iW,f)=\BB(W_1,f)=0,\quad \forall f,g\in \hdot\\
\label{antisym}
\BB(\LLL f,g)=-\BB(f,\LLL g), \quad \forall f,g\in \hdot, \; \LLL f,\LLL g \in \hdot\\
\label{QY=0}
Q(\YYY_+)=Q(\YYY_-)=0,\quad \BB(\YYY_-,\YYY_+)\neq 0.
\end{gather}
Indeed, the only assertion which is not direct is the fact that $\BB(\YYY_-,\YYY_+)\neq 0$. To prove it, one may argue by contradiction. If $\BB(\YYY_-,\YYY_+)$ was $0$, $B$ and $Q$ would be identically $0$ on $\vect\{W_1,iW,\YYY_-,\YYY_+\}$ which is of dimension $4$. But $Q$ is, by Claim \ref{coercivity}, positive definite on $H^{\bot}$, which is of codimension $3$, yielding a contradiction.\par
By \eqref{antisym}, $\LLL$ is antisymmetric for the bilinear form $\BB$. In the following lemma, we give a subspace $G_{\bot}$ of $\hdot$, related to the eigenfunctions of $\LLL$, in which $Q$ is positive definite. 
\begin{lemma}
\label{lem.positivity}
Let $\ds G_{\bot}=\left\{v\in \hdot,\quad (iW,v)_{\hdot}=(W_1,v)_{\hdot}=\BB(\YYY_+,v)=\BB(\YYY_-,v)=0\right\}.$
Then there exists $c>0$ such that
\begin{equation}
\label{positif}
\forall f\in G_{\bot},\quad Q(f)\geq c \|f\|_{\hdot}^2.
\end{equation}
\end{lemma}
Note that $G_{\bot}$ is not stable by $\LLL$. Lemma \ref{lem.positivity} implies the following characterization of the real spectrum of $\LLL$.
\begin{corol}
\label{cor.spectrum}
Let $\Sp(\LLL)$ be the spectrum of the operator $\LLL$ on $L^2$ of domain $D(\LLL)=H^2$. Then
$$ \Sp(\LLL)\cap \RR=\{-e_0,0,e_0\}.$$
\end{corol}
\begin{proof}[Proof of the corollary]
By Lemma \ref{lem.Y}, $\{-e_0,e_0\}\subset \Sp(\LLL)$. Furthermore, the operator $\LLL$ is a compact perturbation of $\begin{pmatrix} 0 & \Delta \\-\Delta & 0 \end{pmatrix}$, thus its essential spectrum is $i\RR$. Consequently, $0\in \Sp(\LLL)$, and $\Sp(\LLL)\cap \RR^*$ contains only eigenvalues. It remains to show that $-e_0$ and $e_0$ are the only eigenvalues of $\LLL$ in $\RR^*$. Assume that for some $f\in H^2$
$$\LLL f=e_1 f,\quad e_1\in \RR \setminus \{-e_0,0,e_0\}.$$
We must show that $f=0$. By \eqref{antisym}, $(e_1+e_0)\BB(f,\YYY_+)=(e_1-e_0)\BB(f,\YYY_-)=0$ and thus
$$ \BB(f,\YYY_+)=\BB(f,\YYY_-)=0.$$
Write 
$$ f=\beta iW+\gamma W_1+g,\quad g\in G_{\bot}, \; \beta= \frac{(f,iW)_{\hdot}}{\|W\|^2_{\hdot}},\;\gamma= \frac{(f,W_1)_{\hdot}}{\|W_1\|^2_{\hdot}}.$$
Again by \eqref{antisym}, $\BB(f,f)=0$ and thus $B(g,g)=0$. This implies by Lemma \ref{lem.positivity} that $g=0$ and thus $e_1f=\LLL f=\beta \LLL(iW)+\gamma \LLL W_1=0$. Recalling that $e_1\neq 0$, we get as announced that $f=0$, which concludes the proof of Corollary \ref{cor.spectrum}
\end{proof}

\begin{proof}[Proof of Lemma \ref{lem.positivity}]
Recall from Claim \ref{coercivity} that there exists a constant $c_1$ such that
\begin{equation}
\label{coercif}
\forall g\in H^{\bot},\quad Q(g)\geq c_1 \|g\|_{\hdot}^2.
\end{equation}
Let $f\in G_{\bot}$. We will eventually deduce \eqref{positif} from \eqref{coercif}. Decompose $f$, $\YYY_+$ and $\YYY_-$ in the orthogonal sum $\hdot=H\oplus H^{\bot}$:
\begin{gather}
\label{directsum}
f=\alpha W+ \tilde{h},\quad
\YYY_+=\eta\,iW+\xi W_1+\zeta W+h_+,\quad \YYY_-=-\eta iW+\xi W_1+\zeta W+h_-,
\end{gather}
where $\tilde{h},h_+,h_-\in H^{\bot}$, $h_-=\overline{h}_+$. 

\medskip

\noindent\emph{Step 1.} We first show
\begin{equation}
\label{Q(f)}
Q(f)=-\frac{\BB(h_{+},\tilde{h})\BB(h_{-},\tilde{h})}{\sqrt{Q(h_+)}\sqrt{Q(h_-)}}+Q(\tilde{h}).
\end{equation}
Note that if $h\in H^{\bot}$, $\BB(W,h)=\frac 12 \int \nabla W\cdot \nabla h_{1}-\frac{p_c}{2}\int W^{p_c} h_{1}=\frac {1-p_c}{2} \int \nabla W\cdot \nabla h_{1}=0$. By \eqref{BBeasy}, \eqref{QY=0} and \eqref{directsum}, we have
\begin{equation}
\label{coeff+.1}
0=Q(\YYY_+)=\zeta^2Q(W)+Q(h_+),\qquad 0=Q(\YYY_-)=\zeta^2Q(W)+Q(h_-).
\end{equation}
Furthermore, developping the equalities $\BB(f,\YYY_+)=\BB(f,\YYY_-)=0$ with \eqref{directsum} we get
\begin{equation}
\label{coeff+.2}
\alpha\zeta Q(W)+\BB(\tilde{h},h_+)=\alpha\zeta Q(W)+\BB(\tilde{h},h_-)=0.
\end{equation}
By \eqref{directsum}, $Q(f)=\alpha^2Q(W)+Q(\tilde{h})$, and \eqref{Q(f)} follows from \eqref{coeff+.1} and \eqref{coeff+.2}.

\medskip

\noindent\emph{Step 2}. We next prove the following assertion:
\begin{quote} 
\emph{The functions $h_+$ and $h_-$ are independent in the real Hilbert space $\hdot$.}
\end{quote}

Note that $h_+=\overline{h}_-$. Thus it is sufficient to show
\begin{equation}
\label{h.nozero}
h_1:=\re h\neq 0\text{ and }h_2:=\im h \neq 0.
\end{equation}
Write $\YYY_1=\re \YYY_+$, $\YYY_2=\im \YYY_+$. Then \eqref{cond.Y} writes down
\begin{equation}
\label{cond.Y'}
\big(\Delta+W^{p_c-1}\big)\YYY_2=e_0 \YYY_1,\quad \big(\Delta+p_cW^{p_c-1}\big)\YYY_1=-e_0\YYY_2.
\end{equation}
We show \eqref{h.nozero} by contradiction. First assume that $h_2=0$. Then by \eqref{directsum}, $\YYY_2$ is in $\vect(W)$, so that $(\Delta +W^{p_c-1})\YYY_2=0$. By \eqref{cond.Y'}, we get that $\YYY_1=0$ and $\YYY_2=0$, which contradicts the definition of $\YYY_+$.\par
Similarly, assuming that $h_1=0$, and recalling that $(\Delta+p_cW^{p_c-1})W_1=0$, we get by \eqref{directsum} and \eqref{cond.Y'} that
$\YYY_2=-\frac{\zeta(p_c-1)}{e_0} W^{p_c}$, and thus $\YYY_1=-\frac{\zeta(p_c-1)}{e_0^2}(\Delta+W^{p_c-1}) W^{p_c}$.
Now, $\YYY_1=\xi W_1+\zeta W$. This implies that $(\Delta+W^{p_c-1}) W^{p_c}$ is in $\vect\{W,W_1\}$, which is not the case as a direct computation shows. Hence \eqref{h.nozero} which concludes this step of the proof.

\medskip

\noindent\emph{Step 3: conclusion of the proof.} 
Recall that $Q$ is positive definite on $H^{\bot}$. We claim that there is a constant $b<1$, such that
\begin{equation}
\label{CauchySchwarz}
\forall X\in H^{\bot},\quad \left|\frac{\BB(h_+,X)\BB(h_-,X)}{\sqrt{Q(h_+)}\sqrt{Q(h_-)}}\right|\leq bQ(X).
\end{equation}
Indeed it is equivalent to show, by orthogonal decomposition on $H^{\bot}$ related to $B$
\begin{equation}
\label{def.max}
b:= \max_{\substack{ X\in \vect\{h_-,h_+\}\\ X\neq 0}} \left(\frac{\BB(h_+,X)} {\sqrt{Q(h_+)Q(X)}}\right)\left(\frac{\BB(h_-,X)}{\sqrt{Q(h_-)Q(X)}}\right)<1.
\end{equation}
Applying twice Cauchy-Schwarz inequality with $B$, we get $b\leq 1$. Furthermore, if $b=1$, there exists $X\neq 0$ such that the two Cauchy-Schwarz inequalities are equalities and thus $X\in\vect\{h_+\}\cap\vect\{h_-\}=\{0\}$, which is a contradiction, showing \eqref{CauchySchwarz}.\par 
By \eqref{coercif}, \eqref{Q(f)} and \eqref{CauchySchwarz}
\begin{equation}
\label{Qgeq}
Q(f)\geq (1-b)Q(\tilde{h})\geq c_1(1-b) \|\tilde{h}\|_{\hdot}^2.
\end{equation}
Noting that by \eqref{directsum}, $Q(f)=\alpha^2Q(W)+Q(\tilde{h})$, and recalling that $Q(W)<0$, we also get by the first inequality in \eqref{Qgeq} that $b Q(\tilde{h})\geq \alpha^2|Q(W)|$. Hence
$$ C Q(f)\geq \alpha^2\|W\|^2_{\hdot}+\|\tilde{h}\|_{\hdot}^2=\|f\|_{\hdot}^2,$$
which concludes the proof of Lemma \ref{lem.positivity}.
\end{proof}

\subsection{Preliminary estimates}
\label{sub.prelim}
In this subsection we gather some elementary estimates needed in the sequel. We start with bounds on the potential and nonlinear terms of equation \eqref{equation.v}. Let 
\begin{equation}
\label{def.VVV} 
\VVV(v):=W^{p_c-1} \re v+pW^{p_c-1} \im v,
\end{equation}
so that equation \eqref{equation.v} writes as a Schr\"odinger equation
\begin{equation}
\label{Schrodinger.v}
i\partial_t v+\Delta v+\VVV(v)+iR(v)=0.
\end{equation}
We start to recall standard Strichartz estimates for the free Schr\"odinger equation (see \cite{St77a,GiVe85,KeTa98}).
\begin{lemma}
Assume $N\geq 3$ and, for $j=1,2$, let $(p_j,q_j)$ such that $(\frac{2}{p_j}+\frac{N}{q_j}=\frac{N}{2},\; p_j\geq 2)$. Denote by $p'_j$ and $q'_j$ the dual conjugate exponents of $p_j$ and $q_j$. Then
\begin{align}
\label{StriFree}
\|e^{it\Delta}u_0\|_{L^{p_1}(\RR,L^{q_1})}&\leq C\|u_0\|_{L^2}\\
\label{StriFreeRM}
\left\|\int_{t}^{+\infty}e^{i(s-t)\Delta}f(s)ds\right\|_{L^{p_1}(\RR,L^{q_1})}&\leq C\|f\|_{L^{p_2'}(\RR,L^{q_2'})}\\
\label{StriFreeDual}
\left\|\int_{-\infty}^{+\infty} e^{is\Delta}f(s)ds\right\|_{L^2}&\leq C \|f\|_{L^{p_2'}(\RR,L^{q_2'})}.
\end{align}
\end{lemma}
If $I$ is a time interval, we will be interested in the spaces $S(I)$ and $Z(I)$ defined in \eqref{defSZ} as well as $N(I):=L^2\Big(I;L^{\frac{2N}{N+2}}\Big)$, which is the dual of the endpoint Strichartz space $L^2\big(I;L^{2^*}\big)$. H\"older and Sobolev inequalities yield immediately:
\begin{lemma}[Linear estimates]
\label{lem.V}
Let $f\in L^{2^*}\left(\RR^N\right)$. Then
\begin{equation}
\label{estim.V1}
\|\VVV(f)\|_{L^{\frac{2N}{N+2}}}\leq C\|f\|_{L^{2^*}}.
\end{equation}
Let $I$ be a finite time interval of length $|I|$ and $f\in S(I)$ such that $\nabla f\in Z(I)$. Then, there exists $C$ independent of $I$, $f$ and $g$ such that
\begin{gather}
\label{Sob}
  \|f\|_{S(I)}\leq C \|\nabla f\|_{Z(I)}\\
\label{estim.V2}
\|\nabla \VVV(f)\|_{N(I)}\leq |I|^{\frac{N}{N+2}} \|\nabla f\|_{Z(I)}.
\end{gather}
\end{lemma}
The proof of following lemma, given in the appendix, is classical.
\begin{lemma}[Non-linear estimates]
\label{lem.R}
Let $f,g$ be functions in $L^{2^*}(\RR^N)$. Then
\begin{equation}
\label{estim.dR1}
\|R(f)-R(g)\|_{L^{\frac{2N}{N+2}}}\leq C\|f-g\|_{L^{2^*}}\left(\|f\|_{L^{2^*}}+\|g\|_{L^{2^*}}+\|f\|_{L^{2^*}}^{p_c-1}+\|g\|_{L^{2^*}}^{p_c-1}\right).
\end{equation}
Let $I$ be a finite time interval and $f,g$ be functions in $S(I)$, such that $\nabla f$ and $\nabla g$ are in $Z(I)$. Then
\begin{multline}
\label{estim.dR2}
\|\nabla R(f)-\nabla R(g)\|_{N(I)}\leq\\ 
C\left\|\nabla f-\nabla g\right\|_{Z(I)}\left[|I|^{\frac{6-N}{2(N+2)}}\big(\|\nabla f\|_{Z(I)}+\|\nabla g\|_{Z(I)}\big)+\|\nabla f\|_{Z(I)}^{p_c-1}+\|\nabla g\|_{Z(I)}^{p_c-1}\right].
\end{multline}
\end{lemma}
We finish this subsection by showing Strichartz estimates on exponentially small solutions $v$ of \eqref{equation.v}.
\begin{lemma}[Strichartz estimates]
\label{lem.Strichartz.v}
Let $v$ be a solution of \eqref{equation.v}. Assume for some $c_0>0$,
\begin{equation}
\label{decvH1}
\exists C>0,\quad \|v(t)\|_{\hdot}\leq C e^{-c_0 t}.
\end{equation}
Then, for any Strichartz couple $(p,q)$ $(\frac{2}{p}+\frac{N}{q}=\frac{N}{2},\; p\geq 2)$
\begin{equation}
\label{Strichartz.v}
\exists C>0,\quad \|v\|_{S\left(t,+\infty\right)}
+\|\nabla v\|_{L^{p}(t,+\infty;L^q)}\leq C e^{-c_0t}.
\end{equation}
\end{lemma}
\begin{proof}
We will first estimate $\|v\|_{S(t,+\infty)}+\|\nabla v\|_{Z(t,+\infty)}$. According to the following claim, we juste need to estimate 
$\|v\|_{S(t,t+\tau_0)}$ and $\|\nabla v\|_{Z(t,t+\tau_0)}$ for some small $\tau_0>0$. 
\begin{claim}[Sums of exponential]
\label{summation}
Let $t_0>0$, $p\in [1,+\infty[$, $a_0 \neq 0$, $E$ a normed vector space, and $f\in L^p_{\rm loc}(t_0,+\infty;E)$ such that 
\begin{equation}
\label{small.tau}
\exists \tau_0>0,\; \exists C_0>0, \;\forall t\geq t_0, \quad \|f\|_{L^p(t,t+\tau_0,E)}\leq C_0e^{a_0 t}.
\end{equation}
Then for $t\geq t_0$,
\begin{equation}
\label{conclu.summation}
\|f\|_{L^p(t,+\infty,E)}\leq \frac{C_0e^{a_0 t}}{1-e^{a_0\tau_0}}\; \text{if } a_0<0;\quad \|f\|_{L^p(t_0,t,E)}\leq  \frac{C_0e^{a_0 t}}{1-e^{-a_0\tau_0}}\; \text{if }a_0>0.
\end{equation}
\end{claim}
\begin{proof}
Assume $a_0<0$. Summing up \eqref{small.tau} at time $t=t_0$, $t=t_0+\tau_0$, $t=t_0+2\tau_0$, {\ldots}, and using the triangle inequality, we get \eqref{conclu.summation}. The case $a_0>0$ is analogue.
\end{proof}
By \eqref{Schrodinger.v}, 
\begin{equation}
\label{eq.nabla}
i\partial_t \nabla v+\Delta(\nabla v)+\nabla\big(\VVV(v)+iR (v)\big)=0.
\end{equation}
Let $t$ and $\tau$ such that $0<t$, $0<\tau<1$. By Strichartz inequalities \eqref{StriFree} and \eqref{StriFreeRM}, and equation \eqref{eq.nabla}
$$ \|\nabla v\|_{Z(t,t+\tau)}\leq C\left( \|v(t)\|_{\hdot}+\left\|\nabla (\VVV(v)+R(v))\right\|_{N(t,t+\tau)} \right).$$
Thus by Lemmas \ref{lem.V} and \ref{lem.R}
\begin{equation*}
\|\nabla v\|_{Z(t,t+\tau)} \leq C\left( \|v(t)\|_{\hdot}+{\tau}^{\frac{N}{N+2}}\|\nabla v\|_{Z(t,t+\tau)}+{\tau}^{\frac{6-N}{2(N+2)}}\|\nabla v\|^2_{Z(t,t+\tau)}+\|\nabla v\|_{Z(t,t+\tau)}^{p_c}\right).
\end{equation*}
Using assumption \eqref{decvH1}, we get, for some constants $K>0$ and $\alpha_N>0$
\begin{equation}
\label{crucial}
\|\nabla v\|_{Z(t,t+\tau)}\leq K\left\{ e^{-c_0 t}+\tau^{\alpha_N}\|\nabla v\|_{Z(t,t+\tau)}+\|\nabla v\|^{p_c}_{Z(t,t+\tau)}\right\}.
\end{equation}
We claim that it implies that for large $t$
\begin{equation}
\label{major.nabla}
\|\nabla v\|_{Z(t,t+\tau_0)}\leq 2K e^{-c_0t},\quad \tau_0:=\frac{1}{(3K)^{1/\alpha_N}}.
\end{equation}
Indeed, fix $t>0$. Then \eqref{crucial} implies $\|\nabla v\|_{Z(t,t+\tau)}<2K e^{-c_0t}$ for small $\tau$. If \eqref{major.nabla} does not hold, then there exists $\tau\in (0,\tau_0]$ such that $\|\nabla v\|_{Z(t,t+\tau)}=2K e^{-c_0t}$, contradicting \eqref{crucial} if $t$ is large. Hence \eqref{major.nabla}.\par
By Claim \ref{summation} and Sobolev inequality \eqref{Sob}
\begin{equation}
\label{aboundofv}
\|v\|_{S(t,+\infty)} +\|\nabla v\|_{Z(t,+\infty)} \leq C e^{-c_0t}.
\end{equation}
Now take any Strichartz couple $(p,q)$. Then by \eqref{eq.nabla}, Strichartz estimates \eqref{StriFree} and \eqref{StriFreeRM} and Lemmas \ref{lem.V} and \ref{lem.R},
$\|\nabla v\|_{L^{p}(t,t+1;L^q)}\leq C \left( \|v(t)\|_{\hdot}+\|\nabla v\|_{Z(t,t+1)}+\|\nabla v\|^{p_c}_{Z(t,t+1)}\right).$
Thus \eqref{aboundofv} implies the bound $\|\nabla v\|_{L^{p}\left(t,t+1;L^q\right)}\leq C e^{-c_0t}$, which concludes, by Claim \ref{summation}, the proof of the lemma.
\end{proof}

\subsection{Estimates on exponential solutions of the linearized equation}
\label{sub.key.anal}
Let us consider the linearized equation with right-member
\begin{equation}
\label{eq.h}
\partial_t h+\LLL h=\eps
\end{equation}
with $h$ and $\eps$ such that for $t\geq 0$, 
\begin{gather}
\label{cond.eps}
\|\nabla \eps\|_{N(t,+\infty)}+\|\eps(t)\|_{L^{\frac{2N}{N+2}}}\leq Ce^{-c_1 t},\\
\label{cond.h}
\|h(t)\|_{\hdot}\leq C e^{-c_0 t},
\end{gather}
where $0<c_0<c_1$. The following proposition asserts that $h$ must decay almost as fast as $\eps$, except in the direction $\YYY_+$ where the decay is of order $e^{-e_0t}$. 
\begin{prop}
\label{prop.crucial.h}
Consider $h$ and $\eps$ satisfying \eqref{eq.h}, \eqref{cond.eps} and \eqref{cond.h}. 
Then, for any Strichartz couple $(p,q)$:
\begin{itemize}
\item if $c_0<c_1\leq e_0$ or $e_0<c_0<c_1$,
\begin{equation}
\label{decay.h}
\forall \eta>0,\quad \|h(t)\|_{\hdot}+\|\nabla h\|_{L^{p}(t,+\infty;L^q)}\leq C_{\eta} e^{-(c_1-\eta) t};
\end{equation}
\item if $c_0\leq e_0<c_1$, there exists $A_+\in \RR$ such that
\begin{equation}
\label{decay.h'}
\forall \eta>0,\quad\Big\|h(t)-A_+e^{-e_0t}\YYY_+\Big\|_{\hdot}+\Big\|\nabla(h-A_+e^{-e_0t}\YYY_+)\Big\|_{L^{p}(t,+\infty;L^q)}\leq C_{\eta} e^{-(c_1-\eta) t}.
\end{equation}
\end{itemize}
\end{prop}
\begin{proof}[Proof of Proposition \ref{prop.crucial.h}]
We will start by proving \eqref{decay.h}. In view of the following claim it is sufficient to prove only the bound of the $\hdot$-norm.
\begin{claim}
\label{claim.Strichartz.h}
Consider $h$ and $\eps$ fullfilling \eqref{eq.h}, \eqref{cond.eps} and \eqref{cond.h} with $0<c_0<c_1$. Then for any Strichartz couple $(p,q)$
\begin{equation}
\|\nabla h\|_{L^{p}(t,+\infty;L^q)}\leq C e^{-c_0 t}.
\end{equation}
\end{claim}
We will omit the proof, which is a simple consequence of Strichartz inequalities and of Lemma \ref{lem.V}, and is similar to the proof of Lemma \ref{lem.Strichartz.v}.\par

Let us decompose $h(t)$ as
\begin{equation}
\label{decompo.v0}
h(t)=\alpha_+(t)\YYY_++\alpha_-(t)\YYY_-+\beta(t)iW+\gamma(t)W_1+g(t), \quad g(t)\in G_{\bot},
\end{equation}
where (recall that by \eqref{BBeasy} and \eqref{QY=0}, $B(iW,\cdot)=B(W_1,\cdot)=0$ and $Q(\YYY_+)=Q(\YYY_-)=0$)
\begin{gather}
\label{def.alpha+-}
\alpha_-:=\frac{\BB(h,\YYY_+)}{\BB(\YYY_+,\YYY_-)},\quad \alpha_+:=\frac{\BB(h,\YYY_-)}{\BB(\YYY_+,\YYY_-)}
\\
\label{def.beta.gamma}
\beta:=\frac{1}{\|W\|^2_{\hdot}}\big(h-\alpha_+\YYY_+-\alpha_-\YYY_-,iW\big)_{\hdot},\quad
\gamma:=\frac{1}{\|W_1\|^2_{\hdot}}\Big(h-\alpha_+\YYY_+-\alpha_-\YYY_-,W_1\Big)_{\hdot}.
\end{gather}
In the sequel, we will assume without loss of generality
\begin{equation}
 \label{c1non0}
c_1\neq e_0.
\end{equation}
We divide the proof into four steps.

\medskip

\noindent\emph{Step 1: differential equations on the coefficients.}
We first claim
\begin{gather}
\label{eq.alpha+-}
\frac{d}{dt}\left(e^{e_0 t} \alpha_+\right)=  e^{e_0t} \frac {\BB(\YYY_-,\eps)}{\BB(\YYY_+,\YYY_-)},\quad
\frac{d}{dt}\left(e^{-e_0 t}\alpha_-\right)=e^{-e_0 t}\frac{\BB(\YYY_+,\eps)}{\BB(\YYY_+,\YYY_-)},\\
\label{eq.Q.beta.gamma}
\frac{d Q(h)}{dt}=2\BB(h,\eps),\quad
\frac{d\beta}{dt}=\frac{(iW,\teps)_{\hdot}}{\|W\|_{\hdot}^2},\quad \frac{d\gamma}{dt}=\frac{(W_1,\teps)_{\hdot}}{\|W_1\|_{\hdot}^2}.
\end{gather}
Where
\begin{equation}
\label{def.H}
\teps:=\eps-\frac{\BB(\YYY_-,\eps)}{\BB(\YYY_+,\YYY_-)}\YYY_+ -\frac{\BB(\YYY_+,\eps)}{\BB(\YYY_+,\YYY_-)}\YYY_- -\LLL g.
\end{equation}

By equation \eqref{eq.h}, 
$$ \BB(\YYY_-,\partial_t h)+\BB(\YYY_-,\LLL h)=\BB(\YYY_-,\eps).$$ 
Furthermore
$ \BB(\YYY_-,\partial_t h)=\frac{d}{dt} \BB(\YYY_-,h)$ and by \eqref{antisym}, $\BB(\YYY_-,\LLL h)=-\BB(\LLL \YYY_-,h)=e_0 \BB(\YYY_-,h)$. In view of \eqref{def.alpha+-}, we get the first equation in \eqref{eq.alpha+-}. A similar calculation yields the second equation.\par
By equation \eqref{eq.h}, $\BB(h,\partial_t h)+\BB(h,\LLL h)=\BB(h,\eps).$
Furthermore by \eqref{antisym}, $\BB(h,\LLL h)=0$ which yields the equation on $Q(h)$ in \eqref{eq.Q.beta.gamma}.\par
It remains to show the equations on $\beta$ and $\gamma$. Differentiating \eqref{def.beta.gamma}, we get
\begin{equation*}
\beta'(t)=\frac{1}{\|W\|_{\hdot}^2}(\widehat{\eps},iW)_{\hdot},\quad \gamma'(t)=\frac{1}{\|W_1\|_{\hdot}^2}(\widehat{\eps},W_1)_{\hdot},\quad \widehat{\eps}:= \eps-\LLL h-\alpha_+'\YYY_+-\alpha_-'\YYY_-.
\end{equation*}
Now, noting that by \eqref{decompo.v0}, $\LLL h=\alpha_+e_0\YYY_+-\alpha_-e_0\YYY_-+\LLL(g)$, and using \eqref{eq.alpha+-}, we obtain
$$ \widehat{\eps}=\eps-\frac{\BB(\YYY_-,\eps)}{\BB(\YYY_+,\YYY_-)}\YYY_+-\frac{\BB(\YYY_+,\eps)}{\BB(\YYY_+,\YYY_-)}\YYY_--\LLL(g)=\teps,$$
which yields the desired result.
\medskip

\noindent\emph{Step 2: bounds on $\alpha_-$ and $\alpha_+$.} We now claim
\begin{align}
\label{dec.alpha-}
|\alpha_-(t)|&\leq Ce^{-c_1 t}\\
\label{dec.alpha+}
\left|\alpha_+(t)\right|&\leq \left\{ 
\begin{array}{ll} 
\ds Ce^{-c_1t} &\text{ if }e_0<c_0,\\
\ds C\left(e^{-e_0t}+e^{-c_1 t}\right)&\text{ if } c_0\leq e_0.
\end{array}
\right.
\end{align}
Let us first show the following general bound on $\BB$.
\begin{claim}
\label{Stri.phi}
For any finite time-interval $I$, of length $|I|$, and any functions $f$ and $g$ such that $f\in L^{\infty}(I,L^{\frac{2N}{N+2}})$, $\nabla f\in N(I)$, $g\in L^{\infty}(I,L^{2^*})$ and $\nabla g \in L^{2}\big(I,L^{2^*}\big)$,
$$ \int_I |\BB(f(t),g(t))|dt \leq C\left[\|\nabla f\|_{N(I)}\|\nabla g\|_{L^2\big(I,L^{2^*}\big)}+|I|\,\|f\|_{L^{\infty}\big(I,L^{\frac{2N}{N+2}}\big)}\|g\|_{L^{\infty}\big(I,L^{2^*}\big)}\right].$$
\end{claim}
\begin{proof}
We have
\begin{align*}
2\BB(f(t),g(t))=a(t)+b(t),\text{where }a(t)&:=\int_{\RR^N} \nabla f_1(t)\nabla g_1(t)+\int_{\RR^N} \nabla f_2(t)\nabla g_2(t)\\
b(t)&:=-p_c\int_{\RR^N} W^{p_c-1}f_1(t)g_1(t)-\int_{\RR^N} W^{p_c-1} f_2(t)g_2(t).
\end{align*}
By  H\"older inequality 
$$\int_I |a(t)|dt\leq C\|\nabla f\|_{\vphantom{\big(}N(I)}\|\nabla g\|_{L^2\big(I,L^{2^*}\big)},\quad |b(t)|\leq C\|f(t)\|_{L^{\frac{2N}{N+2}}}\|g(t)\|_{L^{2^*}\vphantom{L^{\frac{2N}{N+2}}}}\|W^{p_c-1}\|_{L^{\infty}\vphantom{L^{\frac{2N}{N+2}}}}.$$
Integrating the estimate on $b(t)$ over $I$ and summing up, we get the conclusion of the claim.
\end{proof}
Assumption \eqref{cond.eps} on $\eps$, together with the preceding claim yields the inequality
$$\int_{t}^{t+1}|e^{-e_0 s} \BB(\YYY_+,\eps(s))| ds\leq C e^{-(e_0+c_1) t}.$$ 
By Claim \ref{summation}, $\int_{t}^{\infty}|e^{-e_0 s} \BB(\YYY_+,\eps(s))| ds\leq C e^{-(e_0+c_1) t}$.
Integrating the equation on $\alpha_-$ in \eqref{eq.alpha+-} between $t$ and $+\infty$, we get \eqref{dec.alpha-}.

Let us show \eqref{dec.alpha+}. First assume that  $c_0>e_0$. Thus by assumption \eqref{cond.h}, $e^{e_0 t} \alpha_+(t)$ tends to $0$ when $t$ tends to infinity. Furthermore $c_1>c_0>e_0$, and thus by assumption \eqref{cond.eps}, Claims \ref{Stri.phi} and \ref{summation}, $\int_{t}^{+\infty} \left|e^{e_0s}\BB(\YYY_-,\eps(s))\right|ds \leq C e^{(e_0-c_1)t}$.
Integrating between $t$ and $+\infty$  the equation on $\alpha_+$ in \eqref{eq.alpha+-} we get \eqref{dec.alpha+} if $c_0>e_0$.

Now assume that $c_0\leq e_0$. By \eqref{eq.alpha+-} 
\begin{equation*}
 \alpha_+(t)=e^{-e_0t} \alpha_+(0) +\frac{e^{-e_0t}}{\BB(\YYY_+,\YYY_-)} \underbrace{\int_0^t e^{e_0 s} \BB(\YYY_-,\eps(s))ds}_{(a)}.
\end{equation*}
If $c_0\leq e_0<c_1$, assumption \eqref{cond.eps} and Claim \ref{Stri.phi} imply that the integral $(a)$ is bounded, which shows \eqref{dec.alpha+} in this case.\par
It remains to show \eqref{dec.alpha+} when $c_0<c_1<e_0$. By \eqref{cond.eps}, Claim \ref{Stri.phi} and Claim \ref{summation}, $|(a)|\leq Ce^{(e_0-c_1)t}$, which yields again \eqref{dec.alpha+}. Step 2 is complete. 

\medskip

\noindent\emph{Step 3: bounds on $\|g\|_{\hdot}$, $\beta$ and $\gamma$.} We next prove
\begin{equation}
\label{dec.g.beta.gamma}
\big\|g\big(t\big)\big\|_{\hdot}+|\beta(t)|+|\gamma(t)|\leq C e^{-\left(\frac{c_0+c_1}{2}\right)t}.
\end{equation}
By Claims \ref{claim.Strichartz.h} and \ref{Stri.phi}, assumptions \eqref{cond.eps} and  \eqref{cond.h} yield $\int_{t}^{t+1} |\BB(h(s),\eps(s))|ds\leq Ce^{-(c_0+c_1)t}$.
Integrating the equation on $Q$ in \eqref{eq.Q.beta.gamma} between $t$ and $+\infty$ and using  Claim \ref{summation}, we get
\begin{equation}
\label{energy.h}
|Q(h(t))|\leq C e^{-(c_0+c_1)t}.
\end{equation}
Thus
\begin{gather*}
|Q(\alpha_+ \YYY_++\alpha_- \YYY_-+\beta iW+\gamma W_1+g)|\leq Ce^{-(c_0+c_1)t}\\
|2\alpha_+ \alpha_- \BB(\YYY_+,\YYY_-)+Q(g)|\leq C e^{-(c_0+c_1)t}.
\end{gather*}
By \eqref{dec.alpha-} and \eqref{dec.alpha+}
\begin{gather*}
|Q(g)|\leq \left\{ \begin{array}{ll}\ds C\left(e^{-(c_0+c_1)t}+e^{-2c_1t}\right) \leq Ce^{-(c_0+c_1)t}&\text{ if }c_0>e_0\\
\ds C\left(e^{-(c_0+c_1)t} +e^{-(e_0+c_1)t}+e^{-2c_1t}\right)\leq Ce^{-(c_0+c_1)t} &\text{ if } c_0\leq e_0.\end{array}\right.
\end{gather*}
As a consequence of the coercivity of $Q$ on $G_{\bot}$ (Lemma \ref{lem.positivity}), we get the estimate on $\|g\|_{\hdot}$ in \eqref{dec.g.beta.gamma}. It remains to show the bounds on $\beta$ and $\gamma$.\par

Consider the function $\teps$ defined in \eqref{def.H}. By assumption \eqref{cond.eps}
\begin{equation}
\label{scalar.product}
\int_t^{t+1} \left|(iW,\teps(s))_{\hdot}\right|ds\leq Ce^{-c_1 t}+\int_{t}^{t+1}\left|(W,\LLL g(s))_{\hdot}\right|ds.
\end{equation}
We have $ (W,\LLL g)_{\hdot}=\re \int -\Delta (W)\overline{\LLL g}=-\re \int \LLL^*(\Delta W) \overline{g}$,
where $\LLL^*$ is the $L^2$-adjoint of $\LLL$. Note that $\LLL^*(\Delta W)=\LLL^* W^{p_c}$ is in $L^{\frac{2N}{N+2}}$ (indeed by explicit computation, it is a $C^{\infty}$ function of order $\frac{1}{|x|^{2N+4}}$ at infinity). Thus, by the estimate on $\|g\|_{\hdot}$ in \eqref{dec.g.beta.gamma},
$$ \left|(\LLL g(t),W)_{\hdot}\right|\leq C \|g(t)\|_{L^{2^*}}\leq C e^{-\left(\frac{c_0+c_1}{2}\right)t}.$$
In view of Claim \ref{summation} and \eqref{scalar.product}, we get the bound on $\beta$ in \eqref{dec.g.beta.gamma}. An analoguous proof yields the bound on $\gamma$.\par

\medskip

\noindent\emph{Step 4: conclusion.}
Summing up estimates \eqref{dec.alpha-}, \eqref{dec.alpha+} and \eqref{dec.g.beta.gamma}, we get, in view of decomposition \eqref{decompo.v0} of $h$.
$$ \|h(t)\|_{\hdot}\leq \left\{ \begin{array}{ll} 
\ds Ce^{-\left(\frac{c_0+c_1}{2}\right)t} & \text{ if }c_0>e_0\\
\ds C\left[e^{-e_0t} + e^{-\left(\frac{c_0+c_1}{2}\right)t}\right]& \text{ if }c_0\leq e_0.
\end{array}\right.
$$

\noindent\emph{Proof of \eqref{decay.h}.}
Iterating the argument we obtain the bound $\|h(t)\|_{\hdot}\leq C_{\eta} e^{-(c_1-\eta) t}$ if $c_0<c_1<e_0$ or $e_0<c_0<c_1$, which yields (together with Claim \ref{claim.Strichartz.h}), the desired estimate \eqref{decay.h}

\noindent\emph{Proof of \eqref{decay.h'}.}
Let us assume $c_0\leq e_0 <c_1$. Then the equation on $\alpha_+$ in \eqref{eq.alpha+-} shows that $e^{e_0t}\alpha_+(t)$ has a limit $A_+$ when $t\rightarrow +\infty$. Integrating the equation between $t$ and $+\infty$, we get (in view of Claim \ref{Stri.phi})
$$A_+-e^{e_0 t}\alpha_+(t)=e^{e_0t} \int_t^{+\infty} \frac{\BB(\YYY_+,\eps(s))}{\BB(\YYY_+,\YYY_-)}ds= O\big(e^{(e_0-c_1)t}\big).$$
By decomposition \eqref{decompo.v0} and estimates \eqref{dec.alpha-} and \eqref{dec.g.beta.gamma}, we get 
$\|h(t)-A_+e^{-e_0t}\YYY_+\|_{\hdot}\leq Ce^{-\left(\frac{c_0+c_1}{2}\right)t}.$ 
Furthermore, $h_1(t):=h(t)-A_+e^{-e_0t}\YYY_+$ satisfies, as $h$, equation \eqref{eq.h}. Thus the estimate \eqref{decay.h} shown in the preceding step implies \eqref{decay.h'}. The proof of Proposition \ref{prop.crucial.h} is complete.
\end{proof}

\section{Proof of main results}
\label{sec.proof}
We now turn to the proof of Theorems \ref{th.exist} and \ref{th.classif}. In Subsection \ref{sub.contract}, we show the existence of the solutions $W^{\pm}$ of Theorem \ref{th.exist} by a fixed point, approaching them by approximates solutions $W_k^{\aexp}$ of \eqref{CP} constructed in Subsection \ref{sub.dev} and converging exponentially to $W$ for large $t$. Subsection \ref{sub.conclu} is devoted to the conclusion of the proofs of the theorems.

\subsection{A family of approximate solutions converging to $W$}
\label{sub.dev}
\begin{lemma}
\label{lem.approximate}
Let $a\in\RR$. There exist functions $(\Phi^{\aexp}_{j})_{j\geq 1}$ in $\SSS(\RR^N)$, such that $\Phi^{\aexp}_1=\aexp \YYY_+$ and if 
\begin{equation}
\label{defWk}
W^{\aexp}_k(t,x):=W(x)+\sum_{j=1}^k e^{-je_0t}\Phi^{\aexp}_j(x),
\end{equation}
then as $t\rightarrow +\infty$,
\begin{equation}
\label{eqWk}
i\partial_t W^{\aexp}_{k}+\Delta W_k^{\aexp}+\big|W_k^{\aexp}\big|^{p_c-1}W_k^{\aexp}=O(e^{-(k+1)e_0 t}) \text{ in }\SSS(\RR^N).
\end{equation}
\end{lemma}
\begin{remark}
Let 
$
\tilde{\eps}_k^{\aexp}:=i\partial_t W_k^{\aexp}+\Delta W_k^{\aexp}+|W^{\aexp}_{k}|^{p_c-1}W_k^{\aexp}.
$
By \eqref{eqWk} we mean that for all $J,M$, there exists $C_{J,M}>0$ such that $(1+|x|)^M |\partial_x^J \tilde{\eps}_k^{\aexp}(t,x)|\leq C_{J,M}e^{-(k+1)e_0 t}$.
\end{remark}
\begin{proof}[Proof of the lemma]
Let us fix $a\in \RR$.  To simplify notations, we will omit most of the superscripts $a$. We will construct the functions $\Phi_j=\Phi^{\aexp}_{j}$ by induction on $j$. Assume that $\Phi_1$, \ldots, $\Phi_k$ are known, and let $v_k:=W_k-W=\sum_{j=1}^k e^{-je_0t}\Phi_j(x)$. Assertion \eqref{eqWk} writes
\begin{equation}
\tag{\ref{eqWk}'}
\label{eqvk}
\eps_k:=\partial_t v_k+\LLL(v_k)+R(v_k)=O\big(e^{-(k+1)e_0 t}\big) \text{ in } \SSS(\RR^N).
\end{equation}

\medskip

\noindent\emph{Step 1: $k=1$.} Let $\Phi_1:=a \YYY_+$, which is in $\SSS$ (see Remark \ref{rem.S}) and $v_1(t,x):=e^{-e_0t}\Phi_1(x)$. We have $\partial_t v_1+\LLL v_1=0$ and thus 
$$\partial_t v_1+\LLL v_1+R(v_1)=R(v_1).$$
Note that $R(v_1)= W^{p_c} J(W^{-1}v_1)$, where $J(z):=-i\big[|1+z|^{p_c-1}(1+z)-1-\frac{p_c+1}{2}z-\frac{p_c-1}{2}\zbar\big]$ is real-analytic for $\{|z|<1\}$ and satisfies $J(0)=\partial_z J(0)=\partial_{\zbar} J(0)=0$. Write 
\begin{equation}
\label{dev.J}
J(z):=\sum_{j_1+j_2\geq 2} a_{j_1j_2} z^{j_1}\overline{z}^{j_2},
\end{equation}
with normal convergence of the series and all its derivatives, say for $|z|\leq \frac 12$. Chose $t_0$ such that $\forall t\geq t_0,\;\left|v_1(t)\right|\leq \frac{1}{2}W$. Then
\begin{equation}
\label{dev.v1}
\forall t\geq t_0,\; \forall x\in \RR^N,\quad R(v_1)=\sum_{j_1+j_2\geq 2} a_{j_1j_2}W^{p_c} \big(W^{-1}v_1\big)^{j_1}\big(W^{-1}\overline{v}_1\big)^{j_2}.
\end{equation}
As a consequence, there exists a constant $C>0$ such that for large $t$, $|R(v_1)|\leq C|W^{-1}v_1|^2$. Using analoguous inequalities on the derivatives of $R(v_1)$, and the fact that $v_1=a e^{-e_0t}\Phi_1$ with $\Phi_1\in \SSS(\RR^N)$, we get $R(v_1)=O(e^{-2e_0t})$ in $\SSS(\RR^N)$, which gives \eqref{eqvk} for $k=1$.

\medskip

\noindent\emph{Step 2: induction.} Let us assume that $\Phi_1$, \ldots, $\Phi_k$ are known and satisfy \eqref{eqvk} for some $k\geq 1$. To construct $\Phi_{k+1}$, we will first show that there exists $\Psi_k\in \SSS(\RR^N)$ (depending only on $\Phi_1$, \ldots, $\Phi_k$) such that for large $t$
\begin{equation}
\label{form.epsk}
\eps_k(x,t)=e^{-(k+1)e_0 t}\Psi_k(x)+O\big(e^{-(k+2)e_0t}\big) \text{ in } \SSS(\RR^N).
\end{equation}
Indeed by \eqref{eqvk}
\begin{equation}
\label{form.epsk1}
\eps_k(t,x)=\sum_{j=1}^k e^{-je_0t}\big( {-j}e_0\Phi_j(x)+\LLL\Phi_j(x)\big) +R(v_k(t,x)).
\end{equation}
All the functions $\Phi_j$ are in $\SSS(\RR^N)$, so that for large $t$, and all $x$, $\left|v_k(t,x)\right|\leq \frac{1}{2}W(x)$.
Furthermore  $R(v_k)= W^{p_c} J(W^{-1}v_k)$, and by the development \eqref{dev.J} of $J$ we get by \eqref{form.epsk1} that there exist functions $F_j\in\SSS(\RR^N)$ such that for large $t$
$$ \eps_k(t,x)=\sum_{j=1}^{k+1} e^{-je_0t} F_{j}(x)  + O\big(e^{-(k+2) t}\big) \text{ in } \SSS(\RR^N).$$
By \eqref{eqvk} at rank $k$, $F_j=0$ for $j\leq k$ which shows \eqref{form.epsk} with $\Phi_k=F_{k+1}$.\par
By Corollary \ref{cor.spectrum}, $(k+1) e_0$ is not in the spectrum of $\LLL$. Let 
$$\Phi_{k+1}:=-(\LLL-(k+1)e_0)^{-1} \Psi_k$$
which belongs to  $\SSS(\RR^N)$ (see Remark \ref{rem.S}) and is uniquely determined by $\Phi_1$,\ldots,$\Phi_k$. By definition, $v_{k+1}=v_k+e^{-(k+1)e_0t}\Phi_{k+1}$.
Furthermore, 
\begin{align*}
\eps_{k+1}&:=\partial_t v_{k+1}+\LLL v_{k+1}+R(v_{k+1})\\
&=\partial_t v_k+\LLL v_k+R(v_k)-(k+1)e_0\Phi_{k+1} e^{-(k+1)e_0t}+\LLL\Phi_{k+1}e^{-(k+1)e_0t}+R(v_{k+1})-R(v_k)\\
&=\eps_k-e^{-(k+1)e_0t} \Psi_{k}+R(v_{k+1})-R(v_k).
\end{align*}
By \eqref{form.epsk}, $\eps_k-e^{-(k+1)e_0t} \Psi_{k}=O\big(e^{-(k+2)e_0t}\big)$ in $\SSS(\RR^N)$. Writing as before $R=W^{p_c}J(W\cdot)$, and using the developpment \eqref{dev.J} of $J$, we get that $R(v_{k+1})-R(v_k)=O\big(e^{-(k+2)e_0t}\big)$ in $\SSS(\RR^N)$ which yields \eqref{eqvk} at rank $k+1$. The proof is complete.
\end{proof}

\subsection{Contraction argument near an approximate solution}
\label{sub.contract}
\begin{prop}
\label{prop.fxpt}
Let $a\in \RR$. There exists $k_0>0$ such that  for any $k\geq k_0$, there exists $t_k\geq 0$ and a solution $W^{\aexp}$ of \eqref{CP} such that for $t\geq t_k$,
\begin{equation}
\label{CondW_A}
\big\|\nabla\big(W^{\aexp}-W_{k}^{\aexp}\big)\big\|_{Z(t,+\infty)}\leq e^{-(k+\frac{1}{2})e_0t}.
\end{equation}
Furthermore, $W^{\aexp}$ is the unique solution of \eqref{CP} satisfying \eqref{CondW_A} for large $t$. Finally, $W^{\aexp}$ is independent of $k$ and satisfies for large $t$, 
\begin{equation}
\label{CondWa2}
\|W^{\aexp}(t)-W-a e^{-e_0t}\YYY_+\|_{\hdot}\leq e^{-\frac 32e_0t}.
\end{equation}
\end{prop}
\begin{proof}

\emph{Step 1: transformation into a fixed-point problem.}
As in the preceding proof, we will fix $a\in\RR$ and omit most of the superscripts $a$. 
Let
$$ h:=W^{\aexp}-W_k^{\aexp}.$$
The function $W^{\aexp}$ is solution of \eqref{CP} if and only if $w^{\aexp}:=W^{\aexp}-W$ is solution of \eqref{equation.v}. Substracting equations \eqref{equation.v} on $w^{\aexp}$ and \eqref{eqvk} on $v_k:=W_k^{\aexp}-W$, we get that $W^{\aexp}$ satisfies \eqref{CP} if and only if $h=w^{\aexp}-v_k$ satisfies $\partial_t h+\LLL h=-R(v_k+h)+R(v_k)+\eps_k$ (see \eqref{eqvk} for the definition of $\eps_k$). This may be rewritten (recalling \eqref{def.VVV} for the definition of $\VVV$)
$$i\partial_t h+\Delta h=-\VVV(h)-i R(v_k+h)+i R(v_k)+i\eps_k.$$ 
Thus the existence of a solution $W^{\aexp}$ of \eqref{CP} satisfying \eqref{CondW_A} for $t\geq t_k$ may be written as the following fixed-point problem
\begin{multline}
\label{defM}
\forall t\geq t_k,\quad h(t)=\MMM_k(h)(t)\text{ and } \|\nabla h\|_{Z(t,+\infty)}\leq e^{-\left(k+\frac{1}{2}\right)e_0 t} \\
\text{ where }\MMM_k(h)(t):=-\int_t^{+\infty}e^{i(t-s)\Delta} \big[i\VVV(h(s))-R(v_k(s)+h(s))+R(v_k(s))-\eps_k(s)\big]ds.
\end{multline}
Let us fix $k$ and $t_k$. Consider
\begin{align*}
E_Z^k&:=\left\{h\in S(t_k,+\infty),\;\nabla h\in Z(t_k,+\infty);\; \|h\|_{E_Z^k}:=\sup_{t\geq t_k} e^{\left(k+\frac{1}{2}\right)e_0 t}\|\nabla h\|_{Z(t,+\infty)}<\infty\right\}\\
B_Z^k&:=\big\{h\in E_Z^k,\; \|h\|_{E_Z^k}\leq 1\big\}.
\end{align*}
The space $E_Z^k$ is clearly a Banach space. In view of \eqref{defM}, it is sufficient to show that if $t_k$ and $k$ are large enough, the mapping $\MMM_k$ is a contraction on $B_Z^k$. This is the aim of the next step.
\medskip

\noindent\emph{Step 2: contraction property.} Note that by Strichartz inequality \eqref{StriFreeRM}, there is a constant $C^*>0$ such that if $g,h\in E_Z^k$, $k\geq 1$,
\begin{align}
\label{Strichartz.M1}
\|\nabla(\MMM_k(h))\|_{Z(t,+\infty)}&\leq 
C^*\Big[\|\nabla(\VVV(h))\|_{N(t,+\infty)}\\ \notag &\qquad+\|\nabla (R(v_k+h)-R(v_k))\|_{N(t,+\infty)}+\|\nabla \eps_k\|_{N(t,+\infty)}\Big]\\
\label{Strichartz.M2}
\|\nabla(\MMM_k(g)-\MMM_k(h))\|_{Z(t,+\infty)}&\leq  C^*\Big[\|\nabla(\VVV(g-h))\|_{N(t,+\infty)}\\ \notag &\qquad+\|\nabla (R(v_k+h)-R(v_k+g))\|_{N(t,+\infty)}\Big].
\end{align}
\begin{claim}
\label{claim.contract}
There exists $k_0>0$ such that for $k\geq k_0$ the following holds: for all $h\in E_Z^k$
\begin{equation}
\label{fxpt.VVV}
\|\nabla(\VVV(h))\|_{N(t,+\infty)} \leq \frac{1}{4C^*} e^{-(k+\frac 12)e_0t}\|h\|_{E_Z^k};
\end{equation}
and there exists a constant $C_{k}$, depending only on $k$ such that for all $f,g\in B_Z^k$ and $t\geq t_k$
\begin{gather}
\label{fxpt.R}
\big\|\nabla\big(R(v_k+g)-R(v_k+h)\big)\big\|_{N(t,+\infty)} \leq C_{k}e^{-(k+\frac 32)e_0t} \|g-h\|_{E_Z^k},\\
\label{fxpt.epsk}
\|\nabla \eps_k\|_{N(t,+\infty)}\leq C_{k} e^{-(k+1)e_0t}.
\end{gather}
\end{claim}
Let us first assume Claim \ref{claim.contract} and prove the proposition.
Chose $k\geq k_0$. By \eqref{Strichartz.M1}, \eqref{Strichartz.M2}, \eqref{fxpt.VVV}, \eqref{fxpt.R} and \eqref{fxpt.epsk}, we get, if $g,h\in B_Z^k$ 
\begin{gather*}
\|\MMM_k(h)\|_{E_Z^k}\leq \left(\frac{1}{4}+C^*C_{k}e^{-e_0t_k}+C^*C_{k} e^{-\frac 12 e_0t_k}\right),\\
\|\MMM_k(g)-\MMM_k(h)\|_{E_Z^k}\leq \|g-h\|_{E_Z^k}\left(\frac{1}{2}+C^*C_{k} e^{-e_0t_k}\right),
\end{gather*}
which shows, chosing a larger $t_k$ if necessary, that $\MMM_k$ is a contraction of $B_Z^k$.\par
Thus, for each $k\geq k_0$, \eqref{CP} has an unique solution $W^{\aexp}$ satisfying \eqref{CondW_A} for $t\geq t_k$. The preceding proof clearly remains valid taking a larger $t_k$, so that the uniqueness still holds in the class of solutions of \eqref{CP} satisfying \eqref{CondW_A} for $t\geq t'_k$, where $t'_k$ is any real number larger than $t_k$.
Let $k<\tilde{k}$ and $W^{\aexp}$, $\widetilde{W}^{\aexp}$ be the solutions of \eqref{CP} constructed above for $k$ and $\tilde{k}$ respectively. Then, $\widetilde{W}^{\aexp}$ although satisfies \eqref{CondW_A} for large $t$, so that the uniqueness in the fixed-point shows that $W^{\aexp}(t)=\widetilde{W}^{\aexp}(t)$, for large $t$ and thus, by uniqueness in \eqref{CP}, that $W^{\aexp}=\widetilde{W}^{\aexp}$.
This shows that $W^{\aexp}$ does not depend on $k$. 

It remains to show \eqref{CondWa2}. Let $k>0$ be a large integer and $h\in B_Z^k$. By Strichartz inequality \eqref{StriFreeDual}, and the definition of $\MMM_k$, we have, for $t\geq t_k$,
\begin{equation*}
\|\MMM_k(h)(t)\|_{\hdot}\leq \big\|\nabla\big(i\VVV(h)-R(v_k+h)+R(v_k)-\eps_k\big)\big\|_{N(t,+\infty)}.
\end{equation*}
As a consequence of Claim \ref{claim.contract} and the fact that $\|h\|_{E_Z}^k\leq 1$, we get
$$\|\MMM_k(h)(t)\|_{\hdot}\leq C\left(e^{-\left(k+\frac 12 \right)e_0t}\|h\|_{E_Z^k}+e^{-(k+1)e_0 t}\right)\leq Ce^{-\left(k+\frac 12 \right)e_0 t}.$$
Using the preceding inequality on $h=W^{\aexp}-W_k^{\aexp}$ (which satisfies $h=\MMM_k(h)$), and noting that $W_k^{\aexp}=W+a e^{-e_0t} \YYY_++O(e^{-2e_0t})$ in $\hdot$, we get directly \eqref{CondWa2}. To complete the proof of the proposition, it remains to show Claim \ref{claim.contract}.
\end{proof}

\begin{proof}[Proof of Claim \ref{claim.contract}]
Estimate \eqref{fxpt.epsk} follows immediately from \eqref{eqvk}.\par
Let us show \eqref{fxpt.R}. By Lemma \ref{lem.R}, 
\begin{gather}
\label{fxptR1}
\|\nabla(R(v_k+g)-R(v_k+h))\|_{N(t,t+1)}\leq  (A)\|\nabla(g-h)\|_{Z(t,t+1)},\\
\notag
\begin{aligned}
(A):=C \Big(\|\nabla g\|_{Z(t,t+1)}&+\|\nabla h\|_{Z(t,t+1)}+\|\nabla v_k\|_{Z(t,t+1)}\\
&+\|\nabla g\|^{p_c-1}_{Z(t,t+1)}
+\|\nabla h\|^{p_c-1}_{Z(t,t+1)}+\|\nabla v_k\|^{p_c-1}_{Z(t,t+1)}\Big).
\end{aligned}
\end{gather}
By the explicit form of $v_k$ and the fact that $g,h\in B_Z^k$, we get 
\begin{equation}
\label{fxptR2}
(A)\leq C_k' e^{-e_0t},
\end{equation}
where $C_k'$ only depends on $k$. Combining \eqref{fxptR1} and \eqref{fxptR2}, we get 
$$\|\nabla(R(v_k+g)-R(v_k+h))\|_{N(t,t+1)}\leq C_k'e^{-e_0t}\|\nabla(g-h)\|_{Z(t,t+1)}\leq C_k'' e^{-(k+\frac 32)e_0 t}\|g-h\|_{E_Z}$$
which gives \eqref{fxpt.R} in view of Claim \ref{summation}. \par
It remains to show \eqref{fxpt.VVV}. Let $\tau_0>0$. By Lemma \ref{lem.V}, there exists a constant $C_2>0$ such that 
$$\|\nabla(\VVV h)\|_{N(t,t+\tau_0)}\leq C_2\tau_0^{\frac{N}{N+2}}\|\nabla h\|_{Z(t,t+\tau_0)}\leq C_2\tau_0^{\frac{N}{N+2}}e^{-\left(k+\frac{1}{2}\right)e_0 t}\|h\|_{E_Z^k}$$
By Claim \ref{summation},
$$\|\nabla(\VVV h)\|_{N(t,+\infty)}\leq \frac{C_2 e^{-(k+\frac 12)e_0t}} {1-e^{-(k+\frac 12)e_0\tau_0}}\tau_0^{\frac{N}{N+2}} \|h\|_{E_Z^k}.$$
Chosing $\tau_0$ and $k_0$ such that $C_2 \tau_0^{\frac{N}{N+2}}=\frac{1}{8C_1}$ and $e^{-(k_0+\frac 12)e_0\tau_0}\leq \frac 12$, we get \eqref{fxpt.VVV} for $k\geq k_0$.

\end{proof}

\subsection{Conclusion of the proofs of the theorems}
\label{sub.conclu}

\begin{proof}[Proof of Theorem \ref{th.exist}.]
 Denote as before $\YYY_1:=\re \YYY_+=\re \YYY_-$. Note that 
$(W,\YYY_1)_{\hdot}\neq 0.$
Indeed, if $(W,\YYY_1)_{\hdot}=0$, then by the equation $\Delta W=-W^{p_c}$, we would have  $\BB(W,\YYY_+)=\BB(W,\YYY_-)=0$ so that $W\in G_{\bot}$, which contradicts, in view of Lemma \ref{lem.positivity}, the fact that $Q(W)=-\frac{2}{(N-2)C_N^N}<0$. Replacing $\YYY_{\pm}$ by $-\YYY_{\pm}$ if necessary, we may assume
\begin{equation}
\label{WY1>0}
(W,\YYY_1)_{\hdot}>0.
\end{equation}
Let 
$$W^{\pm}:=W^{\pm 1},$$ 
which yields two solutions of \eqref{CP} for large $t>0$. Then all the conditions of Theorem \ref{th.exist} are satisfied. Indeed \eqref{ex.energy} follows from the conservation of the energy and the fact that $W^{\aexp}$ tends to $W$ in $\hdot$, \eqref{ex.lim} is an immediate consequence \eqref{CondWa2}. Furthermore, again by \eqref{CondWa2},
$$ \|W^{\aexp}\|_{\hdot}^2=\|W\|^{2}_{\hdot}+2a e^{-e_0t}(W,\YYY_1)_{\hdot}+O(e^{-\frac{3}{2}e_0t}),$$
which shows, together with \eqref{WY1>0}, that for large $t>0$,
$$ \|W^+(t)\|_{\hdot}>0, \quad \|W^-(t)\|_{\hdot}<0.$$
From Remark \ref{RemPersist}, these inequalities remain valids for every $t$ in the intervals of existence of $W^+$ and $W^-$.
Finally $T_-(W^-)=-\infty$ by Proposition \ref{compactness.u} and $\|u\|_{S(-\infty,0)}<\infty$ by Corollary \ref{corol.incomp}.\par
It remains to show that in the case $N=5$, $T_-(W^+)<\infty$. For this we will show that for any $a$ and any $t$ in the interval of definition of $W^a$,
\begin{equation}
\label{N=5L2}
N=5\Longrightarrow W^a(t)\in L^{2}(\RR^5).
\end{equation}
Consider as in Subsection \ref{sub.compactness.u} a positive radial function $\psi$ on $\RR^5$, such that $\psi=1$ if $|x|\leq 1$ and $\psi=0$ if $|x|\geq 2$. Define, for $R>0$ and large $t$,
$$ F_R(t):=\int_{\RR^5} |W^a(t,x)|^2 \psi\big(\frac{x}{R}\big) dx.$$
Then, $W^a$ being a solution of \eqref{CP}, 
\begin{align*}
F'_R(t)&=\frac{2}{R}\im \int W^a\nabla\overline{W}^a\cdot (\nabla\psi)\big(\frac{x}{R}\big) dx
=\frac{2}{R}\im \int W\nabla(\overline{W}^a-W)\cdot(\nabla\psi)\big(\frac{x}{R}\big) dx\\
&+\frac{2}{R}\im \int (W^a-W)\nabla W\cdot(\nabla\psi)\big(\frac{x}{R}\big) dx+\frac{2}{R}\im \int (W^a-W)\nabla\left(\overline{W}^a-\overline{W}\right)\cdot(\nabla\psi)\big(\frac{x}{R}\big) dx.
\end{align*}
Using that by \eqref{CondWa2}, $\|W^a(t)-W\|_{\hdot}\leq Ce^{-e_0t}$, we get, by Hardy inequality $$|F'_R(t)|\leq C\|W^a(t)-W\|_{\hdot}\left(\|W^a(t)\|_{\hdot}+\|W\|_{\hdot}\right)\leq Ce^{-e_0t},$$ 
with a constant $C$ independent of $R$, and thus, integrating between a large $t$ and $+\infty$,
$$ \left|F_R(t)-\int_{\RR^N} |W(x)|^2 \psi\big(\frac{x}{R}\big) dx\right|\leq Ce^{-e_0t}.$$
Letting $R$ goes to $+\infty$, we get \eqref{N=5L2} and $\|W^a(t)\|_{L^2}=\|W\|_{L^2}$. In particular $W^+(t)\in L^2(\RR^5)$, and thus, by Corollary \ref{corol.sur}, $T_-(W^+)<\infty$ which concludes the proof of Theorem \ref{th.exist}.
\end{proof}

\begin{proof}[Proof of Theorem \ref{th.classif}]
Let us first prove:
\begin{lemma}
\label{lem.unic}
If $u$ is a solution of \eqref{CP} satisfying 
\begin{equation}
\label{devtu1'}
\|u(t)-W\|_{\hdot}\leq Ce^{-\gamma_0 t},\quad E(u)=E(W)
\end{equation}
then 
$$\exists! a\in \RR, \quad u=W^{\aexp}.$$
\end{lemma}
\begin{corol}
\label{corol.unic}
For any $a\neq 0$, there exists $T_a\in \RR$ such that
\begin{equation}
\label{W+-=Wa}\left\{
\begin{aligned}
 W^{\aexp}=W^+(t+T_a) &\text{ if } & a>0\\
W^{\aexp}=W^-(t+T_a) &\text{ if } & a<0.
\end{aligned}
\right.
\end{equation}
\end{corol}
\begin{proof}
Let $u=W+v$ be a solution of \eqref{CP} for $t\geq t_0$ satisfying \eqref{devtu1'}. Recall that $v$ satisfies equation \eqref{equation.v}.

\medskip

\noindent\emph{Step 1.} We show that there exists $a\in\RR$ such that
\begin{equation}
\label{condA1}
\forall \eta>0,\quad \|v(T)-a e^{-e_0 T}\YYY_+\|_{\hdot}+\big\|\nabla\big(v(t)-a e^{-e_0 t}\YYY_+\big)\big\|_{Z(T,+\infty)}\leq C_{\eta}e^{
-(2-\eta)e_0 T}.
\end{equation}
Indeed we will show
\begin{equation}
\label{decay.v}
\|v(t)\|_{\hdot} \leq Ce^{-e_0t}, \quad \|R(v(t))\|_{L^{\frac{2N}{N+2}}}+\|\nabla(R(v))\|_{N(t,+\infty)}\leq  C e^{-2e_0 t}.
\end{equation}
Assuming \eqref{decay.v}, we are in the setting of Proposition \ref{prop.crucial.h} with $h=v$, $\eps=-R(v)$, $c_0=e_0$ and $c_1=2e_0$. The conclusion \eqref{decay.h'} of the proposition would then yield \eqref{condA1}. It remains to prove \eqref{decay.v}.

By Lemma \ref{lem.R}, Claim \ref{claim.Strichartz.h} and Claim \ref{summation}, the bound on $R(v)$ in \eqref{decay.v} follows from the bound on $\|v(t)\|_{\hdot}$, so that we only need to show this first bound.

By Lemma \ref{lem.Strichartz.v}, assumption \eqref{devtu1'} implies $\|v(t)\|_{\hdot}+\|\nabla v\|_{Z(t,+\infty)}\leq C e^{-\gamma_0 t}.$ By Lemma \ref{lem.R} and Claim \ref{summation}
$$ \|R(v(t))\|_{L^{\frac{2N}{N+2}}}+\|\nabla(R(v))\|_{N(t,+\infty)}\leq  C e^{-2\gamma_0 t}.$$
Thus we can apply Proposition \ref{prop.crucial.h}, showing that
$$ \|v(t)\|_{\hdot}\leq C\big(e^{-e_0t}+ e^{-\frac{3}{2}\gamma_0t}\big).$$
If $\frac{3}{2}\gamma_0\geq e_0$ the proof of \eqref{decay.v} is complete. If not, assumption \eqref{devtu1'} on $v$ holds with $\frac{3}{2}\gamma_0$ instead of $\gamma_0$, and an iteration argument yields the result. The proof of \eqref{decay.v} is complete, which concludes Step 1.

\medskip
\noindent\emph{Step 2.} 
Let us show
\begin{equation}
\label{condA2}
\forall m>0,\; \exists t_0>0,\; \forall t\geq t_0,\quad \|u(t)-W^{\aexp}(t)\|_{\hdot}+\big\|\nabla\big(u-W^{\aexp}\big)\big\|_{Z(t,+\infty)}\leq e^{-m t}.
\end{equation}
This will show that $u=W^a$, by uniqueness in Proposition \ref{prop.fxpt}. According to Step $1$, \eqref{condA2} holds for $m=\frac{3}{2}e_0$. Let us assume \eqref{condA2} holds for some $m=m_1>e_0$. We will show that it holds for $m=m_1+\frac{e_0}{2}$, which will yield \eqref{condA2} by iteration and conclude the proof.\par
Write $v(t):=u(t)-W$, $w^{\aexp}(t):=W^{\aexp}(t)-W$ (so that in particular $u-W^{\aexp}=v-w^{\aexp}$). Then
$$ \partial_t (v-w^{\aexp})+\LLL(v-w^{\aexp})=-R(v)+R(w^{\aexp}).$$
We have assumed $\|v(t)-w^{\aexp}(t)\|_{\hdot}+\big\|\nabla\big(v-w^{\aexp}\big)\big\|_{Z(t,+\infty)}\leq e^{-m_1 t}.$
According to Lemma \ref{lem.R} and Claim \ref{summation}
$$ \big\|\nabla\big(R(v)-R(w^{\aexp})\big)\big\|_{N(t,+\infty)}+\big\|R(v(t))-R(w^{\aexp}(t))\big\|_{L^{\frac{2N}{N+2}}}\leq Ce^{-(m_1+e_0) t}.$$
Then by Proposition \ref{prop.crucial.h}
$$ \|v(t)-w^{\aexp}(t)\|_{\hdot}+\|\nabla(v-w^{\aexp})\|_{Z(t,+\infty)}\leq C e^{-\big(m_1+\frac{3}{4}e_0\big)t},$$
which yields \eqref{condA2} with $m=m_1+\frac{e_0}{2}$. By iteration, \eqref{condA2} holds for any $m>0$. Using this with $m=(k_0+1)e_0$ (where $k_0$ is given by Proposition \ref{prop.fxpt}), we get that for large $t>0$
$$ \big\|\nabla\big(u-W_{k_0}^{\aexp}\big)\big\|_{Z(t,+\infty)}\leq e^{-(k_0+\frac{1}{2})e_0t}.$$
By uniqueness in Proposition \ref{prop.fxpt}, we get as announced that $u=W^a$ which concludes the proof of the lemma.
\end{proof}
\begin{proof}[Proof of Corollary \ref{corol.unic}]
Let $a\neq 0$ and chose $T_a$ such that $|a|e^{-e_0T_a}=1$.
By \eqref{CondWa2}, 
\begin{equation}
\label{important.unic}
\|W^{\aexp}(t+T_a)-W\mp e^{-e_0t} \YYY_+\|_{\hdot}\leq Ce^{-\frac 32 e_0t}.
\end{equation}
 Furthermore, $W^{\aexp}(\cdot+T_a)$ satisfies the assumptions of Lemma \ref{lem.unic}, which shows that there exists $a'$ such that $W^{\aexp}(\cdot+T_a)=W^{\aexp'}$.
By \eqref{important.unic}, $a'=1$ if $a>0$ and $a'=-1$ if $a<0$, hence \eqref{W+-=Wa}. 
\end{proof}

Let us turn to the proof of Theorem \ref{th.classif}. Point \eqref{theo.crit} is an immediate consequence of the variational characterization of $W$ (\cite{Au76}, \cite{Ta76}). \par

Let us show \eqref{theo.sub}. Let $u$ be a solution of \eqref{CP} such that $E(u_0)=E(W)$ and $\|u_0\|_{\hdot}<\|W\|_{\hdot}$. Assume that $\|u\|_{S(\RR)}=\infty$. Replacing if necessary $u(t)$ by $\overline{u}({-t})$, we may assume that $\|u\|_{S(0,+\infty)}=\infty$. Then by Proposition \ref{prop.CV0}, there exist $\theta_0\in \RR$, $\mu_0>0$, and $c,C>0$ such that $\|u(t)-W_{[\theta_0,\mu_0]}\|_{\hdot}\leq Ce^{-ct}.$
This shows that $u_{[-\theta_0,\mu_0^{-1}]}$ fullfills the assumptions of Lemma \ref{lem.unic}. Using that $\|u\|_{\hdot}<\|W\|_{\hdot}$, this implies that there exists $a<0$ such that $u_{[-\theta_0,\mu_0^{-1}]}=W^{\aexp}$. Thus by Corollary \ref{corol.unic},
$$ u(t)=W^-_{[\theta_0,\mu_0]}(t+T_a),$$
which shows \eqref{theo.sub}.\par
The proof of \eqref{theo.super} is similar. Indeed if $u$ is a solution of \eqref{CP} defined on $[0,+\infty)$ and such that $E(u_0)=E(W)$, $\|u_0\|_{\hdot}>\|\nabla W\|_{\hdot}$ and $u_0\in L^2$, then by Proposition \ref{PropSuper},  $\|u(t)-W_{[\theta_0,\mu_0]}\|_{\hdot}\leq Ce^{-ct}$, which shows using Lemma \ref{lem.unic} and the same argument as before that for some $t_0\in \RR$,
$$ u(t)=W^+_{[\theta_0,\mu_0]}(t+t_0).$$
The proof of Theorem \ref{th.classif} is complete.
\end{proof}
\section{Appendix}
\subsection{Proofs of some results of decomposition near $W$}
\label{Appendix.Decompo}
\subsubsection{Proof of Lemma \ref{lem.ortho}}
Let us first show the lemma when $f$ is close to $W$.
Consider the following functionals on $\RR\times (0,+\infty) \times \hdot$:
$$ J_0: (\theta,\mu,f)\mapsto \left(f_{[\theta,\mu]},iW\right)_{\hdot},\quad J_1: (\theta,\mu,f)\mapsto \left(f_{[\theta,\mu]},\tW\right)_{\hdot}$$
Then, by \eqref{iW.W1}
\begin{align*}
\frac{\partial J_0}{\partial \theta}(0,1,W)&=\int |\nabla W|^2&\frac{\partial J_0}{\partial \mu}(0,1,W)&=0\\
\frac{\partial J_0}{\partial \mu}(0,1,W)&=0&\frac{\partial J_1}{\partial \mu}(0,1,W)&=-\int |\nabla \tW|^2.
\end{align*}
Furthermore, $J_0(0,1,W)=J_1(0,1,W)=0$.
Thus by the Implicit Function Theorem there exists $\eps_0,\eta_0>0$ such that for $h\in \hdot$:
$$ \|h-W\|_{\hdot}<\eps_0\Longrightarrow \exists! (\theta,\mu),\quad |\theta|+|\mu-1|\leq \eta_0\text{ and }\left(h_{[\theta,\mu]},iW\right)_{\hdot}= \left(h_{[\theta,\mu]},\tW\right)_{\hdot}=0.$$
Let $f$ be as in the proposition. By the variational characterization of $W$, if $\dd(f)$ is small enough, we can choose $\mu_1$ and $\theta_1$ such that $ f_{[\theta_1,\mu_1]}=W+g$, $\|g\|_{\hdot}\leq \eps\big(\dd(f)\big),$ and we are now reduced to the preceding case. The assertions on the uniqueness of $(\theta,\mu)$ and the regularity of the mapping $f\mapsto(\theta,\mu)$ follows from the Implicit Functions Theorem.
The proof of Lemma \ref{lem.ortho} is complete.
\qed

\subsubsection{Proof of Lemma \ref{bound.modul}}
Take $u$ as in Lemma \ref{bound.modul} and let 
\begin{equation}
\label{def.v}
v(t):=u_{[\theta(t),\mu(t)]}(t)-W=\tilde{u}(t)+\alpha(t) W.
\end{equation}

\medskip
\noindent\emph{Proof of \eqref{approximations}.} In this part of the proof, $t$ is just a parameter and we will not write it for the sake of simplicity. By \eqref{def.v},
\begin{equation}
\label{orthog}
\|v\|_{\hdot}^2=\alpha^2\|W\|_{\hdot}^2+\|\tilde{u}\|^2_{\hdot}.
\end{equation}
To get a second relation between $\|v\|_{\hdot}$, $\alpha$ and $\big\|\tilde{u}\big\|_{\hdot}$, we use the equation $E(W)=E(W+v)$ together with \eqref{W+g}. Denote by $\tilde{u}_1$ and $\tilde{u}_2$ the real and imaginary parts of $\tilde{u}$. By the orthogonality of $\tilde u_1$ and $\tilde u_2$ with $W$ in $\hdot$, and the equation $\Delta W+W^{p_c}=0$ we have
\begin{equation*}
 \int \nabla W \cdot\nabla \tilde u_1=\int \nabla W \cdot\nabla \tilde u_2=\int W^{p_c}\tilde{u}_1=\int W^{p_c}\tilde{u}_2 =0.
\end{equation*}
Thus $W$ and $\tilde{u}$ are $Q$-orthogonal and
$ Q(v)=Q(\tilde{u}+\alpha W)=-|Q(W)|\alpha^2+Q(\tilde{u}).$
This yields, using \eqref{W+g}, $\Big|\alpha^2|Q(W)|-Q(\tilde{u})\Big|\leq C\|v\|_{\hdot}^3.$
By the coercivity of $Q$ on $H^{\bot}$ (Claim \ref{coercivity}) which implies $Q(\tilde{u})\approx \big\|\tilde{u}\big\|_{\hdot}^2$, we get
\begin{equation}
\label{fgalpha}
\|\tilde{u}\|_{\hdot}^2\leq C \left(\|v\|^3_{\hdot}+\alpha^2\right),\quad \alpha^2\leq C\left( \|\tilde{u}\|_{\hdot}^2+\|v\|^3_{\hdot}\right).
\end{equation}
It follows from the variational characterization of $W$ that $\|v\|_{\hdot}$ is small when $\dd(u)$ is small. By \eqref{orthog} and \eqref{fgalpha}, we get, for small $\dd(u)$,
\begin{equation}
\label{adtf1}
|\alpha|\approx \|v\|_{\hdot}\approx \|\tilde{u}\|_{\hdot}.
\end{equation}
This is the first part of \eqref{approximations}. It remains to show the estimate on $\dd(u)$. Developing the equation $\|W+v\|_{\hdot}^2=\|W\|_{\hdot}^2+\dd(u)$ we get,
\begin{equation}
\label{claimNH}
 \|v\|_{\hdot}^2+2(v,W)_{\hdot}=\|v\|_{\hdot}^2+2\alpha=\dd(u)
\end{equation}
which gives, thanks to \eqref{adtf1}, the desired result. The proofs of \eqref{approximations} is complete.

\noindent\emph{Proof of \eqref{bound.derives}.}
Let us consider the self-similar variables $y$ and $s$ defined by
$$ \mu(t)y=x,\quad ds=\mu^2(t) dt.$$
Then \eqref{CP} may be rewritten
\begin{equation}
\label{CP'}
i\partial_s  u_{[\theta,\mu]}+\Delta_y u_{[\theta,\mu]}+\left|u_{[\theta,\mu]}\right|^{p_c-1} u_{[\theta,\mu]}+\theta_s u_{[\theta,\mu]}+i\frac{\mu_s}{\mu}\left(\frac{N-2}{2} u_{[\theta,\mu]}+y\cdot \nabla u_{[\theta,\mu]}\right)=0
\end{equation}
where the subscript $s$ denotes the derivative with respect to $s$ and $\Delta_y$ the Laplace operator with respect to the new space variable $y$.\par
We much show
\begin{equation}
\label{bound.derives'}
|\alpha_s(s)|+|\theta_s(s)|+\left|\frac{\mu_s}{\mu}(s)\right|\leq C|d(u(s))|.
\end{equation}
For any complex-valued function $f$, we will write $f_1:=\re f$, $f_2:=\im f$. Writing $u_{[\theta,\mu]}=W+v$, we get
$$ \partial_s v+\LLL v+R(v)-\theta_s iW-\theta_s iv+\frac{\mu_s}{\mu} W_1+\frac{\mu_s}{\mu}\left(\frac{N-2}{2}v+y\cdot \nabla v\right)=0.$$
Where the linear operator $\LLL$ and the remainder term $R$ are defined by \eqref{equation.v}.
We will need the following bound on $R(v)$ (see Lemma \ref{lem.R})
\begin{equation}
\label{boundRmodul}
\|R(v)\|_{L^{\frac{2N}{N+2}}}\leq C\left(\|v\|^2_{\hdot}+\|v\|_{\hdot}^{p_c}\right).
\end{equation}
Writing $v=\tilde{u}+\alpha(s)W$ and keeping in the left-hand side only the terms that are linear in $\tilde{u}$, $\alpha$, $\alpha_s$, $\theta$, $\theta_s$ and $\mu_s/\mu$, we get
\begin{multline}
\label{eq.tildeu}
\partial_s \tilde{u}_1+i\partial_s \tilde{u}_2+\alpha_s W+(\Delta+W^{p_c-1})\tilde{u}_2-i(\Delta+p_cW^{p_c-1})\tilde{u}_1-i\alpha(p_c-1) W^{p_c}\\
-\theta_s iW+\frac{\mu_s}{\mu} W_1=-R(v)+\theta_s i v-\frac{\mu_s}{\mu}\left(\frac{N-2}{2}v+y\cdot \nabla v\right).
\end{multline}
In view of estimates \eqref{approximations}, it is easy to see that the $\hdot$-scalar products of the right-hand term by $W$, $iW$ and $W_1$ are bounded up to a constant by $\eps(s)$, where $\eps(s)$ is defined by
$$\eps(s):=|\dd|\left(|\dd|+|\theta_s(s)|+\Big|\frac{\mu_s}{\mu}(s)\Big|\right),\quad\dd:=\dd(u(s)).$$
For instance, by \eqref{boundRmodul}
\begin{equation}
\label{formalIPP}
\left|\big(R(v),W\big)_{\hdot}\right|=\left|\big(R(v),\Delta W\big)_{L^2}\right|\leq \|R(v)\|_{L^{\frac{2N}{N+2}}}\|\Delta W\|_{L^{2^*}}=O\big(\dd^2\big).
\end{equation}
The formal integration by part in \eqref{formalIPP}, which is rigorous for smooth solutions of \eqref{CP} decaying fast enough at infinity, may be justified by passing to the limit and using the standard Cauchy problem theory for \eqref{CP}.
Projecting equation \eqref{eq.tildeu} in $\hdot$ on $W$, $iW$ and $\tW$, we get \big(denoting by $c:=\NH{W}^2$, $c_1:=\NH{\tW}^2$\big)
\begin{gather}
\label{PiW}
c\alpha_s=-(\Delta \tilde{u}_2,W)_{\hdot}-( W^{p_c-1}\tilde{u}_2,W)_{\hdot}+O(\eps(s))\\
\label{PW}
c\theta_s=-(\Delta \tilde{u}_1,W)_{\hdot}-p_c(W^{p_c-1}\tilde{u}_1,W)_{\hdot}-\alpha(p_c-1)(W^{p_c},W)_{\hdot}+O(\eps(s))\\
\label{PiW1}
\frac{\mu_s}{\mu} c_1=-(\Delta \tilde{u}_2,W_1)_{\hdot}-(W^{p_c-1}\tilde{u}_2 ,W_1)_{\hdot}+O(\eps(s)).
\end{gather}
Justifying as before the integrations by parts, we have
$$\big(\Delta \tilde{u}_1,W)_{\hdot}=(\tilde{u}_1,\Delta W\big)_{\hdot},\quad \big(\Delta \tilde{u}_2,W)_{\hdot}=(\tilde{u}_2,\Delta W\big)_{\hdot},\quad \big(\Delta \tilde{u}_2,W_1)_{\hdot}=(\tilde{u}_2,\Delta W_1\big)_{\hdot}.$$
Consequently all the right-hand terms in equations \eqref{PiW}, \eqref{PW} and \eqref{PiW1} are bounded up to a constant by $\|\tilde{u}\|_{\hdot}+\eps$. By \eqref{approximations}, $\|\tilde{u}\|_{\hdot}\leq C\dd$ which yields \eqref{bound.derives'} and completes the proof of Lemma \ref{bound.modul}.
\qed

\subsection{Spectral properties of the linearized operator}
\label{app.spectral}
This part of the appendix is dedicated to the proof of Lemma \ref{lem.Y}, which is a variation of the classical proof (see \cite{Gr90} and the survey \cite{Sc06} for similar results).

\subsubsection{Proof of the existence of the eigenfunctions}

Note that $\overline{\LLL(v)}=-\LLL(\overline{v})$, so that if $e_0>0$ is an eigenvalue for $\LLL$ with eigenfunction $\YYY_+$, $-e_0$ is an eigenvalue of $\LLL$ with eigenfunction $\overline{\YYY}_+$. Let us show the existence of $\YYY_+$. Writing $\YYY_1=\re \YYY_+$, $\YYY_2=\im \YYY_+$, we must solve
\begin{equation}
\label{eq.vp}
\left\{
\begin{aligned}
(\Delta+p_cW^{p_c-1})\YYY_1=&-e_0 \YYY_2\\
(\Delta+W^{p_c-1}) \YYY_2=&e_0 \YYY_1.
\end{aligned}\right.
\end{equation}
Let $V:=W^{p_c-1}$.
The operator $-\Delta-V$ on $L^2$ with domain $H^2$ is self-adjoint and nonnegative, thus it has a unique square root $(-\Delta-V)^{\frac 12}$ with domain $H^1$ (see \cite{We80}). 
Assume that there exist $f_1\in H^4$ such that 
\begin{equation}
\label{eqP}
Pf_1=-e_0^2 f_1,\text{ where }P:=(-\Delta-V)^{\frac 12}(-\Delta-p_cV)(-\Delta-V)^{\frac 12}.
\end{equation}
Then taking
$$ \YYY_1:= (-\Delta-V)^{\frac 12}f_1,\quad \YYY_2:=\frac{1}{e_0} (-\Delta-p_cV)(-\Delta-V)^{\frac 12}f_1,$$
would yield a solution of system \eqref{eq.vp}, showing the existence of $\YYY_+$ and $\YYY_-$.\par
The remainder of the proof is devoted to proving that the operator $P$ on $L^2$ with domain $H^4$ has a strictly negative eigenvalue. Note that 
$$
P=(\Delta+V)^2-(p_c-1)(-\Delta-V)^{\frac 12}V(-\Delta-V)^{\frac 12}$$
is a relatively compact, selfadjoint, perturbation of $\Delta^2$, so that its essential spectrum is $[0,+\infty)$ (see \cite{We80}). Thus we only need to show the following claim.
\begin{claim}
\label{sigma<0}
\begin{equation*}
\sigma_-(P):=\inf\big\{(Pf,f)_{L^2}, \;\,f \in D(P),\;\|f\|_{L^2}=1\big\}<0.
\end{equation*}
\end{claim}
\begin{proof}
Note that $(Pf,f)_{L^2}=-\big((\Delta+p_cV)F,F\big)_{L^2}$, where $F:= (-\Delta-V)^{\frac 12} f$. Thus it is sufficient to find $F$ such that
\begin{equation}
\label{FL2>0}
\big((\Delta+p_cV)F,F\big)_{L^2}>0,\text{ and }\exists g\in H^4, \; F=(\Delta+V)g.
\end{equation}
We distinguish two cases. First assume that $N=3,4$, so that $W\notin L^2$.
Let $W_a(x):=\chi\big( x/a \big) W(x)$, where $\chi$ is a smooth, radial function such that
$\chi(r)=1$ for $r\leq 1$ and $\chi(r)=0$ for $r\geq 2.$
We first claim
\begin{equation}
\label{agrand}
\exists a>0,\quad E_a:=\int (\Delta+p_cV) W_a W_a >0.
\end{equation}
Recall that $\Delta W=-W^{p_c}$. Thus
\begin{equation*}
(\Delta+p_cV)W_a=(p_c-1)\chi\big( x/a \big) W^{p_c} +\frac 2a (\nabla \chi) \big( x/a \big) \cdot \nabla W+\frac{1}{a^2} (\Delta \chi)\big(x/a \big) W.
\end{equation*}
Hence
\begin{equation*}
\int (\Delta +p_cV)W_a W_a=\int\chi_a^2(p_c-1) W^{p_c+1}+\underbrace{\frac{2}{a} \int(\nabla \chi) \big( x/a \big) \cdot \nabla W\, W}_{(A)}+\underbrace{\frac{1}{a^2}\int (\Delta \chi)\big(x/a \big) W^2}_{(B)}.
\end{equation*}
According to the explicit expression \eqref{defW} of $W$, $W\leq C |x|^{-(N-2)}$ and $|\nabla W|\leq C |x|^{-(N-1)}$ at infinity, which gives 
$|(A)|+|(B)|\leq \frac Ca$ if $N=3$, $|(A)|+|(B)|\leq \frac {C}{a^2}$ if $N=4$. Hence \eqref{agrand}.

Let us fix $a$ such that \eqref{agrand} holds.
Recall that $W$ is not in $L^2$. Thus $\Delta+V$ is a selfadjoint operator on $L^2$, with domain $H^2$, and without eigenfunction. In particular the orthogonal of its range $R(\Delta+V)$ is $\{0\}$, and thus $R(\Delta+V)$ is dense in $L^2$. Let $\eps>0$, and consider
$G_{\eps}\in H^2$ such that
\begin{equation*}
\|(\Delta+V)G_{\eps}-(\Delta+V-1)W_a\|_{L^2}\leq \eps.
\end{equation*}
Taking $F_{\eps}:=(\Delta+V-1)^{-1}(\Delta+V)G_{\eps}$, we obtain  $\|(\Delta+V-1)(F_{\eps}-W_a)\|_{L^2}\leq \eps$ which implies 
$\|F_{\eps}-W_a\|_{H^2}\leq \eps\|(\Delta+V-1)^{-1}\|_{L^2\rightarrow L^2}.$
Hence for some constant $C_0$,
$$ \left| \int_{\RR^N} (\Delta+p_cV) F_{\eps} F_{\eps}-\int_{\RR^N} (\Delta+p_cV)W_a W_a \right|\leq C_0\eps.$$
As a consequence of \eqref{agrand}, we get \eqref{FL2>0} for $F=F_{\eps}$, $\eps=\frac{E_a}{2 C_0}$, which shows the claim in the case $N=3,4$.\par

Assume now that $N=5$, so that $W$ is in $L^2$ and more generally in all spaces $H^s(\RR^N)$. In this case
$(R(\Delta+V))^{\bot}=N(\Delta+V)=\vect\{W\}$, and thus
\begin{equation}
\label{imagedense}
\overline{R(\Delta+V)}=\big\{f\in L^2,\; (f,W)_{L^2}=0\big\}.
\end{equation}
Furthermore, $\Delta+p_cV$ is a self-adjoint compact perturbation of $\Delta$ and $((\Delta+p_c V)W,W)_{L^2}>0$, which shows that $\Delta+p_cV$ has a positive eigenvalue. Let $Z$ be the eigenfunction for this eigenvalue. Recalling that $(\Delta+p_cV)W_1=0$ we get, for any real number $\alpha$
$$ \int_{\RR^N} (\Delta+p_cV)(Z+\alpha W_1)\,(Z+\alpha W_1)=\int_{\RR^N} (\Delta+p_cV)ZZ>0.$$
By explicit calculation, $(W_1,W)_{L^2}\neq 0$, so that we can chose the real number $\alpha$ to have $(Z+\alpha W_1,W)_{L^2}=0$. Hence
$$ \big((\Delta+V-1)(Z+\alpha W_1),W\big)_{L^2}=\big(Z+\alpha W_1,(\Delta+V-1)W\big)_{L^2}=-(Z+\alpha \tW,W)_{L^2}=0.$$
By \eqref{imagedense}, we can chose, for any $\eps>0$ a function $G_{\eps}$ in $H^2$ such that
$$ \|(\Delta+V)G_{\eps}-(\Delta+V-1)(Z+\alpha W_1)\|_{L^2}<\eps.$$
As in the preceding case, $F_{\eps}=(\Delta+V-1)^{-1}(\Delta+V)G_{\eps}$ satisfies \eqref{FL2>0} for small $\eps>0$. Claim \ref{sigma<0} is shown for $N=3,4,5$, which concludes the proof of the existence of the real eigenvalues of $e_0$ and $-e_0$.
\end{proof}

\subsubsection{Decay at infinity of the eigenfunctions}

To conclude the proof of Lemma \ref{lem.Y}, it remains to show that $\YYY_{\pm}\in\SSS(\RR^N)$. 
By a simple boot-strap argument, it is easy to see that the eigenfunctions $\YYY_+$ and $\YYY_-$ are $C^{\infty}$. It remains to show the decay at infinity of $\YYY_+$, $\YYY_-$ and all their derivatives.

Recall that the eigenfunctions $\YYY_+$ and $\YYY_-$ are complex conjugates. According to system \eqref{eq.vp} on $\YYY_1=\re \YYY_+$ and $\YYY_2=\im \YYY_+$, it suffices to show the decay result on $\YYY_1$ only. Furthermore, by Sobolev embeddings, we only have to show that the following property holds for all $k$ and $s$ 
\begin{equation}
\label{PNs}
\tag{$\PPP_{k,s}$} \forall \varphi\in C^{\infty}_0\left( \RR^N\backslash \{0\}\right), \; \exists C,\; \forall R\geq 1,\; \|\varphi(x/R)\YYY_1 \|_{H^s}\leq \frac{C}{(1+R)^k}.
\end{equation}
Recall that $\YYY_1=\sqrt{-\Delta-V} f_1$, with $f_1\in H^4$, so that $(\PPP_{0,3})$ is satisfied. We will show that for $k\geq 0$, $s\geq 3$, $(\PPP_{k,s})$ implies $(\PPP_{k+1,s+1})$. Assume $(\PPP_{k,s})$ and consider $\varphi$ and $\tilde{\varphi}$ in $C^{\infty}_0\left(\RR^N\backslash \{0\}\right)$ such that $\tilde{\varphi}$ is $1$ on the support of $\varphi$. Note that by \eqref{eq.vp}
\begin{equation}
\label{eq.v1}
(\Delta^2+e_0^2)\YYY_1=-V \Delta \YYY_1-\Delta(p_c V\YYY_1)-p_cV^2 \YYY_1.
\end{equation}
By the explicit form of $W$, $V$ and all its derivatives decay at least as $\frac{1}{|x|^4}$ at infinity. Thus \eqref{eq.v1} implies $\|\varphi(x/R) (\Delta^2+e_0^2)\YYY_1\|_{H^{s-3}}\leq \frac{C}{R^4}\|\tphi(x/R)\YYY_1\|_{H^s}$. Hence
\begin{equation}
\label{ineg.v1}
\|(\Delta^2+e_0^2)(\varphi(x/R)\YYY_1)\|_{H^{s-3}}\leq \frac{C}{R}\|\tphi(x/R)\YYY_1\|_{H^s}.
\end{equation}
By ($\PPP_{k,s}$), the right-hand side of  \eqref{ineg.v1} is bounded by $\frac{C}{R^{k+1}}$ for large $R$. Furthermore, $\Delta^2+e_0^2$ is an isomorphism from $H^{s+1}$ to $H^{s-3}$, so that \eqref{ineg.v1} implies $\|\varphi(x/R)\YYY_1\|_{H^{s+1}}\leq \frac{C}{R^{k+1}}$,
which yields exactly $(\PPP_{k+1,s+1})$. The proof is complete.
\qed
\begin{remark}
\label{rem.S}
Let $\Psi\in \SSS(\RR^N)$ and $e_1\in \RR\setminus \{-e_0,0,e_0\}$ (thus by Corollary \ref{cor.spectrum}, $e_1$ is not in the spectrum of $\LLL$). Then by a proof similar to the one above
\begin{equation}
\label{inS}
\Phi:=(\LLL-e_1)^{-1} \Psi\in\SSS(\RR^N).
\end{equation}
Indeed $\Phi_1=\re \Phi$ and $\Phi_2=\im \Phi$ satisfy the equations
\begin{equation}
\label{eq.Phi}
-e_1\Phi_1+(\Delta+V)\Phi_2=\Psi_1,\quad
-e_1\Phi_2-(\Delta+p_cV)\Phi_1=\Psi_2.
\end{equation}
As $\Phi_1$ and $\Phi_2$ are, by definition, in $L^2$, a simple bootstrap argument shows that they are in all $H^s$, $s\geq 0$. Furthermore
$$ (\Delta^2+e_1^2)\Phi_1=-V\Delta \Phi_1-\Delta(p_cV\Phi_1)-p_cV^2\Phi_1-e_1\Psi_1-(\Delta+V)\Psi_2,$$
which gives equation \eqref{eq.v1}, up to a right-member term $-e_1\Psi_1-(\Delta+V)\Psi_2$ which is in $\SSS(\RR^N)$. Thus the iteration argument above shows that $\Phi_1\in \SSS(\RR^N)$, which implies by \eqref{eq.Phi} that $\Phi_2\in\SSS(\RR^N)$. Hence \eqref{inS}.
\end{remark}
\subsection{Proof of Lemma \ref{lem.R}}
We have
$$ R(f)=-i|W+f|^{p_c-1}(W+f)+iW^{p_c}+i\frac{p_c+1}{2}W^{p_c-1}f+i\frac{p_c-1}{2}W^{p_c-1}\fbar=W^{p_c} J\left(W^{-1}f\right)$$
where $J$ is the function defined on $\CC$ by
$$J(z)=-i|1+z|^{p_c-1}(1+z)+i+i\frac{p_c+1}{2}z+i\frac{p_c-1}{2}\zbar.$$
Recall that $p_c>2$. Thus $J$ is of class $C^2$ on $\CC$ and $J(0)=\partial_z J(0)=\partial_{\zbar} J(0)=0$. Furthermore, for large $|z|$, $J$ is bounded by $C|z|^{p_c}$, and its derivatives of order $k=1,2$ by $C|z|^{p_c-k}$. Hence
\begin{align}
\label{estimJ2}
|J(z)-J(z')|&\leq C|z-z'|\big(|z|+|z'|+|z|^{p_c-1}+|z'|^{p_c-1}\big)\\
\label{estimdJ2}
|\partial_z J(z)-\partial_z J(z')|+|\partial_{\zbar} J(z)-\partial_{\zbar} J(z')|&\leq C|z-z'|\left(1+|z|^{p_c-2}+|z'|^{p_c-2}\right).
\end{align}
By \eqref{estimJ2} we get the pointwise bound
\begin{equation}
\label{major.dR}
|R(f)-R(g)|\leq C|f-g|\left(W^{p_c-2}|f|+W^{p_c-2}|g|+|f|^{p_c-1}+|g|^{p_c-1}\right),
\end{equation}
which yields \eqref{estim.dR1} using H\"older inequality $\|abc^{p_c-2}\|_{L^{\frac{2N}{N+2}}}\leq\|a\|_{L^{2^*}}\|b\|_{L^{2^*}}\|c\|_{L^{2^*}}^{p_c-2}$.\par
Now, remark that 
\begin{multline*}
\nabla(R(f))=p_c W^{p_c-1}(\nabla W)J\left(W^{-1} f\right)\\+W^{p_c}\nabla(W^{-1}f)(\partial_zJ)\left(W^{-1}f\right)+W^{p_c}\nabla(W^{-1}f)(\partial_{\zbar} J)\left(W^{-1}f\right).
\end{multline*}
By \eqref{estimJ2} and \eqref{estimdJ2} we get
\begin{gather*}
|\nabla R(f)-\nabla R(g)|\leq C\Big\{(A)+(B)+(C)\Big\}\\ (A):=\frac{1}{|x|+1}|f-g|\Big(W^{p_c-2}|f|+W^{p_c-2}|g|+|f|^{p_c-1}+|g|^{p_c-1}\Big)\\
(B):=\left|W\nabla\big(W^{-1}f-W^{-1}g\big)\right|\Big(W^{p_c-2}|f|+|f|^{p_c-1}\Big)\\ (C):=\left|W\nabla\big(W^{p_c-2}+W^{-1}g\big)\right|\,|f-g|\Big(W^{p_c-2}+|f|^{p_c-2}+|g|^{p_c-2}\Big).
\end{gather*}
Note that $\frac{2N(N+2)}{N^2+4}<N$ for $N=3,4,5$ so that if $u\in S(I)$ and $\nabla u\in Z(I)$, Hardy inequality $\|\frac{1}{|x|} u\|_{Z(I)}\leq \|\nabla u\|_{Z(I)}$ holds. Using H\"older inequality $\left\|abc^{p_c-2}\right\|_{N(I)}\leq \|a\|_{Z(I)}\|b\|_{S(I)}\|c\|_{S(I)}^{p_c-2}$ together with Hardy and Sobolev inequalities we get
\begin{align*}
\|(A)\|_{N(I)} &\leq C\Big\|\frac{1}{|x|+1}(f-g)\Big\|_{Z(I)}\left[\|W\|_{S(I)}^{p_c-2}\left(\|f\|_{S(I)}+\|g\|_{S(I)}\right)+\|f\|^{p_c-1}_{S(I)}+ \|g\|_{S(I)}^{p_c-1}\right]\\
&\leq C \|\nabla(f-g)\|_{Z(I)}\left[|I|^{\frac{6-N}{2(N+2)}}\left(\|\nabla f\|_{Z(I)}+\|\nabla g\|_{Z(I)}\right)+\|\nabla f\|_{Z(I)}^{p_c-1}+\|\nabla g\|_{Z(I)}^{p_c-1}\right].
\end{align*}
The other terms $(B)$ and $(C)$ are handled in the same way. Note in particular that by Hardy inequality $\|W\nabla\big(W^{-1}g\big)\|_{Z(I)}\leq C\|\nabla g\|_{Z(I)}$. The proof of \eqref{estim.dR2} is complete.
\qed

\nocite{KrSc05}\nocite{CoKeStTaTa06}
\bibliographystyle{alpha} 
\bibliography{toto}
\end{document}